\documentclass[a4paper,11pt] {article}
\usepackage{geometry}
\usepackage{mdwlist}
\usepackage[latin1]{inputenc}
\usepackage[T1]{fontenc}
\usepackage[latin1]{inputenc}
\usepackage{lmodern}
\usepackage[english]{babel}
\usepackage{fancybox}
\usepackage{empheq}
\usepackage{xcolor}
\usepackage{bm}
\usepackage{tikz}
\usepackage{cite}

\newcommand{\deltat}{{k}}

\usepackage{amsthm}
\newtheorem{df}{Definition}[section]

\newtheorem{lm}{Lemma}[section]
\newtheorem{cor}{Corollary}[section]
\usepackage{amsmath}
\usepackage{amssymb}
\usepackage{graphicx}
\usepackage{wrapfig}
\usepackage{titlesec}
\newcommand{\mesh}{\mathcal{T}}
\newcommand{\R}{\mathbb{R}}

\newcommand{\sth}{\sum_{K \in {\cal{T}} }}
\newcommand{\sti}{\sum_{\sigma \in {\cal{E}}_{\intt}}}
\newcommand{\stik}{\sum_{\substack{\sigma \in {\cal{E}}(K) \\ \sigma=K|L}}}

\newcommand{\stie}{\sum_{i, \sigma \in {\cal{E}} }}
\newcommand{\stiek}{\sum_{\sigma \in {\cal{E}}(K) }}
\newcommand{\stkl}{\sum_{\substack{\sigma\in{\cal E}_{\intt} \\ \sigma=K|L}} }

\newcommand{\bu}{\boldsymbol{u}}
\newcommand{\bn}{\boldsymbol{n}}

\newcommand{\bv}{\boldsymbol{v}}
\newcommand{\bzero}{\boldsymbol{0}}
\newcommand{\bU}{\boldsymbol{U}}

\newcommand{\bW}{\boldsymbol{W}}

\newcommand{\dS}{\, {\mathrm d}S}
\newcommand{\dx}{\, {\mathrm d}x}

\catcode`\@=11 \@addtoreset{equation}{section}

\catcode`\@=12

\newtheorem{Theorem}{Theorem}[section]
\newtheorem{Proposition}{Proposition}[section]
\newtheorem{Lemma}{Lemma}[section]
\newtheorem{Corollary}{Corollary}[section]
\newtheorem{Remark}{Remark}[section]

\newcommand{\bTheorem}[1]{
\begin{Theorem} \label{T#1} }
\newcommand{\eT}{\end{Theorem}}

\newcommand{\bProposition}[1]{
\begin{Proposition} \label{P#1}}
\newcommand{\eP}{\end{Proposition}}

\newcommand{\bLemma}[1]{
\begin{Lemma} \label{L#1} }
\newcommand{\eL}{\end{Lemma}}

\newcommand{\bCorollary}[1]{
\begin{Corollary} \label{C#1} }
\newcommand{\eC}{\end{Corollary}}

\newcommand{\bFormula}[1]{
\begin{equation} \label{#1}}
\newcommand{\eF}{\end{equation}}

\newcommand{\bRemark}[1]{
\begin{Remark} \label{R#1} }
\newcommand{\eR}{\end{Remark}}

\newcommand{\Ov}[1]{\overline{#1}}

\newcommand{\vr}{\varrho}

\newcommand{\vc}[1]{{\bf #1}}

\newcommand{\Grad}{\nabla_x}

\newcommand{\Rm}{\mbox{\FF R}}

\newcommand{\intO}[1]{\int_{\Omega} #1 \ \dx}

\font\FF=msbm10 scaled 800
\definecolor{grey}{rgb}{0.85,0.85,0.85}
\date{}

\long\def\greybox#1{%
    \newbox\contentbox%
    \newbox\bkgdbox%
    \setbox\contentbox\hbox to \hsize{%
        \vtop{
            \kern\columnsep
            \hbox to \hsize{%
                \kern\columnsep%
                \advance\hsize by -2\columnsep%
                \setlength{\textwidth}{\hsize}%
                \vbox{
                    \parskip=\baselineskip
                    \parindent=0bp
                    #1
                }%
                \kern\columnsep%
            }%
            \kern\columnsep%
        }%
    }%
    \setbox\bkgdbox\vbox{
        \color{grey}
        \hrule width  \wd\contentbox %
               height \ht\contentbox %
               depth  \dp\contentbox
        \color{black}
    }%
    \wd\bkgdbox=0bp%
    \vbox{\hbox to \hsize{\box\bkgdbox\box\contentbox}}%
    \vskip\baselineskip%
}

\DeclareMathOperator{\dv}{div}

\DeclareMathOperator{\dt}{dt}

\DeclareMathOperator{\intt}{int}
\DeclareMathOperator{\extt}{ext}
\DeclareMathOperator{\T}{{\cal{T}}}

\DeclareMathOperator{\E}{{\cal{E}}}

\numberwithin{equation}{section}
\title{Error estimates for a numerical approximation to the compressible barotropic Navier-Stokes equations}
\author{Thierry Gallou\"et \and Rapha\`ele Herbin \and David Maltese \and Antonin Novotny\thanks{This work was supported by the MODTERCOM project within the APEX programme of the Provence-Alpes-C\^ote d'Azur region}}

\begin{document}

\maketitle

\centerline{Aix-Marseille Universit\'e, CNRS,   Centrale  Marseille,  I2M, UMR 7373, 13453 Marseille, France}
\bigskip
\centerline{and}
\bigskip
\centerline{IMATH, EA 2134, Universit\' e de Toulon, BP 20132, 83957 La Garde, France}

\begin{abstract}
We present here a general method based on the investigation of the relative energy of the system, that provides an unconditional error estimate for the approximate solution of the barotropic Navier Stokes equations obtained by time and space discretization.
We use this methodology to derive an error estimate for a specific  DG/finite element scheme for which the convergence was proved in \cite{KARPER}. { This is an extended version of the paper submitted to IMAJNA.}
\end{abstract}
\titleformat\section{}{}{0pt}{\Large\scshape\bfseries\filcenter}
\tableofcontents
\titleformat\section{}{}{0pt}{\Large\scshape\bfseries\filcenter\thesection{} - }

\section{Introduction}\label{intro}

The aim of this paper is to derive an error estimate for approximate solutions of the compressible barotropic Navier-Stokes equations obtained by a discretisation scheme.
These equations are posed on the time-space domain $Q_T=(0,T)\times\Omega$, where $\Omega$ is a bounded polyhedral domain of $\R^d$, $d=2,3$ and  $T>0$, and read:
\begin{subequations}
	 \begin{align}
		& \partial_t \vr + \dv ( \vr \bu) = 0,  \label{cont2} \\
		& \partial_t (\vr\bu) +{\rm div}(\vr\bu\otimes\bu) - \mu \Delta \bu -(\mu+\lambda) \nabla \dv \bu+ \nabla_x p(\vr) = \bm{0}, \label{mov2}
	\end{align}
	\label{pbcont}
\end{subequations}
supplemented with the initial conditions
\begin{equation}\label{ci1}
	 \vr(0,x) = \vr_0 (x),~\vr\bu(0,x) =\vr_0\bu_0,
\end{equation}
where $\varrho_0$ and $\bu_0$ are given functions from $\Omega$ to $\R_+$ and $\R^d$ respectively, and boundary conditions
\begin{equation}\label{ci2}
	\bu_{|(0,T)\times\partial \Omega} =0.
\end{equation}
In the above equations, the unknown functions are the scalar density field $\vr(t,x)\ge 0$ and vector velocity field $\bu=(u_1,\ldots,u_d)(t,x)$, where $t \in (0,T)$ denotes the time and $x\in \Omega$ is the space variable.
The viscosity coefficients $\mu$ and $\lambda$ are such that
\begin{equation}\label{visc}
	\mu > 0,~ \lambda + \frac{2}{d}\mu \ge 0.
\end{equation}
The pressure $p$ is a given by an equation of state, that is a function of density which satisfies
\begin{equation}\label{hypp}
	p \in C([0,\infty)) \cap C^1(0,\infty),\;p(0)=0,\; p'(\vr)>0.
\end{equation}
In addition to (\ref{hypp}), in the error analysis, we shall need to prescribe the asymptotic behavior of the pressure at large densities
\begin{equation}\label{pressure1}
	\lim_{\vr\to \infty}\frac {p'(\vr)}{\vr^{\gamma-1}}=p_\infty>0\quad\mbox{with some } \gamma\ge 1;
\end{equation}
furthermore, if $\gamma <2$ in \eqref{pressure1}, we need the additional condition (for small densities):
\begin{equation}\label{pressure2}
\liminf_{\vr\to 0}\frac {p'(\vr)}{\vr^{\alpha+1}}=p_0>0\quad\mbox{with some }\alpha\le 0.
\end{equation}

The main underlying idea of this paper is to derive the error estimates for approximate solutions of problem \eqref{pbcont}{--(\ref{ci2})} obtained by time and space discretization by using the discrete version of the {\em relative energy method} { introduced on the continuous level in \cite{FeJiNo,FENOSU, FENO7}. In spite of the fact that the relative energy method looks at the first glance pretty much  similar to the widely used {\em relative entropy method} (and both approaches translate the same thermodynamic stability conditions), they are very different in appearance and formulation and may provide different results.}
The notions of relative entropy and relative entropy inequality were first introduced by Dafermos \cite{Daf4} in the context of systems of conservation laws and in particular for the compressible Euler equations.
The relative energy functional was suggested and successfully used for the investigation of the stability of weak solutions to the equations of viscous compressible and heat conducting fluids in \cite{FENO7}.
In contrast with the relative entropy of Dafermos, { for the viscous and heat conducting fluids}, the relative energy approach is able to provide the structural stability of weak solutions, while the relative entropy approach fails in this case. 
\\ \\
Both functionals coincide  { (modulo a change of variables)} in the case of (viscous) compressible flows in the barotropic regime. The  relative energy functional and the intrinsic version of the relative energy inequality have been recently employed to obtain several stability results for the weak solutions to these equations, including the weak strong uniqueness principle, see \cite{FeJiNo,FENOSU}.
Note that particular versions of the relative entropy inequality with particular specific test functions  had been previously derived   in the context of low Mach number limits, see e.g. \cite{MASS,WanJia}.
\\ \\
The { discrete} version of the Dafermos relative entropy was employed in the non viscous case to derive an error estimate for the numerical approximation to a hyperbolic system of conservation laws and, in particular, to the compressible Euler equations \cite{CancesMathisSeguin2014Relative}.
In this latter paper, the authors assume an $L^\infty$   bound for the discrete solution, which is uniform with respect to the size of the space and time disretization (usually called stability hypothesis), that is not provided by the discrete equations. { The same method with the same severe hypotheses have been used in
\cite{YOVAN} to treat the compressible Navier-Stokes equations}.
The error analysis in the present paper relies on the theoretical background introduced in \cite{FeJiNo} and yields an {\it unconditional result}; in particular, {\it we do not need any assumed bound} on the solution to get the error estimate.
\\ \\
The mathematical analysis of numerical schemes for the discretization of the  steady and/or non steady compressible Navier-Stokes and/or compressible Stokes equations has been the object of some recent works.
The convergence of the discrete solutions to the weak solutions of the compressible stationary Stokes was shown for a finite volume-- non conforming P1 finite element \cite{GHL2009iso,EGHL2010isen,eym-10-convII} and for the wellknown MAC scheme which was introduced in \cite{har-65-num} and is widely used in computational fluid dynamics {\ (see e.g.  \cite{Sun-2014})}.
The unsteady Stokes problem was also discretized by some other discretization schemes on a reformulation of the problem, which were  proven to be convergent \cite{KK2010Stokes,KK2011Stokes,KK2012Stokes}.
{ The unsteady barotropic Navier-Stokes equations was recently { investigated} in \cite{KARPER}  in the case $\gamma>3$ (there is a real difficulty in the realistic case $\gamma \le 3$ arising from the treatment of the non linear convective term).}
However, in these works, the rate of convergence is not provided; in fact, to the best of our knowledge, no error analysis has yet been performed for any of the numerical schemes that have been designed for the compressible Navier-Stokes equations, in spite of its great importance for the numerical analysis of the equations and for the mathematical simulations of compressible fluid flows.
We  present here a general technique to obtain an error analysis and apply it to one of the available numerical schemes.
To the best of our knowledge, this is the first result of this type in the mathematical literature on the subject.
\\ \\
To achieve the goal, we systematically use the relative energy method on the discrete level.
From this point of view, this paper is as valuable for the introduced methodology as for the result itself.
Here, we apply the method to the scheme of  \cite{KARPER}.
{ In spite of the fact that this latter scheme is not used in practice (see e.g. \cite{KHL2013baro} for a related schemes used in industrial codes), we begin the error analysis with the scheme \cite{KARPER}} because of its readily available convergence proof.
In fact, we aim to use this approach to investigate the numerical errors of  less academic numerical schemes, such as the finite volume -- non conforming P1 finite element \cite{GHL2010drift,KHL2013baro,GGHL2008baro,GHKLL2011allspeed} or the MAC scheme \cite{bab-11-dis,her-14-inc}.
\\ \\
The paper is organized as follows.
After recalling the fundamental setting of the problem and the relative energy inequality in the continuous case in Section \ref{1}, we proceed in Section \ref{2} to the discretization: we introduce the discrete functional spaces and the definition of the numerical scheme, { and state  the main result of the paper, that is the error estimate formulated} in Theorem \ref{Main}.
The remaining sections are devoted to the proof of Theorem \ref{Main}:
\begin{list}{$\bullet$}{}
\item
 In Section \ref{4} we recall the existence theorem  for the numerical scheme (Lemma \ref{Theorem1}) and derive estimates provided by the scheme.  
\item In Section \ref{5}, we derive the discrete intrinsic version of the relative energy inequality for the solutions of the numerical scheme (see Theorem \ref{Theorem4}).
\item The relative energy inequality is transformed to a more convenient form in Section \ref{6}, see Lemma \ref{refrelenergy}.
\item Finally, in Section \ref{7}, we investigate the form of the discrete relative energy inequality with the test function { being a strong} solution to the original problem. This investigation is formulated in Lemma \ref{strongentropy} and finally leads to a Gronwall type estimate formulated in Lemma \ref{Gronwall}. The latter yields the error estimates and finishes the proof of the main result.
 \end{list}
 Fundamental properties of the discrete functional spaces
needed throughout the paper are  reported in Appendix (Section \ref{3}). Some of them (especially those referring to the $L^p$ setting, $p\neq 2$ that are not currently available in the mathematical literature) are proved. Section \ref{3} is therefore of the independent interest.

\section{The continuous problem}\label{1}
The aim of this section is to recall some fundamental notions and results.
We begin by the definition of weak solutions to problem \eqref{pbcont}-- \eqref{ci2}.

\begin{df}[Weak solutions]\label{ws}
{ Let $\varrho_0  : \Omega \to [0, +\infty) $ and $\bu_0 :\Omega \to \R^d$  with  finite energy $E_0=\int_\Omega (\frac{1}{2} \varrho_0 |\bu_0|^2 + {{H}}(\vr_0)) \dx$ and finite mass $0<M_0=\int_\Omega\vr_0\dx$.}
We shall say that the pair $ (\vr,\bu) $ is a weak solution to the problem   \eqref{pbcont}--\eqref{ci2}  emanating from the initial data $ (\vr_0,\bu_0)$ if:
\begin{description}
\item{(a)}  $ \vr \in L^{\infty}(0,T;L^1(\Omega)),\; \vr \ge 0 ~a.e. ~\text{in}~ (0,T)\time\Omega,$  and $ \bu \in L^2(0,T; W_0^{1,2}(\Omega)).$

\item{(b)}  $\vr\in C_{\rm weak}([0,T]; L^1(\Omega))$, and the continuity equation  $(\ref{cont2})$ is satisfied in the following weak sense
\begin{equation}\label{contf}
	\int_{\Omega} \vr \varphi \dx\Big|_0^\tau  = \int_{0}^{\tau}\int_\Omega \Big(\vr  \partial_t \varphi +  \vr \bu \cdot \nabla_x \varphi\Big) \dx\dt, \; \forall \tau \in [0,T], \,  \forall \varphi \in C^{\infty}_c ( [0,T]\times \overline{\Omega}).
\end{equation}
\item{(c)}   $\vr\bu\in C_{\rm weak}([0,T]; L^1(\Omega))$, and the momentum equation $(\ref{mov2}) $ is satisfied in the weak sense,\begin{multline}\label{movf}
	\int_{\Omega} \vr\bu \cdot \varphi \dx\Big|_0^\tau = \int_0^\tau \int_\Omega \Big(\vr\bu \cdot \partial_t \varphi+ \vr\bu\otimes\bu:\nabla\varphi
	+ p(\vr) \dv\varphi\Big)\dx\dt \\
	- \int_0^\tau \int_\Omega\Big(\mu\nabla \bu : \nabla_x \varphi \dx\dt +(\mu +\lambda){\rm div}\bu{\rm div}\varphi\Big)\dx\dt, \; \forall \tau \in [0,T], \, \forall  \varphi \in C^{\infty}_c ( [0,T] \times \Omega;\R^3).
\end{multline}
\item {(d)} The following energy inequality is satisfied
\begin{equation}\label{ienergief}
	\int_\Omega \Big(\frac{1}{2} \vr|\bu|^2 + {{H}}(\vr)\Big) \dx\Big|_0^\tau + \int_0^\tau \int_\Omega\Big( \mu| \nabla \bu |^2 +(\mu+\lambda)|\dv\bu|^2\Big) \dx\dt\le 0, \mbox{ for a.a. } \tau\in (0,T),
\end{equation}
\begin{equation}\label{H}
	{\ \mbox{ with } \; H(\vr)=\vr\int_1^\vr\frac{p(z)}{z^2}{\rm d}z.}
\end{equation}
\end{description}
Here and hereafter the symbol $\displaystyle \int_{\Omega} g \dx\,|_0^\tau$ is meant for $\displaystyle  \int_\Omega g(\tau,x)\dx - \int_\Omega g_0(x)\dx$.
\end{df}

In the above definition, we tacitly assume that all  the integrals in the formulas \eqref{contf}--\eqref{ienergief} are defined and we recall that $C_{\rm weak}([0,T]; L^1(\Omega))$  is the space of functions of $L^\infty([0,T]; L^1(\Omega))$ which are continuous for the weak topology.

{\ We notice that the function $\vr\mapsto H(\vr)$ is a solution of the ordinary differential equation
$\vr H'(\vr)-H(\vr)=p(\vr)$
with the constant of integration fixed such that $H(1)=0$.}

Note that the existence of weak solutions emanating from the finite energy initial data is well-known on bounded Lipschitz domains under assumptions (\ref{hypp}) and (\ref{pressure1}) provided $\gamma>d/(d-1)$, see Lions \cite{LI4} for "large" values of $\gamma$, Feireisl and coauthors \cite{FNP} for $\gamma>d/(d-1)$.

\medskip

Let us now introduce the notion of relative energy.
We first introduce the function
\begin{equation}\label{E}
	\begin{array}{lll}
		&E: &[0,\infty)\times(0,\infty)\to \R,\\
		& & (\vr,r) \mapsto E(\vr|r)=H(\vr)-H'(r)(\vr-r) -H(r),
	\end{array}
\end{equation}
where $H$ is defined by \eqref{H}.
Due to the monotonicity hypothesis in (\ref{hypp}), $H$  is strictly convex on $[0,\infty)$, and therefore
\[
	E(\vr|r)\ge 0\quad\mbox{and}\quad E(\vr|r)=0\;\Leftrightarrow\;\vr=r.
\]
In order to measure a ``distance'' between a weak solution $(\vr,\bu)$ of the compressible Navier-Stokes system and any other state $(r,\bU)$ of the fluid , we introduce the relative energy functional, defined by
\begin{equation}\label{ent}
	{\cal{E}}(\vr,\bu\Big|r,\bU) = \int_\Omega \Big(\frac{1}{2} \vr| \bu -\bU|^2 + E(\vr\,|\,r)\Big) \dx.
\end{equation}
It was proved recently in \cite{FeJiNo} that, provided assumption \eqref{hypp} holds, any weak solution satisfies the following so-called relative energy inequality
\begin{equation}\label{p5}
\begin{aligned}
	&{\mathcal E} \left( \vr, \bu \Big| r, \bU \right) (\tau) -{\mathcal E} \left( \vr, \bu \Big| r, \bU \right)(0)+ \int_0^\tau \intO{ \Big( \mu|\nabla(\bu-\bU)|^2 +(\mu+\lambda)|\dv(\bu-\bU)|^2 \Big) } \dt \\
	&\phantom{{\mathcal E} \left( \vr, \bu  \right)} \le \int_0^\tau \intO{ \Big( \mu \nabla\bU:\nabla(\bU-\bu) +(\mu+\lambda)\dv \bU \dv(\bU-\bu) \Big) } \dt \\
	 & \phantom{{\mathcal E} \left( \vr, \bu  \right) ) \le } +   \int_0^\tau \intO{  \vr \partial_t \bU\cdot (\bU - \bu )}\dt + \int_0^\tau\intO{\vr\bu {\cdot} \nabla \bU \cdot(\bU - \bu )}\dt \\
	 & \phantom{{\mathcal E} \left( \vr, \bu   \right) \le } -\int_0^\tau \intO{p(\vr)\dv \bU} \dt + \int_0^\tau \intO{ (r - \vr) \partial_t H'(r)}\dt -\int_0^\tau \intO{\vr \nabla H'(r) \cdot\bu }\dt
\end{aligned}
\end{equation}
for a.a. $\tau\in (0,T)$, and for any pair of test functions
\[
r \in C^1 ([0,T] \times \Ov{\Omega}),\ r > 0,\ \bU \in C^1([0,T] \times \Ov{\Omega};\R^3 ),\ \bU|_{\partial \Omega} = 0.
\]


The stability of strong solutions in the class of weak solutions is stated in the following proposition.
\begin{Proposition}[Estimate on the relative energy]\label{proposition}
 Let $\Omega$ be a Lipschitz domain. Assume that the viscosity coefficients satisfy assumptions (\ref{visc}), that the  pressure $p$ is a twice continuously differentiable function on $(0,\infty)$ satisfying (\ref{hypp}) and (\ref{pressure1}), and that $(\vr,\bu)$ is a weak solution to problem \eqref{pbcont}--\eqref{ci2} emanating from initial data $(\vr_0\ge 0,\bu_0)$, with finite energy  $E_0$  and finite mass $M_0{\ =\int_\Omega\vr_0{\rm d} x}>0$.
Let $(r,\bU)$ in the class
\begin{equation} \label{r,U}
	\left\{\begin{array}{l}
		 \displaystyle r\in C^1([0,T]\times\overline\Omega),\; 0<\underline r =\min_{(t,x)\in\overline Q_T}r(t,x)\le r(t,x)\le \overline r= \max_{(t,x)\in\overline Q_T}r(t,x), \\ ~ \\
 		 \bU\in C^1([0,T]\times \overline\Omega;\R^3),\; \bU|_{(0,T)\times\partial\Omega}=0
	\end{array}\right.
 \end{equation}
be a (strong) solution of problem  \eqref{pbcont}  emanating from the  initial data $(r_0,\bU_0)$.
Then there exists
\[
c=c(T,\Omega, M_0, E_0, \underline r, \overline r,  |p'|_{C^1([\underline r,\overline r])}, \|(\nabla r, \partial_t r, \bU, \nabla\bU, \partial_t \bU) \|_{L^\infty(Q_T;\Rm^{19})} )>0
\]
such that for almost all $t\in (0,T)$,
 \begin{equation}\label{cerrorestimate}
  {\cal E}(\vr,\bu\Big|r,\bU)(t)\le c{\cal E}(\vr_0,\bu_0\Big|r_0,\bU_0).
\end{equation}
\eP
This estimate (implying among others the weak-strong uniqueness) was proved in \cite{FeJiNo} (see also \cite{FENOSU}) for pressure laws (\ref{pressure1}) with $\gamma>d/(d-1)$.
It remains valid under weaker hypothesis on the pressure, such as (\ref{pressure1}) with $\gamma\ge 1$; this can be proved using ideas introduced in \cite{BEFENO} and \cite{MANO}.

\section{The numerical scheme}\label{2}
\subsection{Partition of the domain}\label{3.1}
We suppose that $\Omega$ is a bounded domain of $\R^d$, polygonal if $d=2$ and polyhedral if $d=3$.
Let ${\cal{T}}$ be a decomposition of the domain $\Omega$ in { tetrahedra}, which we call hereafter a triangulation of $\Omega$, regardless of the space dimension. By ${\cal{E}}(K)$, we denote the set of the edges ($d=2$) or faces ($d=3$) $\sigma$ of the element $K \in {\cal{T}}$ called hereafter faces, regardless of the dimension.
The set of all faces of the mesh is denoted by ${\cal{E}}$; the set of faces included in the boundary $\partial\Omega$ of $\Omega$ is denoted by ${\cal{E}}_{\extt}$ and the set of internal faces (i.e ${\cal{E}} \setminus {\cal{E}}_{\extt} $) is denoted by ${\cal{E}}_{\intt}$.
The triangulation ${\cal{T}}$ is assumed to be regular in the usual sense of the finite element literature (see e.g. \cite{cia-91-bas}), and in  particular, ${\cal{T}}$ satisfies the following properties:
\begin{itemize}
	\item[$\bullet$] $ \overline{\Omega} = \cup_{K \in {\cal{T}}} \overline{K} $;
	\item[$\bullet$]  if $(K,L) \in {\cal T}^2$, then $ \overline{K} \cap \overline{L} = \emptyset $ or $\overline{K} \cap \overline{L} $ is a vertex or $\overline{K} \cap \overline{L} $ is a common face of $K$ and $L$; in the latter case it is denoted by $K|L$.
\end{itemize}
For each internal face of the mesh $\sigma=K|L$, $\bn_{\sigma, K}$ stands for the normal vector of $\sigma$, oriented from $K$ to $L$ (so that $\bn_{\sigma, K}=-\bn_{\sigma,L}$).
We denote by $ |K|$ and $ |\sigma|$ the ($d$ and $d-1$ dimensional) Lebesgue measure of the { tetrahedron} $K$ and of the face $\sigma$  respectively, and  by $h_K$ and $h_\sigma$ the diameter of $K$ and $\sigma$ respectively.
We measure the regularity of the mesh thanks to  the parameter $\theta$ defined by
{
\begin{equation}\label{reg}
	{\ \theta = \inf \{ \frac{\xi_K}{h_K}, K \in {\cal{T}} \} }
\end{equation}
}
where $\xi_K$ stands for the diameter of the largest ball included in $K$.
Last but not least  we denote by $h$ the maximal size of the mesh,
\begin{equation}\label{maxh}
 {\	h=\max_{K \in {\cal{T}}} h_K.}
 \end{equation}
%
The triangulation $\mathcal T$ is said to be regular if it satisfies
\begin{equation}\label{reg1-}
 	{\ \theta\ge\theta_0>0.}
\end{equation}

\subsection{Discrete {\ function} spaces}
Let $\mathcal T$ be a mesh of $\Omega$.
We denote by $L_h({ \Omega})$ the space of piecewise constant functions on the cells of the mesh;
the space $L_h{ (\Omega)}$ is the approximation space for the pressure and density.
For $1\le p<\infty$, the mapping
\[
	q \mapsto \|q\|_{L_h^p(\Omega)}= \|q\|_{L^p(\Omega)}=\Big(\sum_{K\in{\cal T}}|K| |q_K|^p\Big)^{1/p}
\]
is a norm on $L_h(\Omega)$.
We also introduce spaces of non-negative and positive functions:
\[
  L_h^+{ (\Omega)} = \{q \in L_h{ (\Omega)}, ~q_K \ge 0, ~\forall K \in \T \},\quad L_h^{++}{ (\Omega)} = \{q \in L_h{ (\Omega)}, ~q_K > 0, ~\forall K \in \T \}.
\]
The approximation space for the velocity field  is the space $\bW_h(\Omega)=V_h(\Omega;{\R^d})$, where $V_h(\Omega)$ is the  non conforming piecewise linear finite element space \cite{cro-73-con,ern-04-the} defined by:.
\begin{multline}
	 V_h(\Omega) = \{ v \in L^2(\Omega), ~\forall K \in {\cal{T}}, ~v_{|K} \in \mathbb{P}_1(K),\\
	 \forall \sigma \in {\cal{E}}_{\intt} ,\; \sigma=K|L,\; \int_{\sigma}v_{|K} \dS=\int_{\sigma}v_{|L}\dS,\quad
	\forall \sigma \in {\cal{E}}_{\extt},\; \int_\sigma v \dS=0 \},
 \end{multline}
where $\mathbb{P}_1(K)$ denotes the space of affine functions on $K$ and $\dS$  the integration with respect to the $(d-1)$-dimensional Lebesgue measure { on the face $\sigma$}.
Each element $v\in V_h(\Omega)$ can be written in the form
\begin{equation}\label{CRP}
	v{ (x)}=\sum_{\sigma\in{\cal E}_{\rm int}}v_\sigma\varphi_\sigma{ (x)},\quad  x\in \Omega
\end{equation}
where the set $\{\varphi_\sigma\}_ {\sigma\in{\cal E}_{\rm int}}\subset V_h(\Omega)$ is the classical basis determined by
\begin{equation}\label{CRB}
\forall (\sigma,\sigma')\in {\cal E}^2_{\rm int},\;\int_{\sigma'}\varphi_\sigma \dS=\delta_{\sigma,\sigma'},\quad\forall\sigma'\in {\cal E}_{\rm ext},\; \int_{\sigma'}\varphi_\sigma \dS=0
\end{equation}
{ and $\{v_\sigma\}_{\sigma\in {\cal E}_{\rm int}}\subset R$ is the set of degrees of freedom relative to $v$. }
Notice that $V_h(\Omega)$ approximates the functions with zero traces  in the sense that for all elements in $V_h(\Omega)$, $v_\sigma=0$ provided $\sigma\in {\cal E}_{\rm ext}$.
Since only the continuity of the integral over each face of the mesh is imposed, the functions in $V_h(\Omega)$ may be discontinuous through each face; the discretization is thus nonconforming in $ W^{1,p}(\Omega;\R^d)$, $1\le p\le \infty$.
Finally, we notice that for any $1\le p<\infty$ the expression
\[
	|v|_{V_h^p(\Omega)}=\Big(\sum_{K\in {\cal T}}\|\nabla v\|_{L^p(K;\Rm^d)}^p\Big)^{1/p}
\]
is a norm on $V_h(\Omega)$ and we denote by $V^p_h{ (\Omega)}$ the space $V_h{ (\Omega)}$ endowed with this norm.

We finish this section by introducing some notations.
For a function $v$ in $L^1(\Omega)$, we set
\begin{equation}
 	v_K = \frac 1{|K|}\int_K v \dx \textrm{ for } K \in {\cal T} \textrm{  and }   \hat v{ (x)}= \sum_{K\in {\cal T}} v_K 1_K{ (x)},\; { x\in \Omega}
	\label{vhat}
\end{equation}
so that $\hat v \in L_h(\Omega)$. { Here and in what follows, $1_K$ is the characteristic function of $K$}.

{ If}  $v \in W^{1,p}(\Omega)$, we set
\begin{equation}
	v_\sigma=\frac 1{|\sigma|}\int_\sigma v{\rm d} S\textrm{ for } \sigma\in {\cal E}.
\label{vtilde}
\end{equation}
Finally, if $v \in W^{1,p}_0(\Omega)$, we set
\begin{equation}
   v_h{ (x)}=\sum_{\sigma\in {\cal E_{\rm int}}} v_\sigma \varphi_\sigma{ (x)},\; { x\in \Omega.}
	\label{vh}
\end{equation}
so that $v_h \in V_h(\Omega)$.
{ In accordance with the above notation, for $v\in W^{1,p}_0(\Omega)$, the symbol $\hat v_{h}$ means $\widehat{v_h}(x)=\sum_{\sigma\in {\cal E}_{\rm int}}v_\sigma\hat\phi_\sigma(x)$, and the symbol $v_{h,K}=\frac 1{|K|}\int_K v_h(x){\rm d}x$ and the symbol $\hat v_{h,\sigma}^{\rm up}= [\widehat{(v_h)}]_\sigma^{\rm up}$.}


\subsection{Discrete equations}
Let us consider a partition $0=t_0<t_1<...<t_N =T$ of the time interval $[0,T]$, which, for the sake of simplicity, we suppose uniform. Let $\deltat$ be the constant time step $ \deltat =t_{n}-t_{n-1} $ for $n=1,...,N$. The density
field $\vr(t_n,x)$ and the velocity field ${\bu}(t_n,x)$ will be approximated by the quantities
\begin{equation} \label{notation0}
	\vr^n(x)=\sum_{K\in {\cal T}}\vr_K^n 1_K(x),\quad \bu^n(x)=\sum_{\sigma\in {\cal E}} \bu^n_\sigma \varphi_\sigma(x),
\end{equation}
 where the approximate densities $(\vr^n_K)_{K \in \mesh, n=1,\ldots,N}$ and velocities $(\bu^n_\sigma)_{\sigma \in {\cal{E}}_{\rm int}, n=1,\ldots,N}$ are the discrete unknowns (with { $\vr^n_K \in \R^+$} and $\bu^n_\sigma \in \R^d$).

{ For the future convenience, we denote here and hereafter,
\begin{equation}\label{notation1}
\vr(t,x)=\sum_{n=1}^{N}\vr^n(x)1_{[n-1,n)}(t),\quad\bu(t,x)=\sum_{n=1}^{N}\bu^n(x)1_{[n-1,n)}(t)
\end{equation}
and recall that the usual Lebesgue norms of these functions read
\begin{equation}\label{notation2}
\|\vr\|_{L^\infty(0,T;L^p(\Omega)}\equiv\max_{n=1,\ldots,N}\|\vr^n\|_{L^p(\Omega)}, \quad \|\bu\|_{L^p(0,T;L^q(\Omega;\Rm^3)}\equiv \deltat\Big(\sum_{n=1}^N\|\bu^n\|^p_{L^q(\Omega;\Rm^3)}\Big)^{1/p}
\end{equation}
}

{ Starting from this point, unlike in Section \ref{intro}, here and hereafter, the couple $(\vr,\vc u)$ respectively $(\vr^n,\vc u^n)$ introduced in (\ref{notation0}--\ref{notation2}) denote always exclusively a {\it discrete numerical solution}. }

The numerical scheme consists in writing the equations that are solved to determine these discrete unknowns.
In order to ensure the positivity of the approximate densities, we shall use an upwinding technique for the density in the mass equation.
For  $q\in L_h(\Omega)$ and  $\bu\in \vc W_h (\Omega)$,
the upwinding of  $q$ with respect to  $\bu$ is defined, for $\sigma =K|L\in {\cal E}_{\rm{int}}$ by
\begin{equation}\label{upwind1}
q^{{\rm up}}_\sigma=\begin{cases}
                      q_K \;\mbox{if }\bu_\sigma\cdot\bn_{\sigma, K}> 0\\ q_L\; \mbox{if } \bu_\sigma\cdot\bn_{\sigma, K}{ \le} 0,
                    \end{cases}
\end{equation}
so that
\[
\sum_{\sigma\in{\cal E}(K)} q_\sigma^{\rm up}{\bu}_\sigma\cdot{\vc n}_{\sigma,K}=\stik \Big(q_K [{\bu}_\sigma\cdot{\vc n}_{\sigma,K}]^+ - q_L [{\bu_\sigma}\cdot{\vc n}_{\sigma,K}]^-\Big),
\]
where $a^+ = \max(a,0)$, $a^- =- \min(a,0)$.


Let us then  consider the following numerical scheme \cite{KARPER}:

\vspace{2mm}

{\it Given  $(\vr^0,\bu^0) \in L_h^+(\Omega) \times \bW_h(\Omega) $
{find}} $ (\vr^n)_{1\le n \le N}\subset (L_h(\Omega))^N, (\bu^n)_{1\le n \le N} \subset (\bW_h(\Omega))^N $ {\it such that for all} $ n=1,...,N $
\vspace{2mm}
\begin{subequations}\label{scheme}
\begin{align}
 \label{dcont}
	&|K| \frac{ \vr_{K}^n - \vr_{K}^{n-1}}{\deltat} + \sum_{\sigma \in {\cal E}(K)} |\sigma| \vr_\sigma^{n,{\rm up}}[{\bu_\sigma^n}\cdot\bn_{\sigma,K}]=0,\quad\forall K \in {\cal{T}}, \\
	& \sum_{K\in{\cal T}}\frac{|K|}{\deltat} \Big({\vr^n_K{{\bu}}^n_{K}- \vr^{n-1}_K{{\bu}}^{n-1}_{K}} \Big)\cdot \bv_K+ \sum_{K\in {\cal T}}\sum_{\sigma \in {\cal E}(K)} |\sigma|\vr^{n,{\rm up}}_\sigma {{ \hat\bu}}_{\sigma}^{n,{\rm up}}[\bu^n_\sigma\cdot \bn_{\sigma,K}]\cdot \bv_K \nonumber \\
	&\vspace{-.3cm}\qquad  \qquad - \sum_{K\in{\cal T}}p(\vr^n_K)\sum_{\sigma \in {\cal E}(K)}  |\sigma|\bv_\sigma\cdot {\vc n}_{\sigma,K}+\mu\sum_{K\in{\cal T}}\int_K \nabla\bu^n : \nabla\bv \  \dx  \label{dmom}\\ \vspace{-.3cm}
	&\qquad \qquad    \qquad \qquad\qquad + ( \mu +\lambda)\sum_{K\in{\cal T}}\int_K{\rm div}\bu^n{\rm div}\bv\,  \dx =0,
	\; \forall \bv \in {\bW_h(\Omega) }. \nonumber
\end{align}
\end{subequations}
Note that the boundary condition $\bu^n_{\sigma}=0\quad\mbox{if $\sigma\in {\cal E}_{\rm ext}$}$ is ensured by the definition of the space $V_h(\Omega)$.
Note also that if $\sigma \in {\cal E}_{\rm int}$, $\sigma = K|L$, one has, following \eqref{vhat} and \eqref{upwind1},
\[
{{ \hat\bu}}_{\sigma}^{n,{\rm up}}=u^n_K=\frac 1 {|K]} \int_K \bu^n(x) \dx \textrm{ if } \bu^n_\sigma\cdot \bn_{\sigma,K}>0 \textrm{ and }
{{ \hat\bu}}_{\sigma}^{n,{\rm up}}=u^n_L=\frac 1 {|L]} \int_L \bu^n(x) \dx \textrm{ if } \bu^n_\sigma\cdot { \bn_{\sigma,K}< 0}.
\]

It is well known that any solution $(\vr^n)_{1\le n \le N}\subset (L_h(\Omega))^N$ satisfies $\vr^n >0$  thanks to the upwind choice in \eqref{dcont} { (see e.g. \cite{GGHL2008baro,KARPER})}.
Furthermore, summing \eqref{dcont} over $K \in \mesh$ immediately yields the total conservation of mass, which reads:
\begin{equation}\label{masscons}
	\forall n=1,...N,\quad \int_\Omega \vr^n \dx = \int_\Omega \vr^0 \dx.
\end{equation}

We finally state in this section the existence result, which can be proved by a topological degree argument,
\cite{GGHL2008baro,KARPER}.
\begin{Proposition}[Existence]\label{Theorem1}
	Let $(\vr^0,\bu^0){ \in L_h^{++}}(\Omega) \times \bW_h(\Omega)$.
	Under assumptions \eqref{visc} and \eqref{hypp}, Problem \eqref{scheme} admits at least one  solution
	\[
		(\vr^n)_{1\le n \le N}\in [{ L_h^{++}}(\Omega)]^N, (\bu^n)_{1\le n \le N} \in  [\bW_h(\Omega)]^N.
	\]
\end{Proposition}

\subsection{Main result: error estimate}

Let $(r,\bU) : [0,T]\times\overline\Omega\mapsto (0,\infty)\times \R^3$ be $C^2$ functions such that $\bU=\bzero$ on $\partial \Omega$.
Let  $(\vr,\bu)$ be a  solution  of the discrete problem \eqref{scheme}.
Inspired by (\ref{ent}), we introduce the discrete relative energy functional
\begin{align}\label{dent}
	{\cal E}(\vr^n,\bu^n\Big|r^n,\bU^n) &=\int_{\Omega}\Big(\frac 12\vr^n|\hat\bu^n-\hat{\bU}_h^n|^2
	+ E(\vr^n|\hat r^n)\Big)\dx \\
	&=\sum_{K\in{\cal T}}|K|\Big(\frac 12\vr_K|\bu_K^n-\bU^n_{h,K}|^2+ E(\vr_K^n|r_K^n)\Big), \nonumber
\end{align}
where
\begin{equation}\label{notation2-}
r^n(x)=r(t_n,x),\;\bU^n(x)=\bU(t_n,x),\;{ n=0,\ldots,N,}
\end{equation}
 $(\vr^n,\bu^n)$ is defined in (\ref{notation0}), and $E$ is defined by \eqref{E}.
Let us finally introduce the notations
\[
M_0=\sum_{K\in{\cal K}}|K|\vr_K^0,\mbox{ and } E_0=\sum_{K\in{\cal K}}|K|\Big(\frac 12\vr_K^0|\bu_K^0|^2 + H(\vr_K^0)\Big).
\]

Now, we are ready to state the main result of this paper.
For the sake of clarity, we shall state the theorem and perform the proofs only in the most interesting three dimensional case.
The modifications to be done for the two dimensional case, which is in fact more simple,  are mostly due to the different Sobolev embedings, and are left to the interested reader.

\begin{Theorem}[Error estimate]\label{Main}
Let $\theta_0 > 0$ and ${\cal{T}} $ be a { regular} triangulation of a bounded polyhedral domain $ \Omega\subset\R^3 $ { introduced  in Section \ref{3.1}} such that $ \theta \ge \theta_0 $, where $ \theta $ is defined in (\ref{reg}).
Let $p$ be a twice continuously differentiable function satisfying assumptions (\ref{hypp}), (\ref{pressure1}) with $\gamma\ge 3/2$, and the additional assumption  (\ref{pressure2}) in the case $\gamma<2$.
Let the viscosity coefficients satisfy assumptions (\ref{visc}).
Suppose that $(\vr^0,\bu^0) \in L_h^{+}(\Omega) \times \bW_h(\Omega)$ and that $(\vr^n)_{1\le n \le N}\subset [L_h^{+}(\Omega)]^N$, $(\bu^n)_{1\le n \le N} \subset [\bW_h(\Omega)]^N$ is a solution of the discrete problem \eqref{scheme}.
 Let $(r,\bU)$ in the class
\begin{subequations}\label{dr,U}
\begin{align}
&  r\in C^2([0,T]\times\overline\Omega),\quad 0<\underline r:=\min_{(t,x)\in \overline Q_T}\le r(t,x)\le\overline r:= \max_{(t,x)\in \overline Q_T}r(t,x), \label{dr}\\
 & \bU\in C^2([0,T]\times\overline\Omega;\R^3), \;\bU|_{\partial\Omega}=0 \label{U}
\end{align}
   \end{subequations}
   be a (strong) solution of problem \eqref{pbcont}.
  Then there exists
\begin{multline*}
  c=c\Big(T,|\Omega|, {\rm diam}(\Omega),\theta_0,\gamma, M_0, E_0, \underline r,\overline r,\\ |p'|_{C^1([\underline r,\overline r])},
   \|(\nabla r, \partial_t r, \partial_t\nabla r, \partial^2_t r,  \bU, \nabla\bU, { \nabla^2\bU}, \partial_t\bU, { \partial_t^2\bU,} \partial_t\nabla\bU)\|_{L^\infty(Q_T;{ \Rm^{68}})}\Big) \in (0,+\infty)
\end{multline*}
  (independent of $h$, $\deltat$) such that for any $m=1,\ldots,N,$
  \begin{equation}\label{errorestimate}
  {\cal E}(\vr^{ m},\bu^{ m}\Big|r^{ m},\bU^{ m}){ +\deltat\sum_{n=1}^m\sum_{{K}\in {\cal T}}\int_K|\Grad(u^n-\vc U^n_h)|^2{\rm d x}}
  \le c\Big({\cal E}(\vr^0,\bu^0\Big|r^0,\bU^0)+ h^{A}
  +\sqrt{\deltat}\Big),
  \end{equation}
where
\begin{equation}\label{defAmain}
 A = \begin{cases}
		\frac {2\gamma-3}\gamma\quad\mbox{if }\gamma\in (3/2,2],\\
		1/2\quad\mbox{if }\gamma> 2.
    \end{cases}
\end{equation}
\end{Theorem}

{ Starting from this point, unlike in Section \ref{intro}, here and hereafter, the symbol ${\cal E}$ refers always to the {\it discrete} relative energy functional defined in (\ref{dent}).}

\begin{Remark}\label{rem1}{\rm ~}\\
{
Assumptions (\ref{dr,U}) on the regularity of the strong solution $(r,\vc U)$ in Theorem \ref{Main} may be slightly relaxed: It is enough to suppose
$$
(r,\vc U)\in C^1([0,T]\times\overline\Omega;\R^4),\;\nabla^2 \vc U\in C([0,T]\times\overline\Omega;\R^3),\;
0<\inf_{(t,x)\in \overline Q_T} r(t,x),
$$
$$
\partial_t^2r\in L^1(0,T;L^{\gamma'}(\Omega)),\;\partial_t\nabla r\in L^2(0,T; L^{6\gamma/(5\gamma-6)}(\Omega;\R^3)),\;
(\partial_t^2\vc U,\partial_t\nabla\vc U)\in L^2(0,T; L^{6/5}(\Omega;\R^{12})).
$$
The constant in the error estimate depends on $\underline r$ and the norms of $r$ and $\vc U$ in these spaces.
This improvement is at the price of more technicalities in  estimates of several residual terms, namely in estimates (\ref{R2.1}--\ref{R2.2}), (\ref{R5.2}), (\ref{R6.4}), (\ref{cR2.1}), (\ref{cR2.3}--\ref{cR2.4})
and (\ref{cP1}).}
\end{Remark}
\begin{Remark}\label{rmq-thmp}{\rm ~ }\\
\vspace{-3mm}
\begin{enumerate}
\item
Theorem \ref{Main} holds also for two dimensional bounded { polygonal} domains under the assumption that $\gamma\ge1$.
Assumption (\ref{pressure2}) on the asymptotic behavior of pressure near $0$ is no more necessary in this case.
The value of $A$ in the error estimate (\ref{errorestimate}) is
\[
A =\left\{
\begin{array}{c}
\frac {2\gamma-2}\gamma\quad\mbox{if $\gamma\in (1,2]$},\\
1\quad\mbox{if $\gamma> 2$}.
\end{array}
\right.
\]
{
\item Suppose that the discrete initial data $(\vr^0,\vc u^0)$  coincide with the projection $(\hat r^0,\hat {\vc U}_{h}^0)$ of the initial data determining the strong solution. Then formula (\ref{errorestimate}) provides in terms of classical Lebesgue spaces the following bounds:
    $$
    \|\vr^m-r^m\|^2_{L^2(\Omega\cap\{\underline r/2\le\vr^m\le 2\overline r\})}+
    \|\hat{\vc u}^m- {\vc U}^m\|^2_{L^2(\Omega\cap\{\underline r/2\le\vr^m\le 2\overline r\})}\le
     c\Big(h^{A}
  +\sqrt{\deltat}\Big)
    $$
    for the "essential part" of the solution (where the numerical density remains bounded from above and from below outside zero), and
    $$
    |\{\vr^m\le\underline r/2\}|+ |\{\vr^m\ge 2\overline r\}| +\|\vr^m\|^\gamma_{L^\gamma(\Omega\cap \{\vr^m\ge 2\overline r\})}+\|\vr^m|\hat{\vc u}^m-{\vc U}^m|^2\|_{L^1(\Omega\cap\{\vr^m\ge 2\overline r\})}\le c\Big(h^{A}
  +\sqrt{\deltat}\Big)
    $$
    for the "residual part" of the solution, where the numerical density can be "close" to zero or infinity.
    (In the above formula, for $B\subset\Omega$, $|B|$ denotes the Lebesgue measure of $B$.)

    Moreover, in the particular case of $p(\vr)=\vr^2$ (that however represents a non physical situation)
    $E(\vr| r)=(\vr- r)^2$ and the error estimate (\ref{errorestimate}) reads
    $$
    \|\vr^m- r^m\|^2_{L^2(\Omega)}+
    \|\vr^m|\hat{\vc u}^m- {\vc U}^m|^2\|_{L^1(\Omega)}\le
     c\Big(\sqrt h
  +\sqrt{\deltat}\Big)
  $$
}

\item Theorem \ref{Main} can be viewed as a discrete version of Proposition \ref{proposition}.
It is to be noticed that the assumptions on the constitutive law for pressure guaranteeing the error estimates for the scheme \eqref{scheme} are somewhat stronger ($\gamma\ge 3/2$) than the assumptions needed for the stability in the continuous case ($\gamma\ge 1$).
The threshold value  $\gamma=3/2$ is however in accordance with the existence theory of weak solutions.
The assumptions on the regularity of the strong solution to be compared with the discrete solution in the scheme are slightly stronger than those needed to establish the stability estimates in the continuous case.

\item If $d=3$, we notice that the assumptions on the pressure (as function of the density) in Theorem \ref{Main} are  compatible with the isentropic case $p(\vr)=\vr^\gamma$ for all values $\gamma\ge 3/2$.

    \item { The scheme \cite{KARPER} contains in addition artificial stabilizing terms both in the continuity and momentum equations. These terms are necessary for the convergence proof in \cite{KARPER} even for the large values of $\gamma$. It is to be noticed that the error estimate  in Theorem \ref{Main} is formulated for the numerical scheme without these stabilizing terms. Of course similar error estimate is a fortiori valid also for the scheme with the stabilizing terms, however, this issue is not discussed in the present paper.}

\end{enumerate}
\end{Remark}

{ The rest of the paper is devoted to the proof of Theorem \ref{Main}. For the sake of simplicity, and in order to simplify notation, we present the proof for the uniformly regular mesh meaning that there exist positive numbers $c_i=c_i(\theta_0)$ such that
\begin{equation}\label{reg1}
c_1h_K\le h\le c_2 h_\sigma\le c_3 h_K,\quad c_1|K|\le  |\sigma| h\le c_2|\sigma|h_K\le c_3|\sigma|h_\sigma\le c_4 |K|
\end{equation}
for any $K\in {\cal T}$ and any $\sigma\in {\cal E}$.
The necessary (small) modifications needed to accommodate the regular mesh satisfying only (\ref{reg1-}) are straightforward. Even with this simplification the proof is quite involved, and some details have to be necessarily omitted to keep its length within reasonable bounds. The reader can eventually find them in the extended version of this paper available on ArXiv \cite{GHMNArxive}.
}

\section{Mesh independent estimates}\label{4}

We start by a remark on the notation.
From now on, the letter $c$ denotes positive numbers that may tacitly depend on $T$, $|\Omega|$, ${\rm diam}(\Omega)$, $\gamma$, $\alpha$, $\theta_0$, $\lambda$ and $\mu$, and on other parameters;
the dependency on these other parameters (if any) is always explicitly indicated in the arguments of these numbers.
These numbers can take different values even in the same formula.
They are always independent of the size of the discretisation $\deltat$ and $h$.

\subsection{Energy { Identity}}
Our analysis starts with an energy inequality, which is crucial both in the convergence analysis  and in the error analysis.
We recall this energy estimate which is already given in \cite{KARPER}, along with its proof for the sake of completeness.

\begin{lm}\label{Theorem3}
  Let $(\vr^0,\bu^0) \in L_h^{+}(\Omega) \times \bW_h(\Omega)$ and suppose that $(\vr^n)_{1\le n \le N}\in [L_h^{+}(\Omega)]^N$, $(\bu^n)_{1\le n \le N} \in [\bW_h(\Omega)]^N$ is a solution of the discrete problem \eqref{scheme} with the pressure $p$ satisfying condition (\ref{hypp}).
Then there exist
\begin{align*}
&\overline \vr^n_{\sigma}\in [\min(\vr^n_K,\vr^n_L), \max(\vr^n_K,\vr^n_L)],\;\sigma=K|L\in {\cal E}_{\rm int},\; n=1,\ldots,N \\
&\overline\vr_K^{n-1,n}\in [\min(\vr^{n-1}_K,\vr^n_K), \max(\vr^{n-1}_K,\vr^n_K)],\; K\in {\cal T},\; n=1,\ldots,N
\end{align*}
such that
\begin{multline}
\sum_{K\in {\cal T}}{|K|}\Big(\frac 12\vr^m_K|{\bu}^m_K|^2
 +H(\vr_K^m)\Big)
-\sum_{K\in {\cal T}}|K|\Big(\frac 12\vr^{0}_K|{\bu}^{0}_K|^2
+H(\vr_K^{0})\Big)
\\
+\deltat \sum_{n=1}^m\sum_{K\in {\cal T}}\Big(\mu\int_K|\Grad\bu^n|^2 \dx+( \mu +\lambda)\int_K|{\rm div}\bu^n|^2 \dx\Big)
\\
+ [D^{m,|\Delta\bu|}_{\rm time}]+ [D^{m,|\Delta\vr|}_{\rm time}]+ [D^{m,|\Delta\bu|}_{\rm space}] +  [D^{m, |\Delta\vr|}_{\rm space}]= 0,
\label{denergyinequality}
\end{multline}
for all $m=1,\ldots,N$,
where
\begin{subequations}\label{upwinddissipation}
	\begin{align}
		& [D^{m,|\Delta\bu|}_{\rm time}]=\sum_{n=1}^m\sum_{K\in {\cal T}}{|K|}\vr_K^{n-1}\frac {|{\bu}_K^n-{\bu}_K^{n-1}|^2} 2
		\label{upwinddissipation_1}\\
		&[D^{m,|\Delta\vc \vr|}_{\rm time}]=\sum_{n=1}^m\sum_{K\in {\cal T}}|K|H''(\overline\vr_K^{n-1,n})\frac {|{\vr}_K^n-{\vr }_K^{n-1}|^2} 2,
		\label{upwinddissipation_2}\\
		&[D^{m,|\Delta\bu|}_{\rm space}]=\deltat\sum_{n=1}^m\sum_{\sigma=K|L\in{\cal E}_{\rm int}}|\sigma|\vr_\sigma^{n,{\rm up}}\frac{({\bu}^n_K-{{\bu}}^n_L)^2} 2\;|{\bu}^n_\sigma\cdot\bn_{\sigma,K}|,
		\label{upwinddissipation_3}\\
		&[D^{m, |\Delta\vr|}_{\rm space}]=\deltat\sum_{n=1}^m\sum_{\sigma=K|L \in{\cal E}_{\rm int}}|\sigma|H''(\overline \vr^n_{\sigma})\frac{(\vr^n_K-\vr^n_L)^2}2 \;|{\bu}^n_\sigma\cdot{\vc n}_{\sigma,K}|.
		\label{upwinddissipation_4}
\end{align}
 \end{subequations}
\end{lm}

\begin{proof}
Mimicking the formal derivation of the total energy conservation in the continuous case we take as test function $\bv = \bu^n$ in the discrete momentum equation (\ref{dmom})$^n$ and obtain
\begin{equation}\label{denergy1}
	 I_1+ I_2+I_3+I_4 = 0,
\end{equation}
where
\begin{align*}
& { I_1=\sum_{K\in {\cal T}}\frac {|K|}\deltat(\vr_k^n\bu_K^n-\vr_K^{n-1}\bu_K^{n-1})\cdot\bu_K^n},
 	&&{ I_2}= \sum_{K\in {\cal T}}\stik |\sigma|\vr^{n,{\rm up}}_\sigma {{ \hat\bu}}_{\sigma}^{n,{\rm up}}\cdot\bu^n_K\,[\bu^n_\sigma\cdot\bn_{\sigma,K}],
 \\
 &{ I_3} = -\sum_{K\in {\cal T}}\stik |\sigma|p(\vr^n_K)[\bu^n_{\sigma}\cdot{\vc n}_{\sigma,K}], &&
	{ I_4= \sum_{K\in{\cal T}}\int_K\Big(\mu\nabla\bu^n:\nabla\bu^n +( \mu +\lambda){\rm div}{\bu^n}{\rm div}{{\bu}^n}\Big) \dx.}
\end{align*}
Next, we multiply the continuity equation (\ref{dcont})$^n_K$ by $\frac 12|\bu^n_K|^2$ and sum over all
$K\in {\cal T}$. We get
\begin{equation}\label{denergy2}
I_5+I_6 = 0
\end{equation}
\[
\mbox{with }I_5 = -\sum_{K\in {\cal T}}\frac 12\frac{|K|}{\deltat} ({ \vr_{K}^n - \vr_{K}^{n-1}})|\bu^n_K|^2 \mbox{ and } I_6=-\sum_{K\in {\cal T}}\stik\frac 12 |\sigma| \vr_\sigma^{n,\rm up}[{\bu_\sigma^n}\cdot\bn_{\sigma,K}]|\bu^n_K|^2.
\]
Finally, we multiply the continuity equation (\ref{dcont})$^n_K$ by $H'(\vr_K^n)$ and sum over all $K\in {\cal T}$.
We obtain
\begin{equation}\label{denergy3}
I_7+I_8 =0,
\end{equation}
\[
\mbox{with }I_7=\sum_{K\in {\cal T}}\frac{|K|}{\deltat} ({ \vr_{K}^n - \vr_{K}^{n-1}}) H'(\vr^n_K)
 \mbox{ and } I_8=\sum_{K\in {\cal T}}\stik |\sigma| \vr_\sigma^{n,\rm up}[{\bu_\sigma^n}\cdot\bn_{\sigma,K}] H'(\vr^n_K).
\]

We now sum formulas (\ref{denergy1})--(\ref{denergy3}) in several steps.

\vspace{2mm}

{\bf Step 1:} {\it Term $I_1+ I_7$.}
We verify by a direct calculation  that
\[
I_1= \sum_{K\in {\cal T}}\frac{|K|}{\deltat}\Big(\frac 12 \vr^n_K|\bu^n_K|^2-\frac 12 \vr^{n-1}_K|\bu^{n-1}_K|^2\Big)
+\sum_{{ K\in{\cal T}}}\frac{|K|}{\deltat}\vr_K^{n-1}\frac {|{\bu}_K^n-{\bu}_K^{n-1}|^2} 2.
\]
In order to transform the term $I_7$, we employ the Taylor formula
\[
H'(\vr_K^n)\Big(\vr^n_K-\vr^{n-1}_K\Big)= H(\vr_K^n)-H(\vr_K^{n-1}) +\frac 12 H''(\overline\vr_K^{n-1,n})(\vr_K^n-\vr_K^{n-1})^2,
\]
where $\overline\vr_K^{n-1,n}\in [\min(\vr_K^{n-1},\vr_K^{n}), \max(\vr_K^{n-1},\vr_K^{n})]$.
Consequently,
\begin{multline}\label{denergy4}
I_1+I_7=\sum_{K\in {\cal T}}\frac{|K|}{\deltat}\Big(\frac 12 \vr^n_K|\bu^n_K|^2-\frac 12 \vr^{n-1}_K|\bu^{n-1}_K|^2\Big)
+ \sum_{K\in {\cal T}}\frac{|K|}{\deltat}\Big(H(\vr_K^n)-H(\vr_K^{n-1})\Big)
\\+\sum_{K\in{\cal T}}\frac{|K|}{\deltat}\vr_K^{n-1}\frac {|{\bu}_K^n-{\bu}_K^{n-1}|^2} 2
+\sum_{K\in{\cal T}}\frac{|K|}{\deltat}H''(\overline\vr_K^{n-1,n})\frac {|{\vr}_K^n-{\vr }_K^{n-1}|^2} 2.
\end{multline}
{\bf Step 2:} {\it Term $I_2+I_6$}.
The contribution of the face $\sigma=K|L$ to the sum $I_2+I_6$ reads, by virtue of (\ref{upwind1}),
\begin{multline*}
|\sigma|\, [\bu^n_\sigma\cdot\bn_{\sigma,K}]^+\,\vr_K \Big(|\bu_K^n|^2 -\bu^n_K\cdot\bu_L^n -
\frac 12|\bu^n_K|^2 + \frac 12|\bu^n_L|^2\Big) \\
+ |\sigma|\,[\bu^n_\sigma\cdot\bn_{\sigma,L}]^+\,\vr_L \Big(|\bu_L^n|^2 -\bu^n_K\cdot\bu_L^n -
\frac 12|\bu^n_L|^2 + \frac 12|\bu^n_K|^2\Big).
 \end{multline*}
Consequently,
\begin{equation}\label{denergy5}
I_2+I_6=\sum_{\sigma=K|L\in{\cal E}_{\rm int}} |\sigma| |\bu^n_\sigma\cdot\bn_{\sigma,K}|\vr_{\sigma}^{n,{\rm up}}
\frac{(\bu^n_K-\bu^n_L)^2} 2.
\end{equation}
{\bf Step 3:} {\it Term $I_3+I_8$.}
We have
\begin{multline*}
I_8=\sum_{K\in {\cal T}}\stik|\sigma|\,[{\bu_\sigma^n}\cdot\bn_{\sigma,K}]\,\Big( H'(\vr^n_K)( \vr_\sigma^{n,\rm up}-\vr^n_K)+ H(\vr^n_K)\Big)
\\+ \sum_{K\in {\cal T}}\stik|\sigma|\,[{\bu_\sigma^n}\cdot\bn_{\sigma,K}]\,\Big( \vr^n_K H'(\vr^n_K) -H(\vr^n_K)\Big).
 \end{multline*}
Recalling (\ref{upwind1}), we may write the contribution of the face $\sigma=K|L$ to the first sum in $I_8$; it reads
\begin{multline*}
|\sigma|\,[{\bu_\sigma^n}\cdot\bn_{\sigma,K}]^+\,\Big(H(\vr^n_K)- H'(\vr^n_L)(\vr^n_K-\vr^n_L)-H(\vr^n_L)\Big)
\\+|\sigma|\,[{\bu_\sigma^n}\cdot\bn_{\sigma,L}]^+\,\Big(H(\vr^n_L)- H'(\vr^n_K)(\vr^n_L-\vr^n_K)-H(\vr^n_K)\Big).
 \end{multline*}
Recalling that $rH'(r)-H(r)=p(r)$, we get, employing the Taylor formula
\[
I_3+I_8=\sum_{\sigma=K|L\in {\cal E}_{\rm int}}\,|{\bu_\sigma^n}\cdot\bn_{\sigma,K}| H''(\overline \vr^n_{\sigma})\frac{(\vr^n_K-\vr^n_L\Big)^2} 2
\]
with some $\overline \vr^n_{\sigma}\in [\min(\vr^n_K,\vr^n_L), \max(\vr^n_K,\vr^n_L)]$.

\vspace{2mm}

{\bf Step 4:} {\it Conclusion}

Collecting the results of Steps 1-3 we arrive at
\begin{multline}\label{denergyinequality1}
\sum_{K\in {\cal T}}\frac 12\frac{|K|}{\deltat}\Big(\vr^n_K|{\bu}^n_K|^2-\vr^{n-1}_K|{\bu}^{n-1}_K|^2\Big)
+ \sum_{K\in {\cal T}}\frac{|K|}{\deltat}\Big(H(\vr_K^n)-H(\vr_K^{n-1})\Big) +\sum_{K\in {\cal T}}\Big(\mu\int_K|\Grad\bu^n|^2 \dx
\\+( \mu+\lambda)\int_K|{\rm div}\bu^n|^2 \dx\Big)
+\sum_{K\in {\cal T}}\frac{|K|}{\deltat}\vr_K^{n-1}\frac {|{\bu}_K^n-{\bu}_K^{n-1}|^2} 2
+\sum_{K\in {\cal T}}\frac{|K|}{\deltat}H''(\overline\vr_K^{n-1,n})\frac {|{\vr}_K^n-{\vr }_K^{n-1}|^2} 2
\\+ \sum_{\substack{\sigma\in{\cal E}_{\rm int} \\ \sigma=K|L}}|\sigma|\vr_\sigma^{n,{\rm up}}\frac{({\bu}^n_K-{{\bu}}^n_L)^2} 2\;|{\bu}^n_\sigma\cdot\bn_{\sigma,K}|
+\sum_{\substack{\sigma\in{\cal E}_{\rm int} \\ \sigma=K|L}}|\sigma|H''(\overline \vr^n_{\sigma})\frac{(\vr^n_K-\vr^n_L\Big)^2} 2 \;|{\bu}^n_\sigma\cdot{\vc n}_{\sigma,K}|= 0.
\end{multline}
At this stage, we get the statement of Lemma \ref{Theorem3} by multiplying (\ref{denergyinequality1})$^{n}$ by $\deltat$ and summing from $n=1$ to $n=m$. Lemma \ref{Theorem3} is proved.
\end{proof}

\subsection{Estimates}

We have the following corollary of Lemma \ref{Theorem3}.
\begin{cor}\label{Corollary1}
\begin{description}
\item{(1)}
Under assumptions of Lemma \ref{Theorem3}, there exists $c=c(M_0,E_0)>0$ (independent of $h$ and $\deltat$) such that
\begin{equation}\label{est0}
|\bu|_{L^2(0,T;V^2_h(\Omega;\Rm^3)}\le c
\end{equation}
\begin{equation}\label{est1}
\|\bu\|_{L^2(0,T;L^6(\Omega;\Rm^3))}\le c
\end{equation}
\begin{equation}\label{est2}
\|\vr\hat{\bu}^2\|_{L^\infty(0,T;L^1(\Omega))}\le c.
\end{equation}
 \item{(2)} If in addition the pressure satisfies assumption (\ref{pressure1}) then
\begin{equation}\label{est3}
\|\vr\|_{L^\infty(0,T;L^\gamma(\Omega))}\le c
\end{equation}
\item{(3)} If the pair $(r,\bU)$ belongs to the class (\ref{dr,U}) there exists $c=c(M_0,E_0,\underline r,\overline r, \|\bU, \nabla \bU\|_{L^\infty(Q_T;\Rm^{12})})>0$
such that for all $n=1,\ldots,N$,
\begin{equation}\label{est4}
{\cal E}(\vr^n,\bu^n|r^n,\bU^n)\le c,
\end{equation}
where the discrete relative energy ${\cal E}$ is defined in (\ref{dent}).
\end{description}
\end{cor}
\begin{proof}
Recall that
\[|\bu|^2_{L^2(0,T;V^2_h(\Omega;\Rm^3)} = \deltat \sum_{n=1}^N\sum_{K\in {\cal T}}\int_K|\Grad\bu^n|^2 \dx;\]
the estimate (\ref{est0}) follows from (\ref{denergyinequality}).
The estimate (\ref{est1}) holds due to imbedding (\ref{sob1}) in Lemma \ref{Lemma2+} and bound (\ref{est0}).
The estimate (\ref{est2}) is just a short transcription of the bound for the kinetic energy in (\ref{denergyinequality}).

{  We prove estimate (\ref{est3}). First, we deduce from (\ref{hypp}) and the definition (\ref{H}) of $H$ that $0\le -H(\vr)\le c_1$ with some $c_1>0$, provided $0<\vr \le 1$ and $H(\vr)> 0$ if $\vr>1$. This fact in combination with the bound for  $\int_\Omega H(\vr){\rm d}x$ derived in (\ref{denergyinequality}) yields
 \begin{equation}\label{H+}
 \int_\Omega |H(\vr)|{\rm d}x\le c<\infty.
\end{equation}
Second, relations (\ref{hypp}--\ref{pressure2}) imply that there are $\overline\vr>1$ and $0<\underline p<\overline p<\infty$ such that
$$
\left\{
\begin{array}{c}
\vr^{\alpha}p_0/2\le\frac {p(\vr)}{\vr^2}\mbox{if $0<\vr<1/\overline\vr$},\\
\underline p\le \frac {p(\vr)}{\vr^2}\le\overline p\;\mbox{if $1/\overline\vr\le\vr\le \overline\vr$},\\
\vr^{\gamma-2} {p_\infty}/2\le \frac {p(\vr)}{\vr^2}\;\mbox{if $\vr>2\overline\vr$}
\end{array}
\right\}.
$$
Using these bounds and the definition (\ref{H}) of $H$ we verify that
$$
\vr^\gamma\le c (|H(\vr)|+ \vr +1)
$$
with a convenient positive constant $c$. Now, bound (\ref{est3}) follows readily
from the boundedness of
$\int_\Omega\vr^m{\rm d}x\equiv\sum_{K\in {\cal T}}|K|\vr_K^m$ and $\int_\Omega H(\vr^m){\rm d}x\equiv\sum_{K\in {\cal T}}{|K|} H(\vr_K^m)$   established in (\ref{masscons}) and
(\ref{denergyinequality}).
}

 Finally, to get (\ref{est4}), we have employed  (\ref{E}), (\ref{dent}), (\ref{masscons}), (\ref{H+})
to estimate $\int_\Omega E(\vr^n|\hat r^n)\dx$ and (\ref{est2}), (\ref{L2-3}), (\ref{L1-3}) to evaluate $\sum_{K\in {\cal T}}\int_K\vr_K^n|\bU_{h,K}^n-\bu_K^n|^2\dx$.
\end{proof}

The following estimates are obtained thanks to the numerical diffusion due to the upwinding, as is classical in the framework of hyperbolic conservation laws, see e.g. \cite{egh-book}.

\begin{lm}[Dissipation estimates on the density]\label{Lemma5}
Let $(\vr^0,\bu^0) \in L_h^{+}(\Omega) \times \bW_h(\Omega)$.
Suppose that $(\vr^n)_{1\le n \le N}\subset [L_h^{+}(\Omega)]^N$, $(\bu^n)_{1\le n \le N} \subset [\bW_h]^N(\Omega)$ is a solution of problem \eqref{scheme}.
Finally assume that the pressure satisfies hypotheses (\ref{hypp}) and
(\ref{pressure1}).
Then we have:
\begin{description}
\item {(1)} If $\gamma\ge 2$ then there exists $c=c(\gamma,\theta_0, E_0)>0$ such that
\begin{equation}\label{dissipative2}
\deltat \sum_{n=1}^N\sum_{\sigma=K|L \in{\cal E}_{\rm int}}|\sigma|\frac{(\vr^n_K-\vr^n_L)^2}{{\rm max}(\vr^n_K,\vr^n_L)} \;|{\bu}^n_\sigma\cdot{\vc n}_{\sigma,K}|\le c.
\end{equation}
\item{(2)} If $\gamma\in [1, 2)$ and the pressure satisfies additionally assumption (\ref{pressure2})
then there exists $c=c(M_0, E_0)>0$ such that
\begin{multline}\label{dissipative1}
\deltat\sum_{n=1}^N \sum_{\sigma=K|L \in{\cal E}_{\rm int}}|\sigma|\frac{(\vr^n_K-\vr^n_L)^2}{[{\rm max}(\vr^n_K,\vr^n_L)]^{2-\gamma}} 1_{\{\overline\vr^n_\sigma\ge 1\}} \;|{\bu}^n_\sigma\cdot{\vc n}_{\sigma,K}|
\\+
\deltat \sum_{n=1}^N\sum_{\sigma=K|L \in{\cal E}_{\rm int}}|\sigma|(\vr^n_K-\vr^n_L)^2 1_{\{\overline\vr^n_\sigma <1\}} \;|\bu^n_\sigma\cdot{\vc n}_{\sigma,K}|\le c,
\end{multline}
where the numbers $\overline\vr^n_\sigma$ are defined in Lemma \ref{Theorem3}.
\end{description}
\end{lm}
\begin{proof} We start by proving the simpler statement \textit{ (2)}.
Taking into account the continuity of the pressure, we deduce from assumptions (\ref{pressure1}) and (\ref{pressure2}) that there exist numbers $\overline p_0>0$, $\overline p_\infty>0$ such that
\[
H''(s)\ge \begin{cases}
 \frac{\overline p_\infty}{s^{2-\gamma}},\quad\mbox{if } s\ge 1,
\\  \overline p_0 s^\alpha\ge
\overline p_0, \quad\mbox{if }s< 1,.
          \end{cases}
\]
whence, splitting the sum in the definition of the term $[D^{N,\Delta\vr}_{\rm space}]$ (see (\ref{upwinddissipation_4})) into two sums, where $(\sigma, n)$ satisfies $\overline\vr_\sigma^n\ge 1$ for the first one and $\overline\vr_\sigma^n< 1$ for the second, we obtain the desired  result.

Let us now turn to the proof of  statement \textit{(1)}.
Multiplying the discrete continuity equation (\ref{dcont})$_K^n$ by $\ln \vr_K^n$ and summing over $K\in {\cal T}$, we get
\begin{equation*}
\sth |K| \frac{ \vr_{K}^n - \vr_{K}^{n-1}}{\deltat}\ln\vr_K^n  + \sth \sum_{\sigma\in {\cal E}(K),\sigma=K|L}(\ln\vr_K^n)\vr_\sigma^{n,{\rm up}} \bu_\sigma^n\cdot\bn_{\sigma,K}=0.
\end{equation*}
By virtue of  the convexity of the function $\vr\mapsto \vr\ln\vr-\vr$ on the positive real line, and due to the Taylor formula, we have
\[
\vr_K^n\ln \vr_K^n-\vr_K^{n-1}\ln \vr_K^{n-1} -( \vr_K^n-\vr_K^{n-1})\le \ln \vr_K^n (\vr_K^n-\vr_K^{n-1});
\]
whence, thanks to the mass conservation \eqref{masscons} and the definition of $\vr_\sigma^{{\rm up}}$, we arrive at
\begin{multline*}
\sth |K| \frac{ \vr_{K}^n\ln\vr_K^n- \vr_{K}^{n-1}\ln\vr_K^{n-1}}{\deltat}
+ \stkl |\sigma| \vr_K^n [\bu_\sigma^n \cdot \bn_{\sigma,K}]^+ \Big( \ln\vr_K^n-\ln\vr_L^n\Big)
\\+ \stkl |\sigma| \vr_L^n [\bu_\sigma^n \cdot \bn_{\sigma,L}]^+ \Big( \ln\vr_L^n-\ln\vr_K^n\Big)\le 0,
\end{multline*}
or equivalently
\begin{multline}\label{est5}
\deltat\stkl |\sigma|  [\bu_\sigma^n \cdot \bn_{\sigma,K}]^+ \Big( \vr_K^n(\ln\vr_K^n-\ln\vr_L^n) -(\vr_K^n-\vr_L^n)\Big)
+\deltat\stkl |\sigma| [\bu_\sigma^n \cdot \bn_{\sigma,L}]^+ \Big( \vr_L^n ( \ln\vr_L^n-\ln\vr_K^n)-(\vr_L^n-\vr_K^n)\Big)\le
\\-\sth |K| \Big({ \vr_{K}^n\ln\vr_K^n- \vr_{K}^{n-1}\ln\vr_K^{n-1}}\Big)
 +\deltat\stkl |\sigma| \Big([\bu_\sigma^n \cdot \bn_{\sigma,K}]^+ (\vr_L^n-\vr_K^n))+  [\bu_\sigma^n \cdot \bn_{\sigma,L}]^+ (\vr_K^n-\vr_L^n))\Big).
\end{multline}
From \cite[Lemma C.5]{FettahGallouet2013Stokes}, we know that if $\varphi$ and $\psi$ are  functions in $C^1((0,\infty);\R)$ such that $s\psi'(s)=\varphi'(s)$ for all $ s \in (0,\infty)$, then for any $(a,b)\in (0,\infty)^2$ there exits $c \in [a,b]$ such that
\[
		(\psi(b)-\psi(a))b -(\varphi(b)-\varphi(a)) = \frac{1}{2}(b-a)^2 \psi'(c).
\]
Applying this result with $\psi(s)=\ln s$, $\varphi(s)=s$ we obtain that the left hand side of \eqref{est5} is greater or equal to
\[
\deltat\stkl |\sigma|  \Big([\bu_\sigma^n \cdot \bn_{\sigma,K}]^+ + [\bu_\sigma^n \cdot \bn_{\sigma,L}]^+\Big)
\frac{(\vr^n_K-\vr^n_L)^2}{\max(\vr^n_K,\vr^n_L)}.
\]
On the other hand, the first term at the right hand side is bounded from above by  $\|\vr^n\|_{L^\gamma(\Omega)}^\gamma$.
 Finally the second term at the right hand side is equal to
\[
-\deltat\sum_{K\in {\cal T}}\int_K\vr_K^n{\rm div}\bu^n {\le \deltat\sum_{K\in {\cal T}}\|\vr_K\|_{L^2(K)}
\|{\rm div}\vc u^n\|_{L^2(K)},}
\]
whence bounded from above by $\deltat\|\bu^n\|_{V_h^2(\Omega;\Rm^3)}\|\vr^n\|_{L^2(\Omega)}$, { where we have used the H\"older inequality and the definition of the $V_h^2(\Omega)$-norm}. The statement \textit{(1)} of  Lemma \ref{Lemma5} now follows from the estimates of Corollary \ref{Corollary1}.

\end{proof}

%
%
\section{Exact relative energy inequality for the discrete problem}\label{5}

The goal of this section is to prove  the discrete version of the relative energy inequality.

\begin{Theorem}\label{Theorem4} 

Suppose that $\Omega\subset \R^3$ is a polyhedral domain and  ${\cal T}$ its regular triangulation { introduced in Section \ref{3.1}}. Let $p$ satisfy hypotheses (\ref{hypp}) and the viscosity coefficient $\mu$, $\lambda$ obey (\ref{visc}).
Let $(\vr^0,\bu^0) \in L_h^{+}(\Omega) \times \bW_h(\Omega)$ and suppose that $(\vr^n)_{1\le n \le N}\in [L_h^{+}(\Omega)]^N$, $(\bu^n)_{1\le n \le N} \in [\bW_h(\Omega)]^N$ is a solution of the discrete problem \eqref{scheme}.
Then there holds for all $m=1,\ldots,N$,
\begin{equation}\label{drelativeenergy}
\begin{aligned}
&\sum_{K\in {\cal T}}\frac 12{|K|}\Big(\vr^{ m}_K|{\bu}^m_K-{\bU}^m_{h,K}|^2-\vr^{0}_K|{\bu}^{0}_K-{\bU}^{0}_{h,K}|^2\Big)
+ \sum_{K\in {\cal T}}{|K|}\Big(E(\vr_K^m| r_K^m)-E(\vr_K^{0}| r_K^{0})\Big)
\\ &\phantom{\sum_{K\in {\cal T}}}\qquad+\deltat\sum_{n=1}^m\sum_{K\in {\cal T}}\Big(\mu\int_K|\Grad(\bu^n-\bU^n_h)|^2 \dx+( \mu +\lambda)\int_K|{\rm div}(\bu^n-\bU^n_h)|^2 \dx\Big)
\\&\phantom{\sum_{K\in {\cal T}}} \le\deltat\sum_{n=1}^m\sum_{K\in {\cal T}}\Big(\mu\int_K\Grad\bU^n_h:\Grad(\bU^n_h-\bu^n) \dx+(\mu +\lambda)\int_K{\rm div}\bU^n_h{\rm div}(\bU^n_h-\bu^n) \dx\Big)
\\&\phantom{\sum_{K\in {\cal T}}}\qquad+\deltat\sum_{n=1}^m\sum_{K\in{\cal T}}{|K|}\vr_K^{n-1}\frac{{\bU}_{h,K}^{n}-{\bU}_{h,K}^{n-1}}{\deltat}\cdot \Big(\frac{{\bU}_{h,K}^{n-1} + {\bU}_{h,K}^{n} }2  -
\bu_K^{n-1}\Big)
\\
&\phantom{\sum_{K\in {\cal T}}}\qquad-\deltat\sum_{n=1}^m\sum_{K\in{\cal T}}\stik|\sigma|\vr_\sigma^{n,{\rm up}}\Big(\frac{\bU^n_{h,K}+\bU^n_{h,L}}2-
\hat{\bu}_{\sigma}^{n,{\rm up}}\Big)\cdot{\bU}^n_{h,K} [\bu^n_\sigma\cdot\bn_{\sigma,K}]
\\
&\phantom{\sum_{K\in {\cal T}}}\qquad-\deltat\sum_{n=1}^m\sum_{K\in {\cal T}}\stik |\sigma|p(\vr^n_K)[{\bU}_{h,\sigma}^{n}\cdot\bn_{\sigma,K}]
\\&\phantom{\sum_{K\in {\cal T}}}\qquad+
\deltat\sum_{n=1}^m\sum_{K\in {\cal T}}\frac{|K|}{\deltat} ( r^n_K-\vr^n_K)\Big(H'( r^n_K)-H'( r^{n-1}_K)\Big)
\\
&\phantom{\sum_{K\in {\cal T}}}\qquad+\deltat\sum_{n=1}^m\sum_{K\in {\cal T}}\stik |\sigma|\vr_\sigma^{n,{\rm up}}H'( r_K^{ n-1})[\bu^n_\sigma\cdot\bn_{\sigma,K}],
\end{aligned}
\end{equation}
for any $ 0<r\in C^1([0,T]\times\overline\Omega)$, $\bU\in C^1([0,T]\times\overline\Omega)$, $\bU|_{\partial\Omega}=0$, { where we have used notation (\ref{notation2-}) for $r^n$, $\bU^n$ { and (\ref{vhat}--\ref{vtilde}) for $\vc U^n_h$, $\vc U^n_{h,K}$, $r^n_K$, $\vc u^n_{\sigma}$.}}
\end{Theorem}

We notice, comparing the terms in the ``discrete'' formula (\ref{drelativeenergy}) with the terms in the ``continuous''
formula (\ref{p5}), that Theorem \ref{Theorem4} represents a discrete counterpart of the ``continuous'' relative energy inequality  (\ref{p5}).
The rest of this section is devoted to its proof.
To this end, we shall follow the proof of the ``continuous'' relative energy inequality (see \cite{FeJiNo} and \cite{FENOSU}) and adapt it to the discrete case.

\begin{proof}
First, noting that the  numerical diffusion
{ represented by terms  (\ref{upwinddissipation_1}--\ref{upwinddissipation_4})}
in the energy identity (\ref{denergyinequality}) is positive, we infer
\begin{equation}\label{dentropy0}
I_1+I_2+I_3\le 0,
\end{equation}
with
\begin{align*}
&I_1 := \sum_{K\in {\cal T}}\frac 12\frac{|K|}{\deltat}\Big(\vr^n_K|{\bu}^n_K|^2-\vr^{n-1}_K|{\bu}^{n-1}_K|^2\Big), \qquad I_2  :=   \sum_{K\in {\cal T}}\frac{|K|}{\deltat}\Big(H(\vr_K^n)-H(\vr_K^{n-1})\Big), \\
&I_3  :=  \sum_{K\in {\cal T}}\Big(\mu\int_K|\Grad\bu^n|^2 \dx+(\mu +\lambda)\int_K|{\rm div}\bu^n|^2 \dx\Big).
\end{align*}

Next, we multiply the discrete continuity equation (\ref{dcont})$^n_K$ by $\frac 12 |{ \vc U^n_{h,K}}|^2$ and sum over $K\in{\cal T}$ to obtain
\begin{equation}
 I_4   := \sum_{K\in {\cal T}}\frac 12\frac{|K|}{\deltat} ({ \vr_{K}^n - \vr_{K}^{n-1}})|\bU^n_{h,K}|^2
  = -\sum_{K\in {\cal T}}\stik\frac 12 |\sigma| \vr_\sigma^{n,\rm up}[{\bu_\sigma^n}\cdot\bn_{\sigma,K}]|\bU^n_{h,K}|^2:=J_1
\label{dentropy2}
\end{equation}
In the next step,  taking  $-\bU^n$  as test function in the discrete momentum equation (\ref{dmom}); we get
\[
	I_5 =-\sum_{K\in{\cal T}}\frac{|K|}{\deltat} \Big({\vr^n_K{{\bu}}^n_{K}-\vr^{n-1}_K{{\bu}}^{n-1}_{K}}\Big)\cdot\bU_{h,K}^n =J_2+J_3+J_4,
\]
with
\begin{align*}
& J_2	  =\sum_{K\in {\cal T}}\stik |\sigma|\vr^{n,{\rm up}}_\sigma { \hat{\bu}}_{\sigma}^{n,{\rm up}}\cdot\bU^n_{h,K}\,[\bu^n_\sigma\cdot\vc 	n_{\sigma,K}],\\
& { J_3= \mu\sum_{K\in{\cal T}}\int_K\nabla\bu^n:\nabla\bU_h^n \dx+  (\mu+\lambda)\sum_{K\in{\cal T}}\int_K{\rm div}{\bu^n}{\rm div}{{\bU}_h^n} \dx }\\
&\mbox{ and }\\
& J_4=  -\sum_{K\in {\cal T}}\stik |\sigma|p(\vr^n_K)[\bU^n_{\sigma}\cdot{\vc n}_{\sigma,K}].
 \end{align*}
We then multiply the continuity equation (\ref{dcont})$^n_K$ by $H'(r_K^{n-1})$ and sum over all $K\in {\cal T}$ and obtain
\[
	-\sum_{K\in {\cal T}}\frac{|K|}{\deltat} ({ \vr_{K}^n - \vr_{K}^{n-1}}) H'(r^{n-1}_K)
	=\sum_{K\in {\cal T}}\stik |\sigma| \vr_\sigma^{n,\rm up}[{\bu_\sigma^n}\cdot\bn_{\sigma,K}] H'(r^{n-1}_K).
\]
Observing that $
	\vr_K^n H'(r_K^n)-\vr_K^{n-1}H'(r_K^{n-1})= \vr_K^n\Big(H'(r_K^n)-H'(r_K^{n-1})\Big) + (\vr_K^n-\vr_K^{n-1})H'(r^{n-1}_K),
$
we rewrite the last identity in the form
\begin{equation}\label{dentropy3}
\begin{aligned}
& I_6:=-\sum_{K\in {\cal T}}\frac{|K|}{\deltat} \Big(\vr_K^n H'(r_K^n)-\vr_K^{n-1}H'(r_K^{n-1})\Big)= J_5+J_6 \\
&\mbox{with }{ J_5 =-\sum_{K\in {\cal T}}\frac{|K|}{\deltat}\vr_K^n \Big(H'(r_K^n)-H'(r_K^{n-1})\Big)}  \mbox{ and } J_6 = \sum_{K\in {\cal T}}\stik |\sigma| \vr_\sigma^{n,\rm up}[{\bu_\sigma^n}\cdot\bn_{\sigma,K}] H'(r^{n-1}_K).
\end{aligned}
\end{equation}
Finally, thanks to { the}  convexity of the function $H$, we have
\begin{equation}\label{dentropy4}
\begin{aligned}I_7&:=\sum_{K\in{\cal T}}\frac{|K|}{\deltat}\Big[\Big(r_K^nH'(r_K^n)-H(r^n_K)\Big)-\Big(r_K^{n-1}H'(r_K^{n-1})-H(r^{n-1}_K)\Big)\Big]
\\
&=\sum_{K\in{\cal T}}\frac{|K|}{\deltat}r^n_K\Big(H'(r^n_K)-H'(r^{n-1}_K)\Big)-\sum_{K\in{\cal T}}\frac{|K|}{\deltat}\Big(H(r_K^n) -(r^n_K-r^{n-1}_K)H'(r^{n-1}_K)- H(r^{n-1}_K\Big)
\\
&\le \sum_{K\in{\cal T}}\frac{|K|}{\deltat}r^n_K\Big(H'(r^n_K)-H'(r^{n-1}_K)\Big):=J_7,
\end{aligned}
\end{equation}

Now, we gather the expressions (\ref{dentropy0})-(\ref{dentropy4}); this is performed in several steps.

\vspace{2mm}

{\bf Step 1:} {\it Term $I_1+I_4+I_5$.}
We observe that
\begin{equation*}
\left\{
\begin{aligned}
	& \frac {|{\bU}_{h,K}^n|^2}2(\vr_K^n-\vr_K^{n-1})=\frac{\vr_K^n|{\bU}_{h,K}^n|^2-\vr_K^{n-1}|{\bU}_{h,K}^{n-1}|^2}2+\vr_K^{n-1}\frac {{\bU}_{h,K}^{n-1}+{\bU}_{h,K}^n}2\cdot ({\bU}_{h,K}^{n-1}-{\bU}_{h,K}^n),
	\\
	& -(\vr_K^n\bu_K^n-\vr_K^{n-1} \bu_K^{n-1})\cdot{\bU}_{h,K}^{n}= -(\vr_K^n\bu_K^n\cdot{\bU}_{h,K}^{n}-\vr_K^{n-1} \bu_K^{n-1}\cdot{\bU}_{h,K}^{n-1}) -\vr_K^{n-1}\bu_K^{n-1}\cdot ({\bU}_{h,K}^{n-1}-{\bU}_{h,K}^n).
\end{aligned}
\right.
\end{equation*}
Consequently,
\begin{multline}  \label{term1}
	I_1+I_4+I_5= \sum_{K\in{\cal T}}\frac 12 \frac{|K|}{\deltat}\,\Big(\vr_K^n|\bu_K^n-{\bU}_{h,K}^{n}|^2- \vr_K^{n-1}|\bu_K^{n-1}-{\bU}_{h,K}^{n-1}|^2\Big)  \\
	-\sum_{K\in{\cal T}}{|K|}\vr_K^{n-1}\frac{{\bU}_{h,K}^{n}-{\bU}_{h,K}^{n-1}}{\deltat}\cdot \Big(\frac{{\bU}_{h,K}^{n-1} + {\bU}_{h,K}^{n} }2 - \bu_K^{n-1}\Big)
\end{multline}
\vspace{2mm}\noindent
{\bf Step 2:} {\it Term $J_1+J_2$.}
The contribution of the face $\sigma=K|L$ to $J_1$ reads
\[
-|\sigma|\vr_K^n \frac{\bU_{h,K}^n+\bU_{h,L}^n}2\cdot(\bU_{h,K}^n -\bU_{h,L}^n)\,[\bu^n_\sigma\cdot\bn_{\sigma,K}]^+
-|\sigma|\vr_L^n \frac{\bU_{h,K}^n+\bU_{h,L}^n}2\cdot(\bU_{h,L}^n -\bU_{h,K}^n)\,[\bu^n_\sigma\cdot\bn_{\sigma,L}]^+.
\]
Similarly, the contribution of the face $\sigma=K|L$ to $J_2$ is
\[
|\sigma|\vr_K^n\bu^n_K\cdot(\bU^n_{h,K}-\bU^n_{h,L})[\bu^n_\sigma\cdot\bn_{\sigma,K}]^+
+ |\sigma|\vr_L^n\bu^n_L\cdot(\bU^n_{h,L}-\bU^n_{h,K})[\bu^n_\sigma\cdot\bn_{\sigma,L}]^+.
\]
Consequently,
\begin{equation}\label{term2}
J_1+J_2= -\sum_{K\in{\cal T}}\sum_{\sigma=K|L\in {\cal E}_K}|\sigma|\vr_\sigma^{n,{\rm up}} \Big(\frac{\bU^n_{h,K}+ \bU^n_{h,L}}2- {\hat \bu}_{\sigma}^{n,{\rm up}}\Big)\cdot{\bU}^n_{h,K} [\bu^n_\sigma \cdot\bn_{\sigma,K}].
\end{equation}

\vspace{2mm}\noindent
{\bf Step 3:} {\it Term $I_3-J_3$.}
This term can be written in the form
\begin{equation}\label{term3}
\begin{aligned}
I_3-J_3 &=\sum_{K\in {\cal T}}\Big(\mu\int_K|\Grad(\bu^n-\bU_h^n)|^2 \dx+(\mu +\lambda)\int_K|{\rm div}(\bu^n-\bU_h^n)|^2 \dx\Big)
\\ &\qquad - \sum_{K\in {\cal T}}\mu\int_K\Big(\nabla\bU_h^n:\nabla(\bU_h^n-\bu^n)
+(\mu +\lambda)\int_K{\rm div}\bU_h^n{\rm div}(\bU_h^n-\bu^n)\Big).
\end{aligned}
\end{equation}
{\bf Step 4:} {\it Term $I_2+ I_6 +I_7$.}
By virtue of (\ref{dentropy0}), (\ref{dentropy3}--\ref{dentropy4}), we easily find that
\begin{equation}\label{term4}
I_2+I_6 +I_7=\sum_{K\in{\cal T}}\frac{|K|}{\deltat}\Big( E(\vr^n_K\,|\,r^n_K)-E(\vr^{n-1}_K\,|\,r^{n-1}_K)\Big),
\end{equation}
where the function $E$ is defined in (\ref{E}).

\vspace{2mm}
{\bf Step 5:} {\it Term $J_5+J_6+J_7$.}
Coming back to (\ref{dentropy3}--\ref{dentropy4}),
we deduce that
\begin{equation}\label{term5}
J_5+J_6+J_7=
\sum_{K\in {\cal T}}\frac{|K|}{\deltat}\big(r_K^n-\vr_K^n\Big)\Big(H'(r_K^n)-H'(r_K^{n-1})\Big)
+\sum_{K\in {\cal T}}\stik |\sigma| \vr_\sigma^{n,\rm up}[{\bu_\sigma^n}\cdot\bn_{\sigma,K}] H'(r^{n-1}_K).
\end{equation}

\vspace{2mm}
{\bf Step 6:} {\it Conclusion}

According to (\ref{dentropy0})--(\ref{dentropy4}), we have
\[\sum_{i=1}^7I_i\le\sum_{i=1}^7 J_i; \]
whence, writing this inequality by using expressions (\ref{term1})--(\ref{term5}) calculated in steps 1-5, we get
\begin{equation}\label{drelativeenergy1}
 \begin{aligned}
& \sum_{K\in {\cal T}}\frac12 \frac{|K|}{\deltat} \Big(\vr^n_K|{\bu}^n_K-{\bU}^n_{h,K}|^2 - \vr^{n-1}_K|{\bu}^{n-1}_K - {\bU}^{n-1}_{h,K}|^2\Big) + \sum_{K\in {\cal T}}\frac{|K|}{\deltat}\Big(E(\vr_K^n| r_K^n)-E(\vr_K^{n-1}| r_K^{n-1})\Big)
\\ &\qquad   \qquad    +\sum_{K\in {\cal T}}\Big(\mu\int_K|\Grad(\bu^n-\bU^n_h)|^2 \dx+( \mu +\lambda)\int_K|{\rm div}(\bu^n-\bU^n_h)|^2 \dx\Big)
 \\&
 \qquad  \le \sum_{K\in {\cal T}}\Big(\mu\int_K\Grad\bU^n_h:\Grad(\bU^n_h-\bu^n) \dx+(\mu +\lambda)\int_K{\rm div}\bU^n_h{\rm div}(\bU^n_h-\bu^n) \dx\Big)
\\ &\qquad   \qquad  +\sum_{K\in{\cal T}}{|K|}\vr_K^{n-1}\frac{{\bU}_{h,K}^{n}-{\bU}_{h,K}^{n-1}}{\deltat}\cdot \Big(\frac{{\bU}_{h,K}^{n-1} + {\bU}_{h,K}^{n} }2  -
\bu_K^{n-1}\Big)
\\
&\qquad   \qquad   -\sum_{K\in{\cal T}}\sum_{\sigma=K|L\in {\cal E}_K}|\sigma|\vr_\sigma^{n,{\rm up}}\Big(\frac{\bU^n_{h,K}+\bU^n_{h,L}}2-
{\hat \bu}_{\sigma}^{n,{\rm up}}\Big)\cdot{\bU}^n_{h,K} [\bu^n_\sigma\cdot\bn_{\sigma,K}]
\\
&\qquad   \qquad   -\sum_{K\in {\cal T}}\sum_{\sigma=K|L\in {\cal E}_K}|\sigma|p(\vr^n_K)[{\bU}_{h,\sigma}^{n}\cdot\bn_{\sigma,K}]
+\sum_{K\in {\cal T}}\frac{|K|}{\deltat} ( r^n_K-\vr^n_K)\Big(H'( r^n_K)-H'( r^{n-1}_K)\Big)
\\& \qquad   \qquad  +\sum_{K\in {\cal T}}\sum_{\sigma=K|L\in {\cal E}_K}|\sigma|\vr_\sigma^{n,{\rm up}}H'( r_K^{ n-1})[\bu^n_\sigma\cdot\bn_{\sigma,K}].
\end{aligned}
\end{equation}
We obtain formula (\ref{drelativeenergy}) by summing (\ref{drelativeenergy1})$^n$ from $n=1$ to $n=m$ and multiplying
the resulting inequality by $\deltat$.
\end{proof}

\section{Approximate relative energy inequality for the discrete problem}\label{6}
The exact relative energy inequality as stated in Section \ref{5} is a general inequality for the given numerical scheme, however it does not immediately provide a comparison of the { approximate} solution with the strong solution of the compressible Navier-Stokes equations.
Its right hand side has to be conveniently transformed (modulo the possible appearance of residual terms vanishing as the space and time steps tend to $0$) to provide such comparison tool via a Gronwall type argument.

The goal of this section is to derive a version of the discrete relative energy inequality, still with arbitrary (sufficiently regular) test functions $(r,\bU)$, that will be convenient for the comparison of the discrete solution with the strong solution.

\begin{lm}[Approximate relative energy inequality]\label{refrelenergy}\label{6.6}
Suppose that $\Omega\subset \R^3$ is a bounded polyhedral domain and  ${\cal T}$ its regular triangulation { introduced in Section \ref{3.1}}.
Let the pressure $p$ be a $C^2(0,\infty)$ function satisfying hypotheses (\ref{hypp}), (\ref{pressure1}) with
$\gamma\ge 3/2$ and satisfying the additional condition  (\ref{pressure2})  if $\gamma<2$.

Let $(\vr^0,\bu^0) \in L_h^{+}(\Omega) \times \bW_h(\Omega)$ and suppose that $(\vr^n)_{1\le n \le N}\in [L_h^{+}(\Omega)]^N$, $(\bu^n)_{1\le n \le N} \in [\bW_h(\Omega)]^N$ is a solution of the discrete problem \eqref{scheme} with the viscosity coefficients $\mu$, $\lambda$ obeying (\ref{visc}).

Then
there exists
\[
c=c( M_0,E_0,\underline r, \overline r, |p'|_{{C^1}[\underline r,\overline r]}, \|(\partial_t r,\partial_t^2 r, \nabla r, \partial_t\nabla r, \bU,
\partial_t\bU, \nabla\bU, \partial_t\nabla\bU)\|_{L^\infty(Q_T;\Rm^{31})}
)>0
\]
(where  $\overline r=\max_{(t,x)\in \overline {Q_T}} r(t,x)$, $\underline r=\min_{(t,x)\in \overline {Q_T}} r(t,x)$),
such that for all $m=1,\ldots,N$, we have:
\begin{equation}\label{relativeenergy-}
\begin{aligned}
	&{\cal E}(\vr^m,\bu^m\Big| r^m, \bU^m)- {\cal E}(\vr^0,\bu^0\Big|r(0),\bU(0))
	\\& \qquad \qquad \qquad +\deltat \sum_{n=1}^m\sum_{K\in {\cal T}}\Big(\mu\int_K|\Grad(\bu^n-\bU^n_h)|^2 \dx+( \mu +\lambda)\int_K|{\rm div}(\bu^n-\bU^n_h)|^2 \dx\Big)
	\\& \qquad \le \deltat \sum_{n=1}^m\sum_{K\in {\cal T}}\Big(\mu\int_K\Grad\bU^n_h:\Grad(\bU^n_h-\bu^n) \dx+(\mu +\lambda)\int_K{\rm div}\bU^n_h{\rm div}(\bU^n_h-\bu^n) \dx\Big)
	\\& \qquad+\deltat\sum_{n=1}^m\sum_{K\in{\cal T}}|K|\vr_K^{n-1}\frac{{\bU}_{h,K}^{n}-{\bU}_{h,K}^{n-1}}{\deltat}\cdot \Big({\bU}_{h,K}^{n}   -\bu_K^{n}\Big)
	\\ &\qquad { + \deltat\sum_{n=1}^m\sum_{K\in{\cal T}}\sum_{\sigma\in {\cal E}(K)}|\sigma|\vr_\sigma^{n,{\rm up}}
\Big(\hat{\bU}^{n,{\rm up}}_{h,\sigma}-\hat{\bu}^{n,{\rm up}}_{\sigma}\Big)\cdot\Big(\bU^n_\sigma-{\bU}^n_{h,K}\Big) \hat\bU_{h,\sigma}^{n,{\rm up}}\cdot\bn_{\sigma,K}}
%
	\\& \qquad-\deltat \sum_{n=1}^m\sum_{K\in {\cal T}} \int_K p(\vr_K^n)\dv\bU^n\dx
	+\deltat \sum_{n=1}^m\sum_{K\in {\cal T}}\int_K ( r^n_K-\vr^n_K)\frac{p'(r_K^n)}{r_K^n} [\partial_t r]^n\dx
	\\ &\qquad -\deltat\sum_{n=1}^m\sum_{K \in {\cal T}}\int_K\frac{\vr^n_K}{r^n_K} p'(r^n_K) { \bu^n}\cdot\nabla r^n\dx	
\, { + R^m_{h,\deltat}} { + G^m}
\end{aligned}
\end{equation}
for  any pair $(r,\bU)$ belonging to the class (\ref{dr,U}), where
\begin{equation}\label{A1}
{ |G^m|\le c\,\deltat\sum_{n=1}^m{\cal E}(\vr^n,\bu^n\Big| r^n, U^n),}\;\;
{ |R^m_{h,\deltat}|\le c (\sqrt{\deltat} + h^A)},\quad \mbox{ and } A=\left\{\begin{array}{c}
{ \frac{2\gamma-3}{\gamma}}\;\mbox{if $\gamma\in [3/2,2)$}
\\
1/2 \;\mbox{ if { $\gamma\ge 2$}},
\end{array}
\right.
\end{equation}
 { and where we have  used notation (\ref{notation2-}) for $r^n$, $\bU^n$  and (\ref{vhat}--\ref{vtilde}) for $\vc U^n_h$, $\vc U^n_{h,K}$, $r^n_K$, $\vc u^n_{\sigma}$.}
\end{lm}

\begin{proof}
\label{6.0}
The right hand side of the relative energy inequality (\ref{drelativeenergy}) is a sum $\sum_{i=1}^6T_i$, where
\begin{align*}
& T_1=\deltat \sum_{n=1}^m\sum_{K\in {\cal T}}\Big(\mu\int_K\Grad\bU^n_h:\Grad(\bU^n_h-\bu^n) \dx+(\mu +\lambda)\int_K{\rm div}\bU^n_h{\rm div}(\bU^n_h-\bu^n) \dx\Big),
\\ & T_2=\deltat\sum_{n=1}^m\sum_{K\in{\cal T}}{|K|}\vr_K^{n-1}\frac{{\bU}_{h,K}^{n}-{\bU}_{h,K}^{n-1}}{\deltat}\cdot \Big(\frac{{\bU}_{h,K}^{n-1} + {\bU}_{h,K}^{n} }2  -
\bu_K^{n-1}\Big),
\\ & T_3=-\deltat\sum_{n=1}^m\sum_{K\in{\cal T}}\sum_{\sigma=K|L\in {\cal E}(K)}|\sigma|\vr_\sigma^{n,{\rm up}}\Big(\frac{\bU^n_{h,K}+\bU^n_{h,L}}2-
\hat {\bu}_{\sigma}^{n,{\rm up}}\Big)\cdot{\bU}^n_{h,K} [\bu^n_\sigma\cdot\bn_{\sigma,K}],
\\ & T_4 = -\deltat\sum_{n=1}^m\sum_{K\in {\cal T}}\sum_{\sigma=K|L\in {\cal E}(K)}|\sigma|p(\vr_K)[{\bU}_{h,\sigma}^{n}\cdot\bn_{\sigma,K}],
\\ & T_5=
\deltat\sum_{n=1}^m\sum_{K\in {\cal T}}|K| ( r^n_K-\vr^n_K)\frac{H'( r_K^{n})-H'(r^{n-1}_K)}{\deltat},
\\& T_6=\deltat \sum_{n=1}^m \sum_{K\in {\cal T}}\sum_{\sigma=K|L\in {\cal E}(K)}|\sigma|\vr_\sigma^{n,{\rm up}}H'( r_K^{ n-1})[\bu^n_\sigma\cdot\bn_{\sigma,K}].
 \end{align*}

The  term $T_1$ will be kept as it is; all the other terms $T_i$  will be transformed to a more convenient form, as described in the following steps.

\vspace{2mm}\noindent
{\bf Step 1:} {\it Term $T_2$.}
\label{6.1}
We have
\[
T_2=T_{2,1}+ R_{2,1},\mbox{ with }T_{2,1}=\deltat\sum_{n=1}^m\sum_{K\in{\cal T}}|K|\vr_K^{n-1}\frac{{\bU}_{h,K}^{n}-{\bU}_{h,K}^{n-1}}{\deltat}\cdot \Big({\bU}_{h,K
}^{n-1}   -
\bu_K^{n-1}\Big),
\mbox{ and }R_{2,1}=\deltat\sum_{n=1}^m\sum_{K\in{\cal T}}R^{n,K}_{2,1},
\]
where
$$
 R_{2,1}^{n,K}=
\frac {|K|}2
\vr_K^{n-1}\frac{({\bU}_{h,K}^{n}-{\bU}_{h,K}^{n-1})^2}{\deltat}={ \frac {|K|}2
\vr_K^{n-1}\frac{([{\bU}^{n}-{\bU}^{n-1}]_{h,K})^2}{\deltat}}.
$$
We may write by virtue of the first order Taylor formula applied to function $t\mapsto \vc U(t,x)$,
$$
\Big|\frac{[{\bU}^{n}-{\bU}^{n-1}]_{h,K}}{\deltat}\Big|=
\Big|\frac 1{|K|}\int_K\Big[\frac 1\deltat \Big[\int_{t_{n-1}}^{t_n} \partial_t\vc U(z,x) {\rm d} z\Big]_h\Big]{\rm d}x\Big|
$$
$$
=\Big|\frac 1{|K|}\int_K\Big[\frac 1\deltat \int_{t_{n-1}}^{t_n} [\partial_t\vc U(z) \Big]_h(x){\rm d} z\Big]{\rm d}x\Big|\le \|[\partial_t\vc U ]_h\|_{L^\infty(0,T;L^\infty(\Omega;\Rm^3))}\le \|\partial_t\vc U \Big\|_{L^\infty(0,T;L^\infty(\Omega;\Rm^3))},
$$
where we have used the property (\ref{ddd}) of the projection onto the space $V_h{ (\Omega)}$.
Therefore, thanks to the mass conservation \eqref{masscons}, we finally get
\begin{equation}\label{R2.1}
|R^{n,K}_{2,1}|\le\frac {M_0} 2|K|\deltat\|\partial_t\bU\|^2_{L^\infty(0,T; L^{\infty}(\Omega;\Rm^3))}.
\end{equation}

%
Let us now decompose the term  $T_{2,1}$ as
\begin{equation}\label{T2}
T_{2,1}=T_{2,2}+R_{2,2},\mbox{ with }T_{2,2}=\deltat\sum_{n=1}^m\sum_{K\in{\cal T}}|K|\vr_K^{n-1}\frac{{\bU}_{h,K}^{n}-{\bU}_{h,K}^{n-1}}{\deltat}\cdot \Big({\bU}_{h,K
}^{n}   - \bu_K^{n}\Big),
 \mbox{ and } R_{2,2}=\deltat\sum_{n=1}^m R_{2,2}^n,
\end{equation}
where $\displaystyle
	R^n_{2,2}=\sum_{K\in{\cal T}}|K|\vr_K^{n-1}\frac{{\bU}_{h,K}^{n}-{\bU}_{h,K}^{n-1}}{\deltat}\cdot \Big({\bU}_{h,K}^{n-1}   - \bU_{h,K}^{n}\Big)
 	- \sum_{K\in{\cal T}}|K|\vr_K^{n-1}\frac{{\bU}_{h,K}^{n}-{\bU}_{h,K}^{n-1}}{\deltat}\cdot \Big({\bu}_{K}^{n-1}   		 -\bu_{K}^{n}\Big).$

By the same token as above, we may estimate the residual term as follows
\[
|R^n_{2,2}|\le \deltat\; c M_0 \|\partial_t\bU\|^2_{L^\infty(0,T;W^{1,\infty}(\Omega;\Rm^3)}+
c M_0^{1/2}\Big(\sum_{K\in{\cal T}}|K|\vr_K^{n-1}|{\bu}_{K
}^{n-1}   -
\bu_{K}^{n}|^2\Big)^{1/2} \|\partial_t\bU\|_{L^\infty(0,T;L^{\infty}(\Omega;\Rm^3))},
\]
where we have used the H\"older inequality to treat the second term;
whence, by virtue of  estimate (\ref{upwinddissipation_1}),
\begin{equation}\label{R2.2}
|R_{2,2}|\le \sqrt{\deltat} \,c(M_0, E_0, \|(\partial_t\bU, \partial_t\nabla\bU)\|_{L^\infty(Q_T;\Rm^{12})}).
\end{equation}

\vspace{2mm}\noindent
{\bf Step 2:} {\it Term $T_3$.}
\label{6.2}
Employing the definition (\ref{upwind1}) of upwind quantities, we easily establish that
\begin{align*}
& T_3= T_{3,1} + R_{3,1}, \\
& \mbox{with }T_{3,1}= \deltat\sum_{n=1}^m \sum_{K\in{\cal T}}\sum_{\sigma\in {\cal E(K)}}|\sigma|\vr_\sigma^{n,{\rm up}}\Big(\hat {\bu}_{\sigma}^{n,{\rm up}}-
\hat {\bU}^{n, {\rm up}}_{h,\sigma}\Big)\cdot{\bU}^n_{h,K} \bu_\sigma^n\cdot\bn_{\sigma,K}, \quad
R_{3,1}=
\deltat\sum_{n=1}^m { \sum_{\sigma \in {\cal E}_{\rm int}}R_{3,1}^{n,\sigma}}, \\
&\mbox{and }{ R_{3,1}^{n,\sigma}}= |\sigma|\vr_K^n \frac{|\bU_{h,K}^n-\bU_{h,L}^n|^2}2\,[\bu^n_\sigma\cdot\bn_{\sigma,K}]^+
+ |\sigma|\vr_L^n \frac{|\bU_{h,L}^n-\bU_{h,K}^n|^2}2\,[\bu^n_\sigma\cdot\bn_{\sigma,L}]^+, \; \forall \sigma=K|L \in {\cal E}_{\rm int}.
\end{align*}
{ Writing
$$
\bU^n_{h,K}-\bU^n_{h,L}= \bU^n_{h,K}-\bU^n_{h,\sigma}+\bU^n_{h,\sigma}-\bU^n_{h,L},\; \sigma=K|L\in {\cal E}_{\rm int},
$$
}
employing estimates (\ref{L2-1}) and (\ref{L1-3})$_{s=1}$ and the continuity of the mean value $\bU^n_\sigma{ =\bU_{h,\sigma}^n}$ of $\bU^n_h$ over faces $\sigma$, we infer by using the Taylor formula applied to function
$x\mapsto U^n(x)$,
\[
	{ |R_{3,1}^{n,\sigma}|}\le h^2\;c \|\nabla\bU\|^2_{L^\infty (Q_T;\Rm^9)} |\sigma|(\vr^n_K+\vr^n_L) |\bu^n_\sigma|,  \; \forall \sigma=K|L \in {\cal E}_{\rm int},
\]
whence
\begin{equation}\label{R3.1}
\begin{aligned}
  |R_{3,1}| & \le h\;c \|\nabla\bU\|^2_{L^\infty (Q_T;\Rm^9)}\Big(\sum_{K\in{\cal T}}\sum_{\sigma=K|L\in {\cal E}(K)} h|\sigma|(\vr^n_K+\vr^n_L)^{6/5}\Big)^{5/6}
\Big[\deltat \sum_{n=1}^m\Big(\sum_{K\in{\cal T}}\sum_{\sigma\in {\cal E}(K)}h|\sigma||\bu^n_\sigma|^6\Big)^{1/3}\Big]^{1/2}
\\ &\le h \; c(M_0,E_0,\|\nabla\bU\|_{L^\infty(Q_T;\Rm^{9})}),
\end{aligned}
\end{equation}
provided $\gamma\ge 6/5$,
thanks to  the discrete H\"older inequality, the equivalence relation (\ref{reg1}), the equivalence of norms (\ref{norms1}) and energy bounds  listed in Corollary \ref{Corollary1}.

Evidently, for each face $\sigma=K|L\in {\cal E}_{\rm int}$,
$
\bu_{\sigma}^n\cdot\bn_{\sigma,K}+\bu_\sigma^n\cdot{\vc n}_{\sigma,L}=0;$ whence, finally
\begin{equation}\label{T3.1}
T_{3,1}= \deltat\sum_{n=1}^m\sum_{K\in{\cal T}}\sum_{\sigma\in {\cal E}(K)}|\sigma|\vr_\sigma^{n,{\rm up}}\Big(\hat
{\bu}_{\sigma}^{n,{\rm up}}-\hat{\bU}^{n,{\rm up}}_{h,\sigma}\Big)\cdot\Big({\bU}^n_{h,K}-\bU^n_\sigma\Big) \bu_\sigma^n\cdot\bn_{\sigma,K}
\end{equation}
{ Let us now decompose the  term $ T_{3,1}$ as
\begin{equation*}
 \begin{aligned}
&T_{3,1}= T_{3,2}+ R_{3,2}, \mbox{ with } { R_{3,2}=\deltat\sum_{n=1}^mR^{n}_{3,2}}, \\
&T_{3,2}= \deltat\sum_{n=1}^m\sum_{K\in{\cal T}}\sum_{\sigma\in {\cal E}(K)}|\sigma|\vr_\sigma^{n,{\rm up}}
\Big(\hat{\bU}^{n,{\rm up}}_{h,\sigma}-\hat{\bu}^{n,{\rm up}}_{\sigma}\Big)\cdot\Big(\bU^n_\sigma-{\bU}^n_{h,K}\Big) \hat\bu_\sigma^{n,{\rm up}}\cdot\bn_{\sigma,K},   \mbox{ and }\\
&{ R^{n}_{3,2}}=\sum_{K\in{\cal T}}\sum_{\sigma\in {\cal E}(K)} |\sigma|\vr_\sigma^{n,{\rm up}}
\Big(\hat{\bU}^{n,{\rm up}}_{h,\sigma}-\hat{\bu}^{n,{\rm up}}_{\sigma}\Big)\cdot\Big(\bU^n_\sigma-{\bU}^n_{h,K}\Big) \Big(\bu_{\sigma}^n-\hat\bu_\sigma^{n,{\rm up}}\Big)\cdot\bn_{\sigma,K}.
 \end{aligned}
\end{equation*}
By virtue of { discrete} H\"older's inequality and the { first order Taylor formula applied to function $x\mapsto \vc U^n(x)$ in order to evaluate the difference $\bU^n_\sigma-{\bU}^n_{h,K}$}, we get
\begin{equation*}
 \begin{aligned}
	{ |R^{n}_{3,2}|} & \le
	c \|\nabla\vc  U\|_{L^\infty(Q_T;\Rm^9)}\Big(\sum_{K\in{\cal T}}\sum_{\sigma\in {\cal E}(K)}h|\sigma|
\vr_\sigma^{n,{\rm up}}\Big|\hat\bu_\sigma^{n,{\rm up}}-\hat\bU_{h,\sigma}^{n,{\rm up}}\Big|^2\Big)^{1/2}\\
& \times
\Big(\sum_{K\in{\cal T}}\sum_{\sigma\in {\cal E}(K)}h|\sigma|
|\vr_\sigma^{n,{\rm up}}|^{\gamma_0}\Big)^{1/(2{\gamma_0})}\Big(\sum_{K\in{\cal T}}\sum_{\sigma\in {\cal E}(K)}h|\sigma|
\Big|\bu_\sigma^{n}-\hat\bu_\sigma^{n,{\rm up}}\Big|^q\Big)^{1/q},
 \end{aligned}
\end{equation*}
where $\frac 12+\frac 1{2\gamma_0}+\frac 1q=1$, $\gamma_0={\rm min}\{\gamma, 2\}$ and $\gamma\ge 3/2$. For the sum in the last term of the above product, we have
$$
\sum_{K\in{\cal T}}\sum_{\sigma\in {\cal E}(K)}h|\sigma|
\Big|\bu_\sigma^{n}-\hat\bu_\sigma^{n,{\rm up}}\Big|^q\le c \sum_{K\in{\cal T}}\sum_{\sigma\in {\cal E}(K)}h|\sigma|
|\bu_\sigma^{n}-\bu_K^{n}|^q
$$
$$
\le c \Big(\sum_{K\in {\cal T}}\sum_{\sigma\in {\cal E}(K)}
\Big(\|\bu_\sigma^{n}-\bu^n\|_{L^q(K;\Rm^3)}^q+
\sum_{K\in {\cal T}}
\|\bu^n-\bu_K^{n}\|_{L^q(K;\Rm^3)}^q\Big)
\le c h^{\frac {2\gamma_0-3}{2\gamma_0}q}|\bu^n|_{V^2_h(\Omega;\Rm^3)}^q,
$$
where we have used the definition (\ref{upwind1}), the Minkowski inequality and the interpolation inequalities
(\ref{L2+-1}--\ref{L2+-2}). Now we can go back to the estimate of $R_{3,2}^n$ taking into account the upper bounds
(\ref{est0}), (\ref{est3}--\ref{est4}), in order to get
 \begin{equation}\label{R3.2}
 |R_{3,2}|\le h^A\;c (M_0,E_0,\|\nabla\vc  U\|_{L^\infty(Q_T;\Rm^9)})
 \end{equation}
 provided $\gamma\ge 3/2$, where $A$ is given in (\ref{A1}).

 Finally, we rewrite term $T_{3,2}$ as
 \begin{equation}\label{T3}
 \begin{aligned}
&T_{3,2}= T_{3,3}+ R_{3,3}, \mbox{ with } { R_{3,3}=\deltat\sum_{n=1}^mR^{n}_{3,3}}, \\
&T_{3,3}= \deltat\sum_{n=1}^m\sum_{K\in{\cal T}}\sum_{\sigma\in {\cal E}(K)}|\sigma|\vr_\sigma^{n,{\rm up}}
\Big(\hat{\bU}^{n,{\rm up}}_{h,\sigma}-\hat{\bu}^{n,{\rm up}}_{\sigma}\Big)\cdot\Big(\bU^n_\sigma-{\bU}^n_{h,K}\Big) \hat\bU_{h,\sigma}^{n,{\rm up}}\cdot\bn_{\sigma,K},   \mbox{ and }\\
&{ R^{n}_{3,3}}=\sum_{K\in{\cal T}}\sum_{\sigma\in {\cal E}(K)} |\sigma|\vr_\sigma^{n,{\rm up}}
\Big(\hat{\bU}^{n,{\rm up}}_{h,\sigma}-\hat{\bu}^{n,{\rm up}}_{\sigma}\Big)\cdot\Big(\bU^n_\sigma-{\bU}^n_{h,K}\Big) \Big(\hat\bu_\sigma^{n,{\rm up}}-\hat \bU_{h,\sigma}^{n,{\rm up}}\Big)\cdot\bn_{\sigma,K};
 \end{aligned}
\end{equation}
whence
\begin{equation}\label{R3.3}
|R_{3,3}|\le c(\|\nabla\bU\|_{L^\infty(Q_T,\Rm^9)})\; \deltat\sum_{n=1}^m{\cal E}(\vr^n,\vc u^n\,|\, r^n,\bU^n).
\end{equation}

 }

\vspace{2mm}\noindent
{\bf Step 3:} {\it Term $T_4$.}
\label{6.3}
Using the Stokes formula and the property (\ref{L1-1}) in Lemma \ref{Lemma1}, we easily see that
 \begin{equation}\label{T4}
 T_4=-\deltat \sum_{n=1}^m\sum_{K\in {\cal T}}{{\rm \int_K}} p(\vr_K^n)\dv\bU^n\dx.
 \end{equation}

\vspace{2mm}\noindent
{\bf Step 4:} {\it Term $T_5$.}
 \label{6.4}
Using the Taylor formula, we get
\[
H'(r_K^n)-H'(r_K^{n-1})=H''(r_K^{n})(r_K^n-r_K^{n-1}) -\frac 12H'''(\overline r_K^n)(r_K^n-r_K^{n-1})^2,
\]
where $\overline r_K^n\in[\min(r_K^{n-1},r_K^n), \max(r_K^{n-1},r_K^n)]$;
we infer
\begin{equation*}
\begin{aligned}
	& T_5= T_{5,1}+ R_{5,1},\mbox{ with }  T_{5,1}=\deltat \sum_{n=1}^m\sum_{K\in {\cal T}}|K| ( r^n_K-\vr^n_K)\frac{p'(r_K^n)}{r_K^n} \frac{r_K^n-r_K^{n-1}}{\deltat}, \,  R_{5,1}=\deltat \sum_{n=1}^m\sum_{K\in {\cal T}} R_{5,1}^{n,K}, \mbox{ and } \\
	& R_{5,1}^{n,K}= \frac 12|K|H'''(\overline r_K^n)\frac{(r_K^n-r_K^{n-1})^2}{\deltat}(\vr_K^n-r_K^n).
 \end{aligned}
\end{equation*}
Consequently, by the { first order Taylor formula applied to function $t\mapsto r(t,x)$ on the interval $(t_{n-1}, t_n)$} and thanks to the mass conservation \eqref{masscons}
\begin{equation}\label{R5.1}
	|R_{5,1}|\le \deltat \;c(M_0,\underline r,\overline r, |p'|_{C^1([\underline r,\overline r]},\|\partial_t r\|_{L^\infty(Q_T)}),
\end{equation}
where $\underline r$, $\overline r$ are defined in (\ref{dr,U}).

\vspace{2mm}
Let us now decompose  $T_{5,1}$ as follows:
\begin{equation}\label{T5}
\begin{aligned}
	& T_{5,1}=T_{5,2}+ R_{5,2}, \mbox{ with }T_{5,2}=\deltat \sum_{n=1}^m\sum_{K\in {\cal T}}{ \int_K} ( r^n_K-\vr^n_K)\frac{p'(r_K^n)}{r_K^n} [\partial_t r]^n { {\rm d}x}, \,  R_{5,2}=\deltat \sum_{n=1}^m\sum_{K\in {\cal T}} R_{5,2}^{n,K}, \mbox{ and} \\
	& R_{5,2}^{n,K}= {\int_K} ( r^n_K-\vr^n_K)\frac{p'(r_K^n)}{r_K^n}\Big(\frac{r_K^n-r_K^{n-1}}{\deltat} -[\partial_t r]^n\Big){ {\rm d} x}. \end{aligned}
\end{equation}
In accordance with (\ref{notation2-}), here and in the sequel, $[\partial_t r]^n(x)=\partial_t r(t_n,x)$.
{ We write using twice the Taylor formula in the integral form and the Fubini theorem,
$$
|R_{5,2}^{n,K}|= \frac 1\deltat\Big|{p'(r^n_K)}{r^n_K} (\vr^n_K-r^n_K)\int_K\int^{t_n}_{t_{n-1}}\int_s^{t_n}\partial_t ^2 r(z){\rm d}z{\rm d}s {\rm d}x\Big|
$$
$$
\le \frac {p'(r^n_K)}{r^n_K}\int^{t_n}_{t_{n-1}}\int_K|\vr^n_K-r^n_K|\Big|\partial_t ^2 r(z)\Big|{\rm d}x{\rm d}z{\rm d}s
$$
$$
\le \frac{p'(r^n_K)}{r^n_K}
\|\vr^n-\hat r^n\|_{L^{\gamma}(K)}\int^{t_n}_{t_{n-1}}\|\partial_t^2 r(z)\|_{L^{\gamma'}(K)}{\rm d z}{\rm d s}.
$$
Therefore, by virtue of Corollary \ref{Corollary1}, we have estimate
\begin{equation}\label{R5.2}
	|R_{5,2}|\le \deltat\; c(M_0, E_0,\underline r,\overline r, |p'|_{C^1([\underline r,\overline r]},\|\partial^2_t r\|_{L^1(0,T; L^{\gamma'}(\Omega)}).
\end{equation}
}

\vspace{2mm}\noindent
{\bf Step 5:} {\it Term $T_6$.}
%
Using the same argumentation as in formula (\ref{T3.1}), we may write
\begin{equation}
 \begin{aligned}
	  &T_6=T_{6,1} + R_{6,1},\quad  R_{6,1}=\deltat \sum_{n=1}^m \sum_{K\in {\cal T}}\sum_{\sigma\in {\cal E}(K)}R_{6,1}^{n,\sigma,K}, \mbox{ with} \\
	&T_{6,1}=\deltat\sum_{n=1}^m\sum_{K\in {\cal T}}\sum_{\sigma=K|L\in {\cal E}(K)}|\sigma|\vr_K^{n}\Big( H'( r_K^{ {n-1}})-H'(r_\sigma^{ {n-1}})\Big)\bu_\sigma^n\cdot\bn_{\sigma,K}, \mbox{ and} \\
	&  R_{6,1}^{n,\sigma,K}=|\sigma|\Big(\vr_\sigma^{n,{\rm up}}-\vr_K^{n}\Big)\Big( H'( r_K^{ {n-1}})-H'(r_\sigma^{ {n-1}})\Big)\bu_\sigma^n\cdot\bn_{\sigma,K}, { \mbox{ for } \sigma=K|L.}
 \end{aligned}
\end{equation}
We   estimate this term separately for $\gamma\le 2$ and $\gamma>2$.
If $\gamma\le 2$, motivated by Lemma \ref{Lemma5}, we may write
\begin{multline}\label{R6.1a}
	 |R_{6,1} ^{n,\sigma,K}| \le \sqrt h\, \|\nabla H'(r)\|_{L^\infty(Q_T;\Rm^3)} |\sigma| \\
	\times \Big( \frac { |\vr_\sigma^{n,{\rm up}} - \vr_K^{n} |} {\max(\vr_K,\vr_L)^{(2-\gamma)/2}} \,\sqrt{|\bu_\sigma^n\cdot\bn_{\sigma,K}}|  1_{\overline\vr_\sigma^n\ge1} \sqrt h(\vr^n_K+\vr^n_L)^{(2-\gamma)/2} \sqrt{|\bu_\sigma^n\cdot\bn_{\sigma,K}}|
	\\ + |\vr_\sigma^{n,{\rm up}}-\vr_K^{n}|\,\sqrt{|\bu_\sigma^n\cdot\bn_{\sigma,K}|} 1_{\overline\vr_\sigma^n<1}
\sqrt h \sqrt{|\bu_\sigma^n\cdot\bn_{\sigma,K}|} \Big),
\end{multline}
where we again use the { first order Taylor formula applied to function $H'$ between endpoints $r_K^{n-1}$,
$r_\sigma^{n-1}$,}  and where the numbers $\overline\vr_\sigma^n$ are defined in Lemma \ref{Theorem3}.
Consequently, an application of the H\"older and Young inequalities yields
\begin{equation}\label{R6.1b}
 \begin{aligned}
 |R_{6,1}|
	& \le \sqrt h\,  { c} \|\nabla H'(r)\|_{L^\infty(Q_T;\Rm^3)}
		\deltat \sum_{n=1}^m\Big[\Big(\sum_{K\in {\cal T}}\sum_{\sigma=K|L\in {\cal E}(K)}|\sigma| \frac {(\vr_\sigma^{n,{\rm up}}-\vr_K^{n})^2}{\max(\vr_K,\vr_L)^{(2-\gamma)}}\,|\bu_\sigma^n\cdot\bn_{\sigma,K}| 1_{\overline\vr_\sigma^n\ge1}\Big)^{1/2}
		\\ & \hspace{6cm} \times \Big(\sum_{K\in {\cal T}}\sum_{\sigma\in {\cal E}(K)}|\sigma| h\vr_K^{2-\gamma} \,|\bu_\sigma^n\cdot\bn_{\sigma,K}|\Big)^{1/2}
		\\  & \quad  + \Big(\sum_{K\in {\cal T}}\sum_{\sigma=K|L\in {\cal E}(K)}|\sigma|h (\vr_\sigma^{n,{\rm up}}-\vr_K^{n})^2\,|\bu_\sigma^n\cdot\bn_{\sigma,K}| 1_{\overline\vr_\sigma^n<1}\Big)^{1/2}
 		 \Big(\sum_{K\in {\cal T}}\sum_{\sigma\in {\cal E}(K)}|\sigma| h \,|\bu_\sigma^n\cdot\bn_{\sigma,K}|\Big)^{1/2}\Big]
	\\ & \le \sqrt h\, { c}   \|\nabla H'(r)\|_{L^\infty(Q_T;\Rm^3)} \deltat \sum_{n=1}^m\Big[\Big(\sum_{K\in {\cal T}}\sum_{\sigma=K|L\in 	{\cal E}(K)}|\sigma| \frac {(\vr_\sigma^{n,{\rm up}}-\vr_K^{n})^2} {\max(\vr_K,\vr_L)^{(2-\gamma)}}  \,  |\bu_\sigma^n \cdot \bn_{\sigma,K}| 1_{\overline\vr_\sigma^n\ge1}
		\\  & \qquad \qquad + \Big(\sum_{K\in {\cal T}} |K|\vr_K^{6(2-\gamma)/5}\Big)^{5/6} \Big(\sum_{\sigma\in {\cal E}}|\sigma| h|\bu_\sigma^n|^6\Big)^{1/6}
		\\  & \qquad \qquad + \sum_{K\in {\cal T}}\sum_{\sigma=K|L\in {\cal E}(K)}|\sigma|h (\vr_\sigma^{n,{\rm up}}-\vr_K^{n})^2\,|\bu^n_\sigma\cdot\bn_{\sigma,K}| 1_{\overline\vr_K^n<1}+ |\Omega|^{5/6} \Big(\sum_{\sigma\in {\cal E}}|\sigma| h|\bu_\sigma^n|^6\Big)^{1/6}\Big]
\end{aligned}
\end{equation}
We deduce employing the discrete H\"older inequality
{
$$
\deltat \sum_{n=1}^m\Big(\sum_{K\in {\cal T}} |K|\vr_K^{6(2-\gamma)/5}\Big)^{5/6} \Big(\sum_{\sigma\in {\cal E}}|\sigma| h|\bu_\sigma^n|^6\Big)^{1/6}
$$
$$
\le \Big[\deltat \sum_{n=1}^m \Big(\sum_{K\in {\cal T}} |K|\vr_K^{6(2-\gamma)/5}\Big)^{5/3}   \Big]^{1/2}\Big[\deltat \sum_{n=1}^m\Big(\sum_{\sigma\in {\cal E}}|\sigma| h|\bu_\sigma^n|^6\Big)^{1/3}\Big]^{1/2},
$$
and
$$
\deltat\sum_{n=1}^m\Big(\sum_{\sigma\in {\cal E}}|\sigma| h|\bu_\sigma^n|^6\Big)^{1/6}\le \sqrt T
\Big[k\sum_{n=1}^m\Big(\sum_{\sigma\in {\cal E}}|\sigma| h|\bu_\sigma^n|^6\Big)^{1/3}\Big]^{1/2}
$$
}
Coming back to (\ref{R6.1b}) we deduce that
\begin{equation}\label{R6.1}
|R_{6,1}|
\le \sqrt h \; c(M_0,E_0,\underline r,\overline r, |p'|_{C([\underline r,\overline r])}, \|\nabla r\|_{L^\infty(Q_T;\Rm^3)})
\end{equation}
provided $\gamma\ge 12/11$, where we use estimate (\ref{dissipative1}), estimates (\ref{est1}), (\ref{est3}) of Corollary \ref{Corollary1} and equivalence relation (\ref{norms1}).
In the case $\gamma>2$,    the same final bound may be obtained by a similar argument, replacing the  estimate (\ref{dissipative1})  by (\ref{dissipative2}).

\vspace{2mm}
Let us now decompose the term $T_{6,1}$ as
{
\begin{equation*}
\begin{aligned}
	&T_{6,1}=T_{6,2}+ R_{6,2},  \mbox{ with }  T_{6,2}=\deltat\sum_{n=1}^m\sum_{K\in{\cal T}}\sum_{\sigma=K|L\in {\cal E}(K)}
|\sigma|\vr^n_K H''(r_K^{ n-1})(r_K^{ n-1}-r_\sigma^{ n-1}) [\bu^n_\sigma\cdot\bn_{\sigma,K}],
 \\    &  R_{6,2}=\deltat\sum_{n=1}^m\sum_{K\in{\cal K}}\sum_{\sigma\in {\cal E}(K) }
R_{6,2}^{n,\sigma,K},\;
\mbox{and }\\ & R_{6,2}^{n,\sigma,K}=|\sigma|\vr^n_K \Big(H'(r^{n-1}_K)-H'(r^{n-1}_\sigma)-H''(r^{n-1}_K)(r^{n-1}_K-r^{n-1}_\sigma)\Big)
[\bu^n_\sigma\cdot\bn_{\sigma,K}]
 \end{aligned}
\end{equation*}
}
Therefore, by virtue of the { second order Taylor formula applied to function H'}, H\"older's inequality, (\ref{sob1}), (\ref{norms1}),
and (\ref{est0}), (\ref{est3}) in Corollary \ref{Corollary1}, we have, provided
$\gamma\ge 6/5$,
{
\begin{align}
| R_{6,2}| &\le h c \Big(|H''|_{C([\underline r,\overline r])}+|H'''|_{C([\underline r,\overline r])}\Big)\|\nabla r\|_{L^\infty(Q_T;\Rm^3)} \|\vr\|_{L^\infty(0,T;L^\gamma(\Omega))} \|\bu\|_{L^2(0,T;V_h^2(\Omega;\Rm^3))}
\nonumber \\
\label{R6.2}
& \le  h\;
c(M_0,E_0,\underline r, \overline r, |p'|_{C^1([\underline r,\overline r])},  \|\nabla r\|_{L^\infty(Q_T;\Rm^3)} ),
\end{align}
}
where in the first line we have used notation (\ref{notation1}).

\vspace{2mm}

Let us now deal with the term  $T_{6,2}$.
Noting that $ \displaystyle
\int_K\nabla r^{ n-1} \dx =   \sum_{\sigma\in {\cal E}(K)} |\sigma|(r_\sigma^{ n-1}-r_K^{ n-1})\bn_{\sigma,K},$ we may write
\begin{multline*}
\sum_{\sigma\in {\cal E}(K)} |\sigma|\vr^n_K H''(r_K^{ n-1})(r_K^{ n-1}-r_\sigma^{ n-1}) [\bu^n_\sigma\cdot\bn_{\sigma,K}]
\\= -\int_K\vr^n_K H''(r_K^{ n-1}) \bu^n_K\cdot\nabla r^{ n-1}\dx +
\sum_{\sigma\in {\cal E}(K)} |\sigma|\vr^n_K H''(r_K^{ n-1})(r_K^{ n-1}-r_\sigma^{ n-1}) (\bu^n_\sigma-\bu^n_K)\cdot\bn_{\sigma,K}.
\end{multline*}
Consequently, $T_{6,2}= T_{6,3}+ R_{6,3},$ with
$$
\begin{aligned}
& T_{6,3}=
 -\deltat\sum_{n=1}^m\sum_{K \in {\cal T}}\int_K{\vr^n_K} H''(r_K^{ n-1}) \bu^n \cdot\nabla r^{ n-1}\dx, \\
&  R_{6,3} = \deltat\sum_{n=1}^m\sum_{K \in {\cal T}} \int_K\vr^n_K H''(r_K^{ n-1})(\bu^n- \bu^n_K)\cdot\nabla r^{ n-1}\dx
 \\ & \hspace{3cm}+\deltat\sum_{n=1}^m\sum_{K \in {\cal T}}\sum_{\sigma\in {\cal E}(K)} |\sigma|\vr^n_K H''(r_K^{ n-1})(r_K^{ n-1}-r_\sigma^{ n-1})
(\bu^n_\sigma-\bu^n_K)\cdot\bn_{\sigma,K},
\end{aligned}
$$
{
$$
|R_{6,3}|\le c\,\|H''(r)\nabla r\|_{L^\infty(Q_T;\Rm^3)}\Big[ \deltat\sum_{n=1}^m\sum_{K \in {\cal T}}\|\vr^n_K\|_{L^{\gamma_0}(K)}\|\bu^n- \bu^n_K\|_{L^{\gamma_0'}(K;\Rm^3)}+
$$
$$
\deltat\sum_{n=1}^m\sum_{K \in {\cal T}}\sum_{\sigma\in {\cal E}(K)} \|\vr^n_K\|_{L^{\gamma_0}(K)}\|\bu_\sigma^n- \bu^n_K\|_{L^{\gamma_0'}(K;\Rm^3)}\Big],\quad \gamma_0=\min\{\gamma,2\},
$$
where we have used the H\"older inequality, and also the Taylor formula applied to function $x\mapsto r(t_{n-1},x)$
together with equivalence relation (\ref{reg1}) yielding $|\sigma|h\le |K|$, to treat the second term.
Consequently, by virtue of H\"older's inequality, interpolation inequality (\ref{interpol1}) (to estimate
$\|\vc u^n-\vc u^n_K\|_{L^{\gamma_0'}(K;\Rm^3)}$ by $h^{(5\gamma_0-6)/(2\gamma_0)}\|\Grad\vc u^n \|_{L^2(K;\Rm^9)}$, $\gamma_0=\min\{\gamma,2\}$)
in the first term, and by the the H\"older inequality
 and (\ref{interpol1}--\ref{interpol2}) (to estimate $\|\vc u_\sigma^n-\vc u^n_K\|_{L^{\gamma'}(K;\Rm^3)}$ by $h^{(5\gamma_0-6)/(2\gamma_0)}\|\Grad\vc u^n \|_{L^2(K;\Rm^9)}$) in the second term, we get
$$
|R_{6,3}|\le c\,h^{(5\gamma_0-6)/(2\gamma_0)}\,\|H''(r)\nabla r\|_{L^\infty(Q_T;\Rm^3)} \deltat\sum_{n=1}^m\Big(\sum_{K \in {\cal T}}\|\vr^n_K\|^2_{L^{\gamma_0}(K)}\Big)^{1/2}\Big(\sum_{K \in {\cal T}}\|\Grad\bu^n\|^2_{L^{2}(K;\Rm^9)}\Big)^{1/2}
$$
$$
\le
c\,h^{(5\gamma_0-6)/(2\gamma_0)}\,\|H''(r)\nabla r\|_{L^\infty(Q_T;\Rm^3)} \deltat\sum_{n=1}^m\Big(\sum_{K \in {\cal T}}\|\vr^n_K\|^{\gamma_0}_{L^{\gamma_0}(K)}\Big)^{1/\gamma_0}\Big(\sum_{K \in {\cal T}}\|\Grad\bu^n\|^2_{L^{2}(K;\Rm^9)}\Big)^{1/2},
$$
provided $\gamma\ge 6/5$,
where we have used the discrete H\"older inequality and  the algebraic inequality (\ref{dod1*}).
}
Now it remains to use (\ref{est0}), (\ref{est3}) in Corollary \ref{Corollary1} in order to get
\begin{equation}\label{R6.3}
|R_{6,3}|\le h^A \;c(M_0,E_0,\underline r, \overline r, |p'|_{C^1([\underline r,\overline r])} \|\nabla r\|_{L^\infty(Q_T;\Rm^3)} ),
\end{equation}
where $A$ is defined in \eqref{A1}.

{
Finally we write
 $T_{6,3}= T_{6,4}+ R_{6,4},$ with
\begin{equation}\label{T6}
\begin{aligned}
& T_{6,4}=
 -\deltat\sum_{n=1}^m\sum_{K \in {\cal T}}\int_K{\vr^n_K}\frac{ p'(r_K^{n})}{r_K^n} \bu^n\cdot\nabla r^{ n}\dx, \\
&  R_{6,4} = \deltat\sum_{n=1}^m\sum_{K \in {\cal T}} \int_K\vr^n_K \Big(H''(r_K^{ n})\nabla r^{n}-
H''(r_K^{n-1})\nabla r^{n-1}\Big)\cdot\bu^n\dx,
\end{aligned}
\end{equation}
{ where by the same token as in (\ref{R5.2}),
\begin{equation}\label{R6.4}
|R_{6,4}|\le \deltat\; c (M_0,E_0, \underline r,\overline r, |p'|_{C^1([\underline r,\overline r])}, \|\nabla r, \partial_t r\|_{L^\infty(Q_T;\Rm^4)}, \|\partial_t\nabla r\|_{L^2(0,T; L^{{6\gamma}/{(5\gamma-6)}}(\Omega;\Rm^3))}).
\end{equation}
}
}

\vspace{2mm}
We are now in position to conclude the proof of Lemma \ref{6.6}: we obtain the inequality \eqref{relativeenergy-} by gathering the principal terms (\ref{T2}), (\ref{T3}), (\ref{T4}), (\ref{T5}), (\ref{T6}) and  the residual terms estimated in   (\ref{R2.1}), (\ref{R2.2}), (\ref{R3.1}), (\ref{R3.2}), (\ref{R3.3}), (\ref{R5.1}), (\ref{R5.2}), (\ref{R6.1a}), \eqref{R6.1b}, (\ref{R6.2}), (\ref{R6.3}), { (\ref{R6.4}) at}  the right hand side $\sum_{i=1}^6 T_i$  of the discrete relative energy inequality (\ref{drelativeenergy}).

\end{proof}

\section{ A discrete identity satisfied by  the strong solution}\label{7}
{ This section is devoted to the proof of a discrete identity satisfied by any strong solution. This identity
is stated in Lemma \ref{strongentropy} below. It will be used in combination with the approximate relative energy inequality stated in Lemma \ref{refrelenergy} to deduce the convenient form of the relative energy inequality verified by any function being a strong solution to the compressible Navier-Stokes system. This last step is performed in the next section.}

\begin{lm}[A discrete identity for strong solutions]\label{strongentropy}
Suppose that $\Omega\subset \R^3$ is a bounded polyhedral domain and  ${\cal T}$ a regular triangulation { introduced in Section \ref{3.1}}.
Let the pressure $p$ be a $C^2(0,\infty)$ function satisfying hypotheses (\ref{hypp}) and (\ref{pressure1}) with
$\gamma\ge 3/2$. Let $(r,\bU)$ belong to the class (\ref{dr,U}) satisfy equation \eqref{pbcont} with the viscosity coefficients $\mu$, $\lambda$ obeying (\ref{visc}).

Let $(\vr^0,\bu^0) \in L_h^{+}(\Omega) \times \bW_h(\Omega)$ and suppose that $(\vr^n)_{1\le n \le N}\in [L_h^{+}(\Omega)]^N$, $(\bu^n)_{1\le n \le N} \in [\bW_h(\Omega)]^N$ is a solution of the discrete problem \eqref{scheme}.
Then there exists
\begin{align*}
& c=c\Big(M_0, E_0, \underline r,\overline r, |p'|_{C^1([\underline r,\overline r])},
\|(\nabla r, \partial_t r, \bU, \nabla\bU, { \nabla^2\bU}, \partial_t\bU, \partial_t^2\bU,\partial_t\nabla\bU,)\|_{L^\infty (Q_T;\Rm^{58})})\Big)> 0,
\end{align*}
such that for any $m=1,\ldots,N$, the following identity holds:
\begin{equation}\label{strong1}
\begin{aligned}
	&\deltat\sum_{n=1}^m\sum_{K\in {\cal T}}\int_K\Big(\mu\nabla\bU_h^n\cdot\nabla(\bu^n-\bU^n_{ h})+ (\mu +\lambda)\dv \bU^n_{ h}\dv (\bu^n-\bU^n_{ h})\Big)\dx
\\
	&\qquad +\deltat\sum_{n=1}^m\sum_{K\in {\cal T}}\int_Kr_K^{n-1}\frac{\bU^n_{ h,K}-\bU^{n-1}_{ h,K}}{\deltat}\cdot (\bu_K^{ n}-\bU_{h,K}^{ n})\dx
\\
	 &\qquad+ { \deltat\sum_{n=1}^m\sum_{K\in{\cal T}}\sum_{\sigma\in {\cal E}(K)}|\sigma| \hat r_\sigma^{n,{\rm up}}
[\hat \bU_{h,\sigma}^{n,{\rm up}}\cdot{\vc n}_{\sigma,K}](\bU_\sigma^n-\bU^n_{h,K})\cdot(\hat\bu_\sigma^{n,{\rm up}}-\hat\bU_{h,\sigma}^{n,{\rm up}})}
\\
	&\qquad +\deltat \sum_{n=1}^m \sum_{K\in {\cal T}}\int_K p(r_K^n)\dv  \bU^n\dx
+\deltat \sum_{n=1}^m \sum_{K\in{\cal T}}\int_K p'(r^n_K)\bu^n\cdot\nabla r^n\dx
{ + {\cal R}^m_{h,\deltat} =0},
\end{aligned}
\end{equation}
{ where
\[
|{\cal R}_{h, \deltat}^m|\le c\Big( h + \deltat\Big)
\]
 { and where we have  used notation (\ref{notation2-}) for $r^n$, $\bU^n$  and (\ref{vhat}--\ref{vtilde}) for $\vc U^n_h$, $\vc U^n_{h,K}$, $r^n_K$, $\vc u^n_{\sigma}$.}
}
%
\end{lm}
{ Before starting the proof  we recall an auxiliary algebraic inequality whose straightforward proof is left to the reader, and introduce some notations.}
\begin{lm}\label{LL1}
	Let $p$ satisfies assumptions (\ref{hypp}) and (\ref{pressure1}).
	Let $0<a<b<\infty$. Then there exists $c=c(a,b)>0$ such that for all $\vr\in [0,\infty)$ and $r\in [a,b]$ there holds
	\[
		E(\vr|r)\ge c(a,b)\Big(1_{R_+\setminus [a/2,2 b]}{ (\vr)}+\vr^\gamma 1_{R_+\setminus [a/2,2 b]}{ (\vr)}+ (\vr-r)^2 1_{[a/2,2 b]}{ (\vr)}\Big),
	\]
	where $E(\vr|r)$ is defined in (\ref{E}).
\end{lm}
{ If we take in Lemma \ref{LL1} $\vr=\vr^n(x)$, $\vr^n\in L^+_h(\Omega)$, $r=\hat r^n(x)$, $a=\underline r$, $b=\overline r$ (where r is a function belonging to class (\ref{dr,U}) and $\underline r$, $\overline r$
are its lower and upper bounds, respectively), we obtain
\begin{equation}\label{added}
E(\vr^n(x)|\hat r^n(x))\ge c(\underline r,\overline r)\Big(1_{R_+\setminus [\underline r/2,2 \overline r]}{ (\vr^n(x))}+(\vr^n)^\gamma(x) 1_{R_+\setminus [\underline r/2,2 \overline r]}{ (\vr^n(x))}+ (\vr^n(x)-\hat r^n(x))^2 1_{[\underline r/2,2 \overline r]}{ (\vr^n(x))}\Big)
\end{equation}
}
 Now,  for fixed numbers $\underline r$ and $\overline r$  { and  fixed functions $\vr^n\in L^+_h(\Omega)$, $n=0,\ldots, N$,  we introduce the residual and essential subsets  of $\Omega$ (relative to $\vr^n$) as follows:}
\begin{equation}\label{essres}
N_{\rm ess}^n=\{x\in\Omega\,\Big|\,\frac 12\underline r\le \vr^n(x)\le 2\overline r\},\;
N_{\rm res}^n= \Omega\setminus N_{\rm ess}^n.
\end{equation}
and we set
\[
[g]_{\rm ess}{ (x)}= g{ (x)} 1_{N^n_{\rm ess}}{ (x)},\; [g]_{\rm res}{ (x)}= g { (x)}1_{N^n_{\rm res}}{ (x)},\;\; { x\in \Omega},\;\;g\in L^1(\Omega).
\]

 Integrating inequality (\ref{added}) we deduce
\begin{equation}\label{rentropy}
c(\underline r,\overline r)\sum_{K\in{\cal T}} \int_K{\Big(\Big[1\Big]_{\rm res}+\Big[(\vr^n)^\gamma\Big]_{\rm res}+\Big[\vr^n-\hat r^n\Big]^2_{\rm ess}\Big)}\dx\le{\cal E}(\vr^n,\bu^n\Big| r^n,\bU^n).
\end{equation}
for any  pair $(r,\bU)$ belonging to the class (\ref{dr,U}) and any $\vr^n\in L_h(\Omega)$.

{ We are now ready to proceed to the proof of Lemma \ref{strongentropy}.}

\begin{proof}

We start by projecting the momentum equation to the discrete spaces.\label{7.2}
Since $(r,\bU)$ satisfies  \eqref{pbcont} and belongs to the class (\ref{dr,U}), Equation (\ref{mov2}) can be rewritten in the form
\begin{equation}\label{str1}
r\partial_t\bU+r\bU\cdot\nabla\bU +\nabla p(r)=\mu\Delta\bU +(\mu+\lambda)\nabla\dv\bU.
\end{equation}
We write equation (\ref{str1}) at $t=t_n$, multiply scalarly by $\bu^n-\bU^n_{h}$, and integrate over $\Omega$.
We get, after summation from $n=1$ to $m$,
\begin{equation}\label{strong0}
\begin{aligned}
& \sum_{i=1}^5{\cal T}_i=0, \qquad\qquad\qquad \mbox{ with }  && {\cal T}_1 = -\deltat\sum_{n=1}^m\int_\Omega\Big(\mu\Delta \bU^n+ (\mu+\lambda)
\nabla\dv\bU^n\Big)\cdot(\bu^n-\bU^n_h)\dx, &&\\
&{\cal T}_2 =\deltat\sum_{n=1}^m\int_\Omega r^{n}[\partial_t\bU]^n\cdot (\bu^n-\bU^n_h)\dx, && {\cal T}_3 = \deltat \sum_{n=1}^m\int_\Omega r^n\bU^n\cdot\nabla\bU^n\cdot (\bu^n-\bU^n_h)\dx \\
&{\cal T}_4 = \deltat\sum_{n=1}^m\int_\Omega\nabla p(r^n)\cdot\bu^n \dx, && {\cal T}_5 =-\deltat\sum_{n=1}^m\int_\Omega\nabla p(r^n)\cdot\bU^n_h\dx.
 \end{aligned}
\end{equation}

 In the steps below, we deal with each of the terms ${\cal T}_i$.

\vspace{2mm}
\textbf{Step 1: }\textit{Term ${\cal T}_1$.}\label{7.3}
 Integrating by parts, we get:
 {
\begin{equation}\label{cT1}
 \begin{aligned}
&{\cal T}_1={\cal T}_{1,1} + {\cal R}_{1,1},
\\
& \mbox{with }{\cal T}_{1,1} = \deltat\sum_{n=1}^m\sum_{K\in {\cal T}}\int_K\Big(\mu\nabla\bU_{h}^n:\nabla(\bu^n-\bU_h^n)+
(\mu +\lambda)\dv \bU_{ h}^n\dv (\bu^n-\bU_h^n)\Big)\dx
 \\
 &\mbox{and }\;
 {\cal R}_{1,1}=I_1+I_2,\;\mbox{with}\\
 &I_1=
 \deltat\sum_{n=1}^m\sum_{K\in {\cal T}}\int_K\Big(\mu\nabla(\bU^n-\bU_{ h}^n):\nabla(\bu^n-\bU_h^n)+ (\mu+\lambda)
 \dv((\bU^n-\bU_{ h}^n)\dv(\bu^n-\bU_h^n)\Big)
 \dx
 \\
 &
 I_2=-\deltat\sum_{n=1}^m \sum_{K\in{\cal T}}\sum_{\sigma\in {\cal E}(K)} \int_{\sigma} \Big(\mu\bn_{\sigma,K} \cdot\nabla\bU^n \cdot(\bu^n- \bU_h^n) + (\lambda+\mu)\dv\bU^n(\bu^n-\bU^n_h)\cdot\bn_{\sigma,K}\Big)\dS
 \\
 & =
 -\deltat\sum_{n=1}^m\ \sum_{\sigma\in {\cal E}} \int_{\sigma} \Big(\mu\bn_{\sigma} \cdot\nabla\bU^n \cdot\Big[\bu^n- \bU_h^n\Big]_{\sigma,\vc n_\sigma} + (\lambda+\mu)\dv\bU^n\Big[\bu^n-\bU^n_h\Big]_{\sigma,\vc n_\sigma}\cdot\bn_{\sigma}\Big)\dS,
\end{aligned}
\end{equation}
}
where in the last line $\vc n_\sigma$ is a unit  normal to $\sigma$ and $[\cdot]_{\sigma,\vc n_\sigma}$ is the jump over sigma (with respect to $\vc n_\sigma$) defined in Lemma \ref{Lemma6}. Employing estimate (\ref{L1-3})
we easily verify that
{
$$
|I_1|\le h\;c(M_0,E_0, \|\nabla\vc U\|_{L^\infty(0,T; W^{1,\infty}(\Omega))}).
$$
Since the integral over any face $\sigma$ of the jump of a function from $V_h(\Omega)$ is zero, we may write
$$
I_2=-\deltat\sum_{n=1}^m\ \sum_{\sigma\in {\cal E}} \int_{\sigma} \Big(\mu\bn_{\sigma} \cdot\Big(\nabla\bU^n -[\nabla\bU^n]_\sigma\Big)\cdot\Big[\bu^n- \bU_h^n\Big]_{\sigma,\vc n_\sigma}
$$
$$
+ (\lambda+\mu)\Big(\dv\bU^n-[\dv\bU^n]_\sigma\Big)\Big[\bu^n-\bU^n_h\Big]_{\sigma,\vc n_\sigma}\cdot\bn_{\sigma}\Big)\dS;
$$
whence by using the { first order Taylor formula applied to functions $x\mapsto \nabla \vc U^n(x)$ to evaluate the differences $\nabla\bU^n -[\nabla\bU^n]_\sigma$, $\dv\bU^n-[\dv\bU^n]_\sigma$, }
   and H\"older's inequality,
$$
\begin{aligned}
&|I_2| \le \deltat\, h\; c\,  \|\nabla^2\bU\|_{L^\infty(Q_T;\Rm^{27})} \sum_{n=1}^m\sum_{\sigma\in {\cal E}} \sqrt{|\sigma|}\sqrt h\;\Big(\frac 1{\sqrt h}\,\Big\|\Big[\bu^n-\bU_h^n\Big]_{\sigma,\vc n_\sigma}\Big\|_{L^2(\sigma;\Rm^3)}\Big)\\
&\le \deltat\, h\; c\,  \|\nabla^2\bU\|_{L^\infty(Q_T;\Rm^{27})} \sum_{n=1}^m\sum_{\sigma\in {\cal E}}\Big(|\sigma|h+
\frac 1h\,\Big\|\Big[\bu^n-\bU_h^n\Big]_{\sigma,\vc n_\sigma}\Big\|_{L^2(\sigma;\Rm^3)}^2\Big).
\end{aligned}
$$
Therefore,
\begin{equation}\label{cR1.1}
|{\cal R}_{1,1}|\le  h\, c(M_0, E_0,\|\vc U, \nabla\bU, \nabla^2\bU\|_{L^\infty(Q_T,\Rm^{39})}),
\end{equation}
where we have employed Lemma \ref{Lemma6} and
 (\ref{est0}) in Corollary \ref{Corollary1}.
}

\vspace{2mm}
\textbf{Step 2:}\textit{  Term  ${\cal T}_2$.}\label{7.4}
Let us now decompose the term  ${\cal T}_2$ as
\begin{align*}
	 & {\cal T}_2={\cal T}_{2,1}+ {\cal R}_{2,1},\\
	& \mbox{with } {\cal T}_{2,1}=\deltat\sum_{n=1}^m\sum_{K\in {\cal T}}\int_Kr^{n-1}\frac{\bU^n-\bU^{n-1}}{\deltat}\cdot (\bu^n-\bU^n_h)\dx,\quad {\cal R}_{2,1}=\deltat\sum_{n=1}^m\sum_{K\in {\cal T}}{\cal R}_{2,1}^{n,K}, \\
	& \mbox{and }{ {\cal R}_{2,1}^{n,K}=\int_K(r^n-r^{n-1})[\partial_t\vc U]^n\cdot(\bu^n-\bU^n_h)\dx }+ \int_Kr^{n-1}\Big([\partial_t \bU]^n-\frac{\bU^n-\bU^{n-1}}{\deltat}\Big) \cdot(\bu^n-\bU^n_h)\dx.
\end{align*}
The remainder ${\cal R}_{2,1}^{n,K}$ can be rewritten as follows
$$
{\cal R}_{2,1}^{n,K}=\int_K\Big[\int_{t_{n-1}}^{t_n}r(t,\cdot){\rm d}t\Big][\partial_t\vc U]^n\cdot(\bu^n-\bU^n_{h})\dx
 $$
 $$
 + \frac 1k\int_Kr^{n-1}\Big[\int_{t_{n-1}}^{t_n}\int_s^{t_n}\partial^2_t \bU(z,\cdot){\rm d}z{\rm d}s\Big] \cdot(\bu^n-\bU_{h}^n)\dx;
$$
{ whence,
\[
|{\cal R}_{2,1}^{n,K}|\le \deltat\Big[ (\|r\|_{L^\infty(Q_T)}+\|\partial_t r\|_{L^\infty(Q_T)}) (\|\partial_t\bU\|_{L^\infty(Q_T;\Rm^3)}|K|^{5/6}(\|\bu^n\|_{L^6(K)}+ \|\bU^n_h\|_{L^6(K)})
 \]
 \[
 +\|\partial^2_t \bU^n\|_{L^{6/5}(\Omega;\Rm^3))} (\|\bu^n\|_{L^6(K)}+ \|\bU^n_h\|_{L^6(K)}).
\]
Consequently, by the same token as in (\ref{R5.2}) or (\ref{R6.4}),
\begin{equation}\label{cR2.1}
|{\cal R}_{2,1}|\le \deltat\, c\Big(M_0, E_0,\overline r,\|(\partial_t r, \bU, \partial_t\bU, \nabla\bU , )\|_{L^\infty(Q_T;\Rm^{16})}, \|\partial^2_t \bU\|_{L^2(0,T; L^{6/5}(\Omega;\Rm^3))}\Big),
\end{equation}
 where we have used the H\"older and Young inequalities,  the estimates (\ref{L1-2}), (\ref{L1-4}), (\ref{L1-6}), (\ref{sob1}), and to the energy bound (\ref{est0}) from Corollary \ref{Corollary1}.
 }

\vspace{2mm}
{\bf Step 2a:} {\it Term ${\cal T}_{2,1}$.} We decompose the term ${\cal T}_{2,1}$ as
\begin{align*}
&{\cal T}_{2,1}={\cal T}_{2,2}+ {\cal R}_{2,2}, \\
&\mbox{with } {\cal T}_{2,2}=\deltat\sum_{n=1}^m\sum_{K\in {\cal T}}\int_Kr_K^{n-1}\frac{\bU^n-\bU^{n-1}}{\deltat}\cdot (\bu^{n}-\bU^{n}_h)\dx, \;
 {\cal R}_{2,2}=\deltat\sum_{n=1}^m\sum_{K\in {\cal T}}{\cal R}_{2,2}^{n,K}, \\
&\mbox{and }{\cal R}_{2,2}^{n,K}=\int_K(r^{n-1}-r_K^{n-1})\frac{\bU^n-\bU^{n-1}}{\deltat}\cdot(\bu^{n}-\bU_h^n)\dx;
\end{align*}
therefore,
\[
|{\cal R}_{2,2}^{n}|= |\sum_{K\in {\cal T}}{\cal R}_{2,2}^{n,K}| \le h \, c\|\nabla r\|_{L^\infty(Q_T;\Rm^3)}\|\partial_t\bU\|_{L^\infty(Q_T;\Rm^{3})}\|\bu^n-\bU^n_h\|_{L^6(\Omega;\Rm^3)}.
\]
Consequently, by virtue of formula (\ref{est1}) in Corollary \ref{Corollary1} and estimates (\ref{sob1}), (\ref{L1-5}),
\begin{equation}\label{cR2.2}
|{\cal R}_{2,2}|\le h \, c(M_0, E_0, \|(\nabla r, \bU,\partial_t\bU,\nabla\bU)\|_{L^\infty(Q_T;\Rm^{18})}).
\end{equation}

{
\vspace{2mm}
{\bf Step 2b:} {\it Term ${\cal T}_{2,2}$.} We decompose the term ${\cal T}_{2,2}$ as
\begin{align*}
&{\cal T}_{2,2}={\cal T}_{2,3}+ {\cal R}_{2,3}, \\
&\mbox{with } {\cal T}_{2,3}=\deltat\sum_{n=1}^m\sum_{K\in {\cal T}}\int_Kr_K^{n-1}\frac{\bU^n_{h,K}-\bU^{n-1}_{h,K}}{\deltat}\cdot (\bu^{n}-\bU^{n}_h)\dx, \;
 {\cal R}_{2,3}=\deltat\sum_{n=1}^m\sum_{K\in {\cal T}}{\cal R}_{2,3}^{n,K}, \\
&\mbox{and }{\cal R}_{2,3}^{n,K}=\int_K r_K^{n-1}\Big(\frac{\bU^n-\bU^{n-1}}{\deltat}-\Big[\frac{\bU^n-\bU^{n-1}}{\deltat}\Big]_h\Big)
\cdot(\bu^{n}-\bU_h^n)\dx\\ &
+
\int_K r_K^{n-1}\Big(\Big[\frac{\bU^n-\bU^{n-1}}{\deltat}\Big]_h-\Big[\frac{\bU^n-\bU^{n-1}}{\deltat}\Big]_{h,K}\Big)
\cdot(\bu^{n}-\bU_h^n)\dx { =I_1^K+I_2^K}.
\end{align*}
{ We have
$$
|I_2^K|=\frac 1 \deltat r_K^{n-1}\Big|\int_K \Big(\Big[\int_{t_{n-1}}^{t_n}\partial_t\bU(z){\rm d}z\Big]_h- \Big[\int_{t_{n-1}}^{t_n}\partial_t\bU(z){\rm d}z\Big]_{h,K}\Big)
\cdot(\bu^{n}-\bU_h^n)\dx\Big|
$$
$$
\le \frac h \deltat r_K^{n-1}\int_{t_{n-1}}^{t_n}\Big\|\Grad\Big[\partial_t\bU(z)\Big]_h\Big\|_{L^{6/5}(K;\Rm^3)}
\|\bu^n-\bU^n_h\|_{L^6(K;\Rm^3)}
$$
where we have used the Fubini theorem, H\"older's inequality and (\ref{L2-1}), (\ref{L1-3})$_{s=1}$.
Further, employing the Sobolev inequality on Crouzeix-Raviart space (\ref{sob1}), H\"older's and Young's inequalities and estimate (\ref{L1-3})$_{s=1}$, we get
$$
\sum_{K\in {\cal T}}|I_2^K|\le \frac h \deltat r_K^{n-1} \|\bu^n-\bU^n_h\|_{L^6(\Omega;\Rm^3)}\int_{t_{n-1}}^{t_n}\Big\|\Grad\partial_t\bU(z)\Big\|_{L^{6/5}(\Omega;\Rm^3)}
{\rm d}z
$$
$$
\le
c r_K^{n-1} \Big(h|\bu^n-\bU^n_h|^2_{V^2_h(\Omega;\Rm^3)}+\frac h \deltat \int_{t_{n-1}}^{t_n}
\Big\|\Grad\partial_t\bU(z)\Big\|^2_{L^{6/5}(\Omega;\Rm^3)}\Big)
$$
We reserve the similar treatment to the term $I_1^K$. Resuming these calculations we get by using Corollary  \ref{Corollary1}
%
%
\begin{equation}\label{cR2.3}
|{\cal R}_{2,3}|\le h \, c(M_0, E_0, \|(r, \bU,\nabla\bU, \partial_t\bU)\|_{L^\infty(Q_T;\Rm^{16})},\|\partial_t\nabla\bU\|_{L^2(0,T;L^{6/5}(\Omega;\Rm^{9}))}).
\end{equation}
}
}
{ \bf Step 2c:} {\it Term ${\cal T}_{2,3}$.}
{ We rewrite this term in the form
\begin{equation}\label{cT2}\begin{aligned}
& {\cal T}_{2,3}={\cal T}_{2,4}+ {\cal R}_{2,4},\; {\cal R}_{2,4}=\deltat\sum_{n=1}^m\sum_{K\in {\cal T}}{\cal R}_{2,4}^{n,K}, \\
&\mbox{with } {\cal T}_{2,4}=\deltat\sum_{n=1}^m\sum_{K\in {\cal T}} \int_Kr_K^{n-1} \frac {\bU^n_{h,K}-\bU^{n-1}_{h,K}} {\deltat}\cdot (\bu^{n}_K-\bU^{n}_{h,K})\dx,
\\ &\mbox{and }{\cal R}_{2,4}^{n,K}=\int_Kr_K^{n-1}\frac{\bU_{h,K}^n-\bU_{h,K}^{n-1}}{\deltat}\cdot\Big((\bu^{n}-\bu^n_K) -(\bU_h^n-\bU_{h,K}^n)\Big)\dx.
\end{aligned}
\end{equation}
{
First we write, as in (\ref{R2.1}),
$$
\Big|\frac{[{\bU}^{n}-{\bU}^{n-1}]_{h,K}}{\deltat}\Big|=
\Big|\frac 1{|K|}\int_K\Big[\frac 1\deltat \Big[\int_{t_{n-1}}^{t_n} \partial_t\vc U(z,x) {\rm d} z\Big]_h\Big]{\rm d}x\Big|
$$
$$
=\Big|\frac 1{|K|}\int_K\Big[\frac 1\deltat \int_{t_{n-1}}^{t_n} [\partial_t\vc U(z) \Big]_h(x){\rm d} z\Big]{\rm d}x\Big|\le \|[\partial_t\vc U ]_h\|_{L^\infty(0,T;L^\infty(\Omega;\Rm^3))}\le \|\partial_t\vc U \Big\|_{L^\infty(0,T;L^\infty(\Omega;\Rm^3))},
$$
Next we evaluate $\bu^n-\bu^n_K$ employing (\ref{L2-1})$_{p=2}$, and $\bU_h^n-\bU_{h,K}^n$ by using
(\ref{L2-1})$_{p=\infty}$, (\ref{L1-3})$_{s=1}$. Finally we employ  the H\"older inequality to get
\begin{equation}\label{cR2.4}
|{\cal R}_{2,4}|\le h \; c(M_0,E_0,\overline r, \|(\partial_t\bU, \bU, \nabla\bU)\|_{L^\infty(Q_T;\Rm^{15})}, \|\partial_t\nabla\bU )\|_{L^2(0,T;L^{6/5}(\Omega;\Rm^{9}))}).
\end{equation}
}
}

\vspace{2mm}
\textbf{Step 3:} \textit{ Term  ${\cal T}_3$.}
Let us first decompose ${\cal T}_3$ as
\begin{align*}
& {\cal T}_3={\cal T}_{3,1} + {\cal R}_{3,1}, \\
&\mbox{with } {\cal T}_{3,1}=\deltat\sum_{n=1}^m\sum_{K\in{\cal T}}\int_K r_K^n\bU_{h,K}^n\cdot\nabla\bU^n\cdot (\bu_K^n-\bU_{h,K}^n)\dx,
 \quad {\cal R}_{3,1}=\deltat\sum_{n=1}^m \sum_{K\in{\cal T}}{\cal R}_{3,1}^{n,K},
\\ &\mbox{and }{\cal R}_{3,1}^{n, K}=\int_K(r^n-r_K^n)\bU^n\cdot\nabla\bU^n\cdot(\bu^n-\bU^n_h)\dx
+ \int_K r_K^n(\bU^n-\bU^n_h)\cdot\nabla\bU^n\cdot(\bu^n-\bU^n_h)\dx \\
& \phantom{\mbox{and }{\cal R}_{3,1}^{n, K}=} +\int_K r_K^n(\bU_h^n-\bU^n_{h,K})\cdot\nabla\bU^n\cdot(\bu^n-\bU^n_h)\dx +\int_K r_K^n\bU^n_{h,K}\cdot\nabla\bU^n\cdot\Big(\bu^n-\bu^n_K-(\bU^n_h-\bU^n_{h,K})\Big)\dx.
\end{align*}
We find that
\begin{align*}
&|{\cal R}_{3,1}^{n,K}|  \le  h\,\Big[ |K|^{1/2} (\|\bu^n\|_{L^2(K;\Rm^3)}+ \|\bU_h^n\|_{L^2(K;\Rm^3)})+|K|^{ 1/2} (\|\nabla\bu^n\|_{L^2(K;\Rm^3)}+ \|\nabla\bU_h^n\|_{L^2(K;\Rm^3)})\Big]
\\ & \qquad\qquad\qquad\qquad\qquad \times
\Big(\|r\|_{L^\infty(Q_T)}+\|\nabla r\|_{L^\infty(Q_T;\Rm^3)}\Big)\,\Big(\|\bU\|_{L^\infty(Q_T;\Rm^3)} +\|\nabla\bU\|_{L^\infty(Q_T;\Rm^{9})}\Big)^2,
\end{align*}
where we have used several times H\"older's  inequality and the standard first order Taylor formula ({ to evaluate
$r^n-r^n_K$)}, along with the estimates
(\ref{L1-2}) (to evaluate $\bU^n-\bU^n_h$), (\ref{L2-1}), (\ref{L1-3})$_{s=1}$ (to evaluate
$\bU^n_h-\bU^n_{h,K}$), (\ref{L2-1}) (to evaluate $\bu^n-\bu^n_K$).

Consequently, using again (\ref{L1-3})$_{s=1}$ (to estimate $\|\nabla\bU_h^n\|_{L^2(K;\Rm^3)}$), the definition
 of $ |\cdot|_{V^2_h(\Omega)}$ norm, the Sobolev inequality (\ref{sob1}) and the energy
bound (\ref{est0}) from  Corollary \ref{Corollary1}, we conclude that
\begin{equation}\label{cR3.1}
|{\cal R}_{3,1}|\le h\; c(M_0, E_0,\overline r, \|(\nabla r, \bU, \nabla\bU)\|_{L^\infty(Q_T;\Rm^{15})}).
\end{equation}
%
%

Now we shall deal wit term ${\cal T}_{3,1}$. Integrating by parts, we get:

\begin{align*}
	\int_K r_K^n\bU_{h,K}^n\cdot\nabla\bU^n\cdot (\bu_K^n-\bU_{h,K}^n)\dx &=\sum_{\sigma\in {\cal E}(K)}{ |\sigma|} r_K^n  [\bU_{h,K}^n\cdot{\vc n}_{\sigma,K}]\bU_\sigma^n\cdot(\bu_K^n-\bU_{h,K}^n)
	\\
	&=\sum_{\sigma\in {\cal E}(K)}|\sigma| r_K^n [\bU_{h,K}^n\cdot{\vc n}_{\sigma,K}](\bU_\sigma^n-\bU^n_{h,K})\cdot(\bu_K^n-\bU_{h,K}^n),
\end{align*}
thanks to the the fact that $\sum_{\sigma\in {\cal E}(K)}\int_\sigma \bU_{h,K}^n\cdot{\vc n}_{\sigma,K}{\rm d} S=0$.

{ Next we write
\begin{align*}
& {\cal T}_{3,1}= {\cal T}_{3,2}+ {\cal R}_{3,2},  \quad {\cal R}_{3,2}=\deltat\sum_{n=1}^m {\cal R}_{3,2}^{n},
\end{align*}
\begin{align}
&
{\cal T}_{3,2}=\deltat\sum_{n=1}^m\sum_{K\in{\cal T}}\sum_{\sigma\in {\cal E}(K)}|\sigma| \hat r_\sigma^{n,{\rm up}}
[\hat \bU_{h,\sigma}^{n,{\rm up}}\cdot{\vc n}_{\sigma,K}](\bU_\sigma^n-\bU^n_{h,K})\cdot(\hat\bu_\sigma^{n,{\rm up}}-\hat\bU_{h,\sigma}^{n,{\rm up}}), \label{cT3}
\end{align}
\begin{align*}
&  \mbox{and }{\cal R}_{3,2}^{n}=\sum_{K\in {\cal T}}\sum_{\sigma\in {\cal E}(K)}|\sigma|(r_K^n-\hat r^{n,{\rm up}}_\sigma) [\bU_{h,K}^n\cdot{\vc n}_{\sigma,K}](\bU_\sigma^n-\bU^n_{h,K})\cdot(\bu_K^n-\bU^n_{h,K})
\\
&+
\sum_{K\in {\cal T}}\sum_{\sigma\in {\cal E}(K)}|\sigma|\hat r^{n,{\rm up}}_\sigma \Big[\Big(\bU_{h,K}^n-\hat \bU_{h,\sigma}^{n,{\rm up}}\Big)\cdot{\vc n}_{\sigma,K}\Big](\bU_\sigma^n-\bU^n_{h,K})\cdot(\bu_K^n-\bU^n_{h,K})\\
&
+\sum_{K\in {\cal T}}\sum_{\sigma\in {\cal E}(K)}|\sigma|\hat r^{n,{\rm up}}_\sigma [\hat \bU_{h,\sigma}^{n,{\rm up}}\cdot{\vc n}_{\sigma,K}](\bU_\sigma^n-\bU^n_{h,K})
\cdot\Big((\bu_K^n -\hat\bu_{h,\sigma}^{n,{\rm up}}) -(\bU^n_{h,K}-\hat\bU_{h,\sigma}^{n,{\rm up}})\Big).
\end{align*}
We may use several times the { Taylor formula (in order to estimate $r_K^n-\hat r_\sigma^{n,{\rm up}}$,
 $\vc U_\sigma^n-{\vc U}_{h,K}^n$, $\vc U_{h,K}^n-\hat{\vc U}_{h,\sigma}^{n,{\rm up}}$)} to get the bound
$$
|{\cal R}_{3,2}^n|\le h\, c \|r\|_{W^{1,\infty}(\Omega)}\Big(1+ \|\bU\|_{W^{1,\infty}(\Omega;\Rm^3)}\Big)^3
\sum_{K\in {\cal T}} h|\sigma||\bu^n_K|
$$
$$
+ c \|r\|_{W^{1,\infty}(\Omega)}\Big(1+ \|\bU\|_{W^{1,\infty}(\Omega;\Rm^3)}\Big)^2\sum_{K\in{\cal T}}
\sum_{\sigma\in {\cal E}(K)} h|\sigma| |\bu_K^n-\bu_\sigma^n|,
$$
where by virtue of H\"older's inequality, (\ref{sob4}), (\ref{sob3}), (\ref{L2+-1}) (\ref{L2+-2}),
$$
\sum_{K\in {\cal T}} h|\sigma||\bu^n_K|\le c\Big(\sum_{\sigma\in {\cal T}} h|\sigma||\bu^n_K|^6\Big)^{1/6}
\le c\Big[\Big(\sum_{K\in {\cal T}}\|\bu^n-\bu_K^n\|^6_{L^6(K;\Rm^3)}\Big)^{1/6}
$$
$$
+\Big(\sum_{K\in {\cal T}}\|\bu^n\|^6_{L^6(K;\Rm^3)}\Big)^{1/6}\Big]
\le c\Big(\sum_ {K\in {\cal T}}\|\nabla\bu_n\|^2_{L^2(K;\Rm^9)}\Big)^{1/2},
$$
$$
\sum_{K\in{\cal T}}
\sum_{\sigma\in {\cal E}(K)} h|\sigma| |\bu_K^n-\bu_\sigma^n|\le c\Big[\Big(\sum_{K\in {\cal T}}\|\bu^n-\bu_K^n\|^2_{L^2(K;\Rm^3)}\Big)^{1/2}
$$
$$
+ \Big(\sum_{K\in {\cal T}}\sum_{\sigma\in {\cal E}(K)}\|\bu^n-\bu_\sigma^n\|^2_{L^2(K;\Rm^3)}\Big)^{1/2}\Big]
\le h\,c\Big(\sum_{K\in {\cal T}} \|\nabla\bu_n\|^2_{L^2(K;\Rm^9)}\Big)^{1/2},
$$
Consequently, we may use (\ref{est0}) to conclude
\begin{equation}\label{cR3.2}
|{\cal R}_{3,2}|\le h\, c\Big(M_0, E_0, \overline r, \|\nabla r, \bU, \nabla\bU\|_{L^{\infty}(Q_T;\Rm^{15})}\Big).
\end{equation}
}

\vspace{2mm}
\textbf{Step 4:}\textit{ Terms  ${\cal T}_4$ and ${\cal T}_5$.} \label{7.6}
We decompose ${\cal T}_4$ as

\begin{equation}\label{cT4}
\begin{aligned}
	& {\cal T}_4= {\cal T}_{4,1}+ {\cal R}_{4,1},
	\\
	& \mbox{with }{\cal T}_{4,1}=\deltat \sum_{n=1}^m \sum_{K\in{\cal T}}\int_K p'(r^n_K)\bu^n\cdot\nabla r^n\dx,
\quad
{\cal R}_{4,1}= \deltat \sum_{n=1}^m \sum_{K\in{\cal T}}\int_K \Big( p'(r^n)-p'(r^n_K)\Big)\bu^n\cdot\nabla r^n\dx;
\end{aligned}
\end{equation}
whence
\begin{equation}\label{cR4.1}
	|{\cal R}_{4,1}|\le h\, c(M_0, E_0,\underline r,\overline r, |p'|_{C^1([\underline r,\overline r])}, \|\nabla r\|_{L^\infty(Q_T;\Rm^3)}).
\end{equation}

Employing integration by parts, we infer

\begin{equation}\label{cT5}
\begin{aligned}
	& {\cal T}_5= {\cal T}_{5,1} + {\cal R}_{5,1},  \mbox{ with }{\cal T}_{5,1}=
-\deltat \sum_{n=1}^m \sum_{K\in {\cal T}}\int_K \nabla p(r^n)\cdot  \bU^n\dx,
	\\ &
{\cal R}_{5,1}=\deltat \sum_{n=1}^m \sum_{K\in {\cal T}}\int_K\nabla p(r^n)\cdot\Big(\bU^n-\bU^n_h\Big)\dx
\end{aligned}
\end{equation}
and
\begin{equation}\label{cR5.1}
|{\cal R}_{5,1}|\le h\, c(\overline r, | p'|_{C([\underline r,\overline r])},\|{ \nabla r},\nabla \bU\|_{L^\infty(Q_T;\Rm^{12})}).
\end{equation}
\label{7.7}
Integrating by parts, we obtain
$$
{\cal T}_{5,1}= \deltat \sum_{n=1}^m \sum_{K\in {\cal T}}\int_K p(r^n) \dv \bU^n\dx;
$$
whence
$$
\begin{aligned}
	& {\cal T}_{5,1}= {\cal T}_{5,2} + {\cal R}_{5,2},  \mbox{ with }{\cal T}_{5,2}=
\deltat \sum_{n=1}^m \sum_{K\in {\cal T}}\int_K p(r_K^n)\dv \bU^n\dx,
	\\ &
{\cal R}_{5,2}=-\deltat \sum_{n=1}^m \sum_{K\in {\cal T}}\int_K(p(r_K^n)-p(r^n))\dv\bU^n\dx
\end{aligned}
$$
and
\begin{equation}\label{cR5.2}
|{\cal R}_{5,2}|\le h\, c(\overline r, |p'|_{[\underline r,\overline r]},\|\nabla r,\nabla\bU\|_{L^\infty(Q_T;\Rm^{12})}).
\end{equation}

Gathering the formulae (\ref{cT1}), (\ref{cT2}), (\ref{cT3}), (\ref{cT4}), (\ref{cT5}) and estimates for the residual terms (\ref{cR1.1}), (\ref{cR2.1}--\ref{cR2.4}), (\ref{cR3.1}--\ref{cR3.2}), (\ref{cR4.1}), (\ref{cR5.1}), (\ref{cR5.2}) concludes the proof of Lemma \ref{strongentropy}.
\end{proof}

\section{{ End}  of the proof of  the error estimate (Theorem \ref{Main})}

{ In this Section we put together the relative energy inequality (\ref{relativeenergy-}) and the identity (\ref{strong1}) derived in the previous section. The final inequality resulting from this manipulation is formulated in the following lemma.}
\begin{lm}\label{Gronwall}
Under assumptions of Theorem \ref{Main} there exists a positive number
\[
  c=c\Big(M_0, E_0, \underline r,\overline r, |p'|_{C^1([\underline r,\overline r])},
  \|(\nabla r, \partial_t r, \partial_t\nabla r, \partial^2_t r,  \bU, \nabla\bU, { \nabla^2\bU}, \partial_t\bU, \partial_t\nabla\bU)\|_{L^\infty(Q_T;\Rm^{65})}\Big)
\]
(depending tacitly also on $T$, $\theta_0$, $\gamma$, ${\rm diam} (\Omega)$, $|\Omega|$),
such that for all $m=1,\ldots,N,$ there holds:
\[{\cal E}(\vr^m,\bu^m|r^m,\bu^m)
{ +\deltat\frac \mu 2\sum_{n=1}^m\sum_{{K}\in {\cal T}}\int_K|\Grad(\vc u^n-\vc U^n_h)|^2{\rm d x}}
\]
\[
\le c\Big[h^A+\sqrt{\deltat} + {\cal E}(\vr^0,\bu^0|r^0,\bU^0)\Big] + c\,\deltat\sum_{n=1}^m {\cal E}(\vr^n,\bu^n|r^n,\bu^n),
\]
where $A$ is defined in (\ref{A1}).
\end{lm}

\begin{proof}
Gathering the formulae (\ref{relativeenergy-}) and (\ref{strong1}), one gets
\begin{equation}\label{relativeenergy-1}
{\cal E}(\vr^m,\bu^m\Big| r^m, U^m)- {\cal E}(\vr^0,\bu^0\Big|r(0),\bU(0)) + \mu\deltat\sum_{n=1}^m\Big|\bu^n-\bU^n_h\Big|^2_{V^2_h(\Omega;\Rm^3)}
\end{equation}
$$
\le{\cal P}_1+ {\cal P}_2 +{\cal P}_3 + {\cal Q}
$$
where
\begin{align*}
&{\cal P}_1=
\deltat\sum_{n=1}^m\sum_{K\in{\cal T}}|K|(\vr_K^{n-1}-r^{n-1}_K)\frac{{\bU}_{h,K}^{n}-{\bU}_{h,K}^{n-1}}{\deltat}\cdot \Big({\bU}_{h,K}^{n}   - \bu_K^{n}\Big),
\\
&{ {\cal P}_2=\deltat\sum_{n=1}^m\sum_{K\in{\cal T}}\sum_{\sigma=K|L\in {\cal E}(K)}|\sigma|\Big(\vr_\sigma^{n,{\rm up}}-\hat r_\sigma^{n,{\rm up}}\Big)\Big({\hat\bU}^{n,{\rm up}}_{h,\sigma}-
{\hat\bu}^{n,{\rm up}}_\sigma\Big)\cdot\Big(\bU^n_\sigma-\bU^n_{h,K}\Big) \hat \bU_{h,\sigma}^{n, {\rm up}}\cdot\bn_{\sigma,K},}
\\
&{\cal P}_3=
\deltat \sum_{n=1}^m\sum_{K\in {\cal T}} \int_K \Big(p(r_K^n)-p(\vr_K^n)\Big)\dv\bU^n\dx
\\
&\qquad \qquad+\deltat\sum_{n=1}^m\sum_{K \in {\cal T}}\Big[\int_K\frac{r_K^n-\vr^n_K}{r^n_K} p'(r^n_K) \bu^n\cdot\nabla r^n\dx
+\int_K \frac{ r^n_K-\vr^n_K} {r_K^n}p'(r_K^n) [\partial_t r]^n\dx\Big],
\\
&{ {\cal Q}=
{\cal R}^m_{h,\deltat} +R^m_{h,\deltat}  +{ G}^m. }
\end{align*}

Now, we estimate conveniently the terms ${\cal P}_1$, ${\cal P}_2$, ${\cal P}_3$ in three steps.

\vspace{2mm}

{\bf Step 1:} {\it Term ${\cal P}_1$.}
{ { We have
$$
\Big|{\bU}_{h,K}^{n}-{\bU}_{h,K}^{n-1}\Big| \le \int_{t_{n-1}}^{t_n}\|[\partial_t\bU(z)]_h\|_{L^\infty(K;\Rm^9)}{\rm d z}
\le c \int_{t_{n-1}}^{t_n}\|\partial_t\bU(z)\|_{L^\infty(K;\Rm^9)}{\rm d z},
$$
where we have used (\ref{ddd}).

According to Lemma \ref{LL1},
$$
|\vr-r|^\gamma 1_{R_+\setminus[\underline r/2,2\overline r]}(\vr)\le c(p) E^p(\vr|r)
$$
with any $p\ge 1$. In particular, 
\begin{equation}\label{aaa}
|\vr-r|^{6/5} 1_{R_+\setminus[\underline r/2,2\overline r]}(\vr)\le c E(\vr|r)
\end{equation}
provided $\gamma\ge 6/5$.

We get by using the H\"older inequality,
$$
\Big|\sum_{K\in{\cal T}}|K|(\vr_K^{n-1}-r^{n-1}_K)\frac{{\bU}_{h,K}^{n}-{\bU}_{h,K}^{n-1}}{\deltat}\cdot \Big({\bU}_{h,K}^{n}   - \bu_K^{n}\Big)\Big|\le c\Big(\|\partial_t\bU \|_{L^\infty(Q_T;\Rm^3)}+\|\partial_t\nabla\bU \|_{L^\infty(Q_T;\Rm^9)}\Big)\times
$$
$$
\Big[\Big(\sum_{K\in{\cal T}}|K||\vr^{n-1}_K-r^{n-1}_K|^2 1_{[\underline r/2,2\overline r]}(\vr_K)\Big)^{1/2}
+
\Big(\sum_{K\in{\cal T}}|K||\vr^{n-1}_K-r^{n-1}_K|^{6/5} 1_{R_+\setminus [\underline r/2,2\overline r]}(\vr_K)\Big)^{5/6}\Big]
\times
$$
$$
\Big(\sum_{K\in{\cal T}}|K| \Big|{\bU}_{h,K}^{n}   - \bu_K^{n}\Big|^6\Big)^{1/6}
\le c(\|(\partial_t\bU,\partial_t\nabla\bU)\|_{L^\infty(Q_T;\Rm^{12})})\Big({\cal E}^{1/2}(\vr^{n-1},\bu^{n-1}|r^{n-1},\bU^{n-1})
$$
$$
+
{\cal E}^{5/6}(\vr^{n-1},\bu^{n-1}|r^{n-1},\bU^{n-1})\Big)\,\Big(\sum_{K\in {\cal T}}\| {\bU}_{h,K}^{n}   - \bu_K^{n}\|_{L^6(K;\Rm^3)}^6\Big)^{1/6},
$$
where we have used (\ref{aaa}) and estimate (\ref{est4}) to obtain the last line.

  Now, by the Minkowski inequality,
  }
$$
\Big(\sum_{K\in {\cal T}}\| {\bU}_{h,K}^{n}   - \bu_K^{n}\|_{L^6(K;\Rm^3)}^6\Big)^{1/6}\le
\Big(\sum_{K\in {\cal T}}\| ({\bU}_{h,K}^{n}   - \bu_K^{n})-(\bU_h^n-\bu^{n})\|_{L^6(K;\Rm^3)}^6\Big)^{1/6}
$$
$$
+
\|\bU_h^n-\bu^{n}\|_{L^6(\Omega;\Rm^3)}\le c \Big|\bu^n-\bU^n_h\Big|_{V^2_h(\Omega;\Rm^3)},
$$
where we have used estimate (\ref{sob4-}) and the Sobolev inequality (\ref{sob1}). Finally, employing
Young's inequality, and estimate (\ref{est4}), we arrive at
\begin{multline}\label{cP1}
|{\cal P}_1|  \le \;  c({ \delta},M_0,E_0,\underline r,\overline r,\|(\bU,\nabla\bU,\partial_t\bU,\partial_t\nabla\bU)\|_{L^\infty(Q_T,\Rm^{15})} )
\\
\times \Big(\deltat {\cal E}(\vr^0,\bu^0|r^0,\bU^0)+ \deltat\sum_{n=1}^m{\cal E}(\vr^n,\bu^n|r^n,\bU^n)\Big)
+ \delta \deltat\sum_{n=1}^m\Big|\bu^n-\bU^n_h\Big|^2_{V^2_h(\Omega;\Rm^3)}
\end{multline}
with any $\delta>0$.
}

{\bf Step 2:} {\it Term ${\cal P}_2$.}
{ We write ${\cal P}_2=\deltat \sum_{n=1}^m{\cal P}_2^n$ where Lemma \ref{LL1} and the H\"older inequality yield, similarly as in the previous step,
\begin{equation*}
\begin{aligned}
|{\cal P}^n_2| & \le c(\underline r,\overline r, \|\nabla\bU\|_{L^\infty(Q_T;\Rm^9)}) \times
\\ &
\sum_{K\in {\cal T}}\sum_{\sigma\in {\cal E}(K)}|\sigma| h
\Big(E^{1/2}(\vr_\sigma^{n,{\rm up}},\hat r_\sigma^{n,{\rm up}})+ E^{2/3}(\vr_\sigma^{n,{\rm up}},\hat r_\sigma^{n,{\rm up}}\Big)\,|\hat\bU^{n,{\rm up}}_{h,\sigma}|\,|{\hat\bU}_{h,\sigma}^{n,{\rm up}}   -\hat\bu_\sigma^{n,{\rm up}}|
\\
& \le c(\underline r,\overline r, \|(\bU,\nabla\bU)\|_{L^\infty(Q_T;\Rm^{12})})\Big[\Big(\sum_{K\in {\cal T}}\sum_{\sigma\in {\cal E}(K)}|\sigma| h \Big(E(\vr_\sigma^{n,{\rm up}}|\hat r_\sigma^{n,{\rm up}})\Big)^{1/2}
\\
&
+
\Big(\sum_{K\in {\cal T}}\sum_{\sigma\in {\cal E}(K)}|\sigma| h
E(\vr_\sigma^{n,{\rm up}}|\hat r_\sigma^{n,{\rm up}})\Big)^{2/3}\Big]
 \times
\Big(\sum_{K\in {\cal T}}\sum_{\sigma\in {\cal E}(K)}|\sigma| h\Big|{\hat\bU}_{h,\sigma}^{n,{\rm up}}   -\hat\bu_\sigma^{n,{\rm up}}\Big|^6\Big)^{1/6}
\end{aligned}
\end{equation*}
provided $\gamma\ge 3/2$.
Next,  we observe that the contribution of the face $\sigma=K|L$ to the  sums
$
\sum_{K\in {\cal T}}$ $\sum_{\sigma\in {\cal E}(K)}$ $|\sigma| h
E(\vr_\sigma^{n,{\rm up}}|\hat r_\sigma^{n,{\rm up}})$
and $ \sum_{K\in {\cal T}}\sum_{\sigma\in {\cal E}(K)}|\sigma| h|{\hat\bU}_{h,\sigma}^{n,{\rm up}}   -\hat\bu_\sigma^{n,{\rm up}}|^6$
is less or equal than
$2|\sigma| h ( E(\vr_K^{n}|\hat r_K^{n}) + E(\vr_L^{n}|\hat r_L^{n})),
$ and
than $2|\sigma| h (|{\bU}_{h,K}^{n}   -\bu_K^{n}|^6+ |{\bU}_{h,L}^{n}   -\bu_L^{n}|^6)$, respectively. Consequently,we get
by the same reasoning
as in the previous step, under assumption $\gamma\ge 3/2$,
\begin{equation}\label{cP2}
|{\cal P}_2|\le c(\delta, M_0, E_0, \underline r,\overline r, \|(\bU,\nabla\bU)\|_{L^\infty(Q_T;\Rm^{12})})\,\deltat \sum_{n=1}^m {\cal E}(\vr^n,\bu^n|
r^n,\bU^n) + \delta\; \deltat \sum_{n=1}^m  |\bu^n-\bU_{h}^n)|_{V^2_h(\Omega;\Rm^3)}^2.
\end{equation}
}
{\bf Step 3:} {\it Term ${\cal P}_3$.}
Since the pair $(r,\bU)$ satisfies continuity equation (\ref{cont2}) in the classical sense, we have for all $n=1,\ldots,N$,
\[
[\partial_t r]^n +\bU^n\cdot\nabla r^n=-r^n\dv\bU^n,
\]
where we  recall that  $[\partial_t r]^n(x)=\partial_t r(t_n,x)$ in accordance with (\ref{notation2-}).
Using this identity we write
\begin{align*}
&{\cal P}^n_3= {\cal P}_{3,1} +{\cal P}_{3,2},\quad {\cal P}_{3,i}=\deltat\sum_{n=1}^m {\cal P}_{3,i}^n,\\
& \mbox{with } {\cal P}_{3,1}^n=
-\sum_{K\in {\cal T}} \int_K \Big(p(\vr_K^n)-p'(r_K^n)(\vr_K^n-r_K^n)-p(r_K^n)\Big)\dv\bU^n\dx \\
& \mbox{and } {\cal P}_{3,2}^n= \sum_{K \in {\cal T}}\Big[\int_K\frac{r_K^n-\vr^n_K}{r^n_K} p'(r^n_K)( \bu^n-\bU^n)\cdot\nabla r^n\dx.
\end{align*}
Now, we apply Lemma \ref{LL1} in combination with assumption (\ref{pressure1}) to deduce
\begin{equation}\label{cP3}
|{\cal P}_{3,1}|\le c\|\dv \bU\|_{L^\infty(Q_T)}\deltat \sum_{n=1}^m{\cal E}(\vr^n,\bu^n|r^n,\bU^n).
\end{equation}
Finally, the same reasoning as in Step 2 leads to the estimate
\begin{equation}\label{cP4}
\begin{aligned}
|{\cal P}_{3,2}|& \le h\;c(M_0, E_0,\underline r, \overline r, |p'|_{C([\underline r,\overline r])} \|(\nabla r,\nabla\bU)\|_{L^\infty(\Omega;\Rm^9)})
\\ & + c(\delta, \|\underline r, \overline r, |p'|_{C([\underline r,\overline r])} \|\nabla r\|_{L^\infty(\Omega;\Rm^3)})\;
\deltat \sum_{n=1}^m{\cal E}(\vr^n,\bu^n|r^n,\bU^n) +\delta \,\deltat \sum_{n=1}^m|\bu^n-\bU^n_h|_{V^2_h(\Omega;\Rm^3)}^2.
 \end{aligned}
\end{equation}
Gathering the formulae (\ref{relativeenergy-1}) and (\ref{cP1})-(\ref{cP4}) with $\delta$ sufficiently small (with respect to $\mu$), we conclude the proof of Lemma \ref{Gronwall}.
\end{proof}

{ Finally, Lemma \ref{Gronwall} in combination with the bound (\ref{est4}) yields
\[{\cal E}(\vr^m,\bu^m|r^m,\bU^m)\le c\Big[h^A+\sqrt{\deltat}+\deltat + {\cal E}(\vr^0,\bu^0|r^0,\bU^0)\Big] + c\,\deltat\sum_{n=1}^{m-1} {\cal E}(\vr^n,\bu^n|r^n,\bU^n);
\]
whence Theorem \ref{Main} is a direct consequence of the standard discrete version of Gronwall's lemma. Theorem \ref{Main} is thus proved.}

\section{Appendix: Fundamental auxiliary lemmas and estimates}\label{3}

In this section we report several results related to the properties of the Sobolev spaces on tetrahedra
and of the Crouzeix-Raviart (C-R) space. { We refer to the book Brezzi, Fortin \cite{BRFO} for the general introduction to the subject.}

{ We start with the  inequalities that can be obtained by rescaling from the standard inequalities on a reference tetrahedron of size equivalent to one.}

\begin{lm}[{ Poincar\'e, Sobolev  and interpolation inequalities on tetrahedra}]\label{Lemma2}
Let $1\le p\le \infty$. Let $\theta_0 > 0$ and ${\cal{T}} $ be a triangulation of $ \Omega $ such that $ \theta \ge \theta_0 $ where $ \theta $ is defined in (\ref{reg}). Then we have:
\begin{description}
\item{\it (1) Poincar\'e type inequalities on tetrahedra}

Let $1\le p\le\infty$. There exists $c=c(\theta_0,p)>0$ such that for all
$K\in {\cal T}$ and for all $v\in W^{1,p}(K)$ we have
\begin{equation}\label{L2-1}
\|v-v_K\|_{L^p(K)}\le c h\|\nabla v\|_{L^p(K)},
\end{equation}
\begin{equation}\label{L2-2}
\forall \sigma\in {\cal E}(K), \;\|v-v_\sigma\|_{L^p(K)}\le c h\|\nabla v\|_{L^p(K)}.
\end{equation}
\item{\it (2) Sobolev type inequalities on tetrahedra}

Let $1\le p<d$. There exists $c=c(\theta_0,p)>0$ such that for all
$K\in {\cal T}$ and for all $v\in W^{1,p}(K)$ we have
\begin{equation}\label{L2-3}
\|v-v_K\|_{L^{p^*}(K)}\le c \|\nabla v\|_{L^p(K)},
\end{equation}
\begin{equation}\label{L2-4}
\forall \sigma\in {\cal E(}K), \;\|v-v_\sigma\|_{L^{p^*}(K)}\le c \|\nabla v\|_{L^p(K)},
\end{equation}
where $p^*=\frac{dp}{d-p}$.
\item{\it (3) Interpolation inequalities on the tetrahedra}

Let $1\le p<d$ and $p\le q\le p^*$. There exists $c=c(\theta_0,p)>0$ such that for all
$K\in{\cal T}$ and $v\in W^{1,p}(K)$ we have
\begin{equation}\label{interpol1}
\|v-v_K\|_{L^{q}(K)}
\le c  h^\beta\|\nabla v\|_{L^p(K;\Rm^d)},
\end{equation}
\begin{equation}\label{interpol2}
\|v-v_\sigma\|_{L^{q}(K)}
\le ch^\beta \|\nabla v\|_{L^p(K;\Rm^d)},
\end{equation}
where $\frac 1q=\frac \beta p+\frac {1-\beta}{p^*}$.
%
\end{description}
\end{lm}
{ Combining estimates (\ref{L2-1}--\ref{interpol2}) with the algebraic inequality}
\begin{equation}\label{dod1*}
	\Big(\sum_{i=1}^L|a_i|^p\Big)^{1/p}\le \Big(\sum_{i=1}^L|a_i|^q\Big)^{1/q}
\end{equation}
for all $(a=(a_1,\ldots,a_L)\in \Rm^L$, $1\le q\le p<\infty$, we obtain the following corollaries.
\begin{cor}[Poincar\'e and Sobolev type inequalities on the Sobolev spaces]\label{cor1}
Under the assumptions of Lemma \ref{Lemma2}, we have:
\begin{description}
\item{\it (1) Poincar\'e type inequalities on the domain $\Omega$}

 Let $1\le p\le \infty$. There exists $c=c(\theta_0,p)>0$ such that for all $v\in W^{1,p}(\Omega)$ we have
\begin{align}\label{L2-5}
	\|v-\hat v\|_{L^p(\Omega)}\equiv\Big(\sum_{K\in {\cal T}}\|v-v_K\|^p_{L^p(K)}\Big)^{1/p}\le c h\|\nabla v\|_{L^p(\Omega;\Rm^d)},\\
	\label{L2-6}
	\Big(\sum_{K\in {\cal T}}\sum_{\sigma\in {\cal E}(K)}\|v-v_\sigma\|^p_{L^p(K)}\Big)^{1/p}\le c h\|\nabla v\|_{L^p(\Omega;\Rm^d)}
\end{align}
where $\hat v$ and  $v_\sigma$ are defined by \eqref{vhat} and \eqref{vtilde}.
\item{\it (2) Sobolev type inequalities on the domain $\Omega$}

Let $1\le p<d$. There exists $c=c(\theta_0,p)>0$ such that for all
$v \in W^{1,p}(\Omega)$ we have
\begin{equation}\label{L2-7}
\|v-\hat v\|_{L^{p^*}(\Omega)}
\le c  \|\nabla v\|_{L^p(\Omega)},
\end{equation}
\begin{equation}\label{L2-8}
\Big(\sum_{K\in {\cal T}}\sum_{\sigma\in {\cal E}(K)}\|v-v_\sigma\|^{p^*}_{L^{p^*}(K)}\Big)^{1/p^*}
\le c \|\nabla v\|_{L^p(\Omega;\Rm^d)}.
\end{equation}
\item{\it (3) Interpolation inequalities on the domain $\Omega$}

Let $1\le p<d$ and $p\le q\le p^*$. There exists $c=c(\theta_0,p)>0$ such that for all
$v\in W^{1,p}(\Omega)$ we have
\begin{equation}\label{L2-9}
\|v-\hat v\|_{L^{q}(\Omega)}
\le c  h^\beta\|\nabla v\|_{L^p(\Omega)},
\end{equation}
\begin{equation}\label{L2-10}
\Big(\sum_{K\in {\cal T}}\sum_{\sigma\in {\cal E}(K)}\|v-v_\sigma\|^{q}_{L^{q}(K)}\Big)^{\frac 1 q}
\le ch^\beta \|\nabla v\|_{L^p(\Omega;\Rm^d)},
\end{equation}
where $\frac 1q=\frac \beta p+\frac {1-\beta}{p^*}$.
\end{description}
\end{cor}

\begin{cor}[Poincar\'e and Sobolev type inequalities on $V_h$]\label{cor2}
Under assumptions of Lemma \ref{Lemma2}, there holds:
\begin{description}
\item {\it (1) Poincar\'e type inequality in $V_h(\Omega)$:}
Let $1\le p<\infty$. There exists $c=c(\theta_0,p)$  such that for all $v \in V_h,$
\begin{equation}\label{sob4-}
\|v-\hat v\|_{L^{p}(\Omega)}
\le c h |v|_{V_h^p(\Omega)},
\end{equation}
\begin{equation}\label{sob5-}
 \Big(\sum_{K\in {\cal T}}\sum_{\sigma\in {\cal E}(K)}\|v-v_\sigma\|_{L^p(K)}^p\Big)^{\frac 1 p}
\le c h | v|_{V_h^p(\Omega)}.
\end{equation}

\item {\it (2) Sobolev type inequality in $V_h(\Omega)$:}
Let $1\le p<d$. There exists $c=c(\theta_0,p)$  such that for all $v \in V_h(\Omega),$
\begin{equation}\label{sob4}
\|v-\hat v\|_{L^{p^*}(\Omega)}
\le c |v|_{V_h^p(\Omega)},
\end{equation}
\begin{equation}\label{sob5}
\Big(\sum_{K\in {\cal T}}\sum_{\sigma\in {\cal E}(K)}\|v-v_\sigma\|_{L^{p^*}(K)}^{p^*}\Big)^{\frac 1 {p^*}}
\le c | v|_{V_h^p(\Omega)}.
\end{equation}
\item{\it (3) Interpolation type inequalities in $V_h(\Omega)$ }

Let $1\le p<d$ and $p\le q\le p^*$. There exists $c=c(\theta_0,p)>0$ such that for all
$v\in V_h(\Omega)$ we have
\begin{equation}\label{L2+-1}
\|v-\hat v\|_{L^{q}(\Omega)}
\le c  h^\beta|v|_{V_h^p(\Omega)},
\end{equation}
\begin{equation}\label{L2+-2}
\Big(\sum_{K\in {\cal T}}\sum_{\sigma\in {\cal E}(K)}\|v-v_\sigma\|_{L^q(K)}^q\Big)^{\frac 1 q}
\le ch^\beta |v|_{V_h^p(\Omega)},
\end{equation}
where $\frac 1q=\frac \beta p+\frac {1-\beta}{p^*}$.
\end{description}
\end{cor}

{ The next fundamental lemma deals with the properties of the projection $v_h$ defined by (\ref{vh})}.

\begin{lm} [Projection on $V_h$]\label{Lemma1}
Let $\theta_0 > 0$ and ${\cal{T}} $ be a triangulation of $ \Omega $ such that $ \theta \ge \theta_0 $ where $ \theta $ is defined in (\ref{reg}).
\begin{description}
\item{\it (1) Approximation estimates on the tetrahedra}

Let $1\le p\le \infty$. There exists $c=c(\theta_0,p)>0$ such that
{
\begin{equation}\label{ddd}
\forall v\in W^{1,p}_0(\Omega)\cap L^\infty(\Omega), \forall K \in {\cal{T}},\;\|v_h\|_{L^\infty(K)}\le c \|v\|_{L^\infty(K)},
\end{equation}
}
\begin{equation}\label{L1-2}
\forall v \in W^{1,p}_0\cap W^{s,p}(\Omega), \forall K \in {\cal{T}},~ || v-v_h||_{L^p(K)} \le c h^{s}\|\nabla^s v\|_{L^p(K;\Rm^{d^s})},
\end{equation}
\begin{equation}\label{L1-3}
|| \nabla( v- v_h) ||_{L^p(K;\Rm^d)} \le c h^{s-1} \|\nabla^s v\|_{L^p(K;\Rm^{d^s})}, s=1,2.
\end{equation}

\item{\it (2) Preservation of divergence}

\begin{equation}\label{L1-1}
\forall \bv \in W^{1,2}_{0}(\Omega,\R^d), \forall q \in L_h(\Omega), \sum_{K\in {\cal T}}\int_K q ~\dv\bv_h \dx = \int_{\Omega} q ~\dv \bv \dx
\end{equation}

\item{(3) Approximation estimates of the Poincar\'e type on the whole domain}

Let $1\le p<\infty$. There exists $c=c(\theta_0,p)>0$ such that for all $v\in W^{1,p}_0(\Omega)$,
\begin{equation}\label{L1-4}
\|v-v_h\|_{L^p(\Omega)}\le c h \|\nabla v\|_{L^p(\Omega;\Rm^d)}.
\end{equation}

\item{ (4) Approximation estimates of the Sobolev type on the whole domain}

Let $1\le p<d$. There exists $c=c(\theta_0,p)>0$ such that for all $v\in W^{1,p}_0(\Omega)$,
\begin{equation}\label{L1-5}
\|v-v_h\|_{L^{p^*}(\Omega)}\le c \|\nabla v\|_{L^p(\Omega;\Rm^d)}.
\end{equation}
\end{description}
\end{lm}

{ Statement {\it (2)} of Lemma \ref{Lemma1} is proved in
\cite{cro-73-con}, where one can find also the proof of item {\it (1)} for $p=2$. We present here the proof of statements {\it (1), (3), (4)} for arbitrary $p$ for the reader's convenience, since a straightforward reference is not available.
\\ \\
\begin{proof}

{\textbf{ Step 1:}} We start with some generalities. First we complete the Crouzeix-Raviart basis (\ref{CRB}) by functions $\phi_\sigma$ indexed also with $\sigma\in {\cal E}_{\rm ext}$ saying
$$
\frac 1{|\sigma'|}\int_{\sigma'}\phi_\sigma {\rm d} S=\delta_{\sigma,\sigma'},\; \;(\sigma,\sigma')\in {\cal E}^2
$$
and observe that
\begin{equation}\label{+1}
\sum_{\sigma\in {\cal E}(K)}\phi_\sigma(x)= 1\;\mbox{ for any $x\in K$.}
\end{equation}
A scaling argument yields
\begin{equation}\label{+2}
\|\phi_\sigma\|_{L^\infty(\Omega)}\le c(\theta_0),\; h\|\nabla\phi_\sigma\|_{L^\infty(\Omega;\Rm^d)}\le c(\theta_0).
\end{equation}
Second, we define the projection $v\to v_h$ for any $v\in W^{1,p}(\Omega)$ by saying
$$
v_h=\sum_{\sigma\in {\cal E}} v_\sigma\phi_\sigma.
$$
We notice that if $v\in W^{1,p}_0(\Omega)$ then $v_h$ coincides with (\ref{CRP}). Moreover,
\begin{equation}\label{+3}
v_h=v\;\mbox{for any affine function $v$.}
\end{equation}
Third, due to the density argument, it is enough to show the remaining statements {\it (1), (3), (4)} for
$v\in W^{1,p}_0(\Omega)\cap W^{s,\infty}(\Omega)$, $s=1,2$, according to the case.
\\ \\
{\textbf{Step 2:}}  { We realize that ${\rm supp }\phi_\sigma=K\cup L$ and derive (\ref{ddd}) directly by employing representation (\ref{vh}), definition of $v_\sigma$ and estimate (\ref{+2}).}

We denote by $x_K=\frac 1{|K|}\int_K x{\rm d}x$ the center of gravity of the tetrahedron $K$. We calculate by using (\ref{+3}) and the first order Taylor formula
$$
v(x)-v_h(x)= v(x)-v(x_K) - [v-v(x_K)]_h(x)
$$
$$
=(x-x_K)\cdot\int_0^1\nabla v(x_K+t(x-x_K)){\rm d} t
- \sum_{\sigma\in {\cal E}(K)}\phi_\sigma(x) \frac 1{|\sigma|}\int_\sigma\Big[(x-x_K)\cdot\int_0^1\nabla v(x_K+t(x-x_K)){\rm d} t\Big]{\rm d} S,
$$
where $x\in K$.
This formula yields immediately the upper bound  stated in (\ref{L1-2})$_{s=1}$ if $p=\infty$. If $1\le p<\infty$
we calculate the upper bound of the $L^p$-norm of each term at the right-hand side separately by using (\ref{+2}), Fubini's theorem, H\"older's inequality and the change of variables $y= x_K+t(x-x_K)$ together with the convexity of $K$.

The same reasoning can be applied to prove (\ref{L1-2})$_{s=2}$. Indeed, we observe
that
$$
v(x)-v_h(x)= v(x)-(x-x_K)\cdot\nabla v(x_K)-v(x_K) - [v -(x-x_K)\cdot\nabla v(x_K)-v(x_K)]_h(x)
$$
by virtue of (\ref{+3}). Now we apply to the right hand side of the last expression the second order Taylor formula in the integral form, and proceed exactly as described before.

Finally, one applies the same straightforward argumentation to get (\ref{L1-3}). This completes the proof of statement {\it (1)}.
\\ \\
{\textbf{Step 3:}} Statement {\it (3)} follows easily from (\ref{L1-2})$_{s=1}$ and the algebraic inequality (\ref{dod1*}).
\\ \\
{\textbf{ Step 4:}} We use (\ref{+1}) and (\ref{+3}) to write
$$
v(x)-v_h(x)= \sum_{\sigma\in {\cal E}(K)} (v(x)-v_\sigma)\phi_\sigma(x),\;\; x\in K;
$$
whence
$$
\|v-v_h\|_{L^{p^*}(K)}\le c\|\nabla v\|_{L^p(K;\Rm^d)}
$$
where we have used the Sobolev inequality (\ref{L2-4}) on the tetrahedron $K\in {\cal T}$ and the $L^\infty$-bound (\ref{+2}). We conclude the proof of statement {\it (4)} by using the relation (\ref{dod1*}). The proof of Lemma \ref{Lemma1} is complete.
\end{proof}
}

The following corollary is a direct consequence of (\ref{L1-3}).

\begin{cor}[Continuity of the projection onto $V_h$]\label{cor3}
Under assumptions of Lemma \ref{Lemma1}, there exists $c=c(\theta_0,p)>0$ such that
\begin{equation}\label{L1-6}
\forall v \in W^{1,p}_0(\Omega),\;|v_h|_{V_h^p(\Omega)}\le c \|\nabla v\|_{L^p(\Omega;\Rm^d)},
\end{equation}
where $1\le p<\infty$.
\end{cor}

Although the  non conforming finite element space $V_h$ is not a subspace of any Sobolev space, its elements enjoy the Sobolev type inequalities.
This important fact is formulated in the next lemma.

\begin{lm}[Sobolev inequality  on $V_h$]\label{Lemma2+}
Let $\Omega$ be a bounded domain  of $\R^d$.
Let ${\cal{T}}$ be a triangulation of the domain $\Omega$ in simplices such that $\theta \ge \theta_0 >0$ where $\theta$ is defined in $(\ref{reg})$.
Then we have:
\begin{description}
\item{ (1) Sobolev  inequality in $V_h(\Omega)$ (case $1\le p<d$):}

There exists $c=c(\theta_0,p)$  such that for all $v \in V_h(\Omega),$
\begin{equation}\label{sob1}
||v||_{L^{p^*}(\Omega)} \le c |v|_{V_h^{p}(\Omega)}.
\end{equation}
\item{\it (2) Sobolev inequality in $V_h(\Omega)$, case $p\ge d$}

Let $1\le q<\infty$. There here exits $c=c(\theta_0,p,q)>0$  such that forall $v \in V_h(\Omega)$,
\begin{equation}\label{sob3}
\|v\|_{L^q(\Omega)} \le c|v|_{V_h^p(\Omega)}
\end{equation}

\end{description}
\end{lm}
\begin{proof}

\textbf{Step 1}
Let $1 \le r \le \alpha < \infty$. Let $ u \in V_h$. We call $v$ the element of $V_h$ such that $ v_\sigma = |u_\sigma|^\alpha $. Then there exists $C$ only depending on $d,r,\alpha$ such that
\begin{equation}\label{leman3}
 ||u||_{L^r(\Omega)}^\alpha \le || u ||_{L^{\frac{r}{\alpha}}(\Omega)}.
\end{equation}

To prove ($\ref{leman3}$) we remark that, using a change of variable, it is enough to show to prove the existence of $C$  for only the unit symplex $ \hat{K}$. Let $u \in \mathbb{P}_1(\hat{K}) $ and we call $v$ the element of $\mathbb{P}_1(\hat{K}) $ such that $ v_\sigma = |u_\sigma|^\alpha $. Let $T(u) = ||u||_{L^r(\hat{K})} $ and $ S(u) = || u ||_{L^{\frac{r}{\alpha}}(\hat{K})}
^{\frac{1}{\alpha}}$. These  two functions are continuous, homogeneous  of degree $1$ and non zero if $ u \ne 0$.  Since $\mathbb{P}_1(\hat{K}) $  is a finite dimensional space, we can choose a norm on $\mathbb{P}_1(\hat{K}) $ and take $C = (\frac{M}{m})^\alpha $ where $ M = \max \{ T(u), ||u||_{\mathbb{P}_1(\hat{K})}=1 \} $ and  $ m = \min \{ T(u), ||u||_{\mathbb{P}_1(\hat{K})}=1 \} $.
\vspace{2em}

\textbf{ Step 2: Proof for $p=1.$} \\
We set $u=0$ outside $\Omega$. For $\sigma \in {\cal{E}_{\intt}}, \sigma=K|L$, we set $ |[u(x)]|=|u_K(x)-u_L(x)|$ for $x \in \sigma$. For $ \sigma \in {\cal{E}}_{\extt}\cap {\cal{E}}(K) $, we set $|[u(x)]| =|u_K(x)|$ for $x \in \sigma$. We first remark that there exists $C_{1,1}$ and $C_{1,2}$ only depending on $d$ such that
$$ ||u||_{L^{\frac{d}{d-1}}(\Omega)} \le C_{1,1}||u||_{BV(\R^d)} \le C_{1,2} || \nabla_h u||_{L^1(\Omega)} + C_{1,2} \sum_{\sigma \in {\cal{E}}} \int_\sigma |[u]| {\rm d}S. $$
We now prove that there exits $C_{1,3}$ only depending on $d$ and $\theta_0$ such that
$$ \sum_{\sigma \in {\cal{E}}} \int_\sigma |[u]| {\rm d}S \le C_{1,3} || \nabla_h u||_{L^1(\Omega)}. $$
Let $K \in \T $ and $ \sigma \in {\cal{E}}(K) $. Let $x_\sigma $ be the center of mass of $\sigma$. We have, with$ u_K =u $ in $K$,
$$ u_K(x) -u(x_\sigma) = \int_0^1 \nabla u_K \cdot (x-x_\sigma) \dx.$$
Then if $ \sigma=K|L$ we have
$$ |u_K(x)-u_L(x)| \le h_\sigma \Big( |\nabla u_K| + |\nabla u_L| \Big). $$
Integrating this inequality on $\sigma$ gives
$$ \int_\sigma |[u]| {\rm d}S \le |\sigma|h_\sigma  \Big( |\nabla u_K| + |\nabla u_L| \Big) \le \frac{2}{\theta_0^d} \Big( || \nabla u||_{L^1(K)} +|| \nabla u||_{L^1(L)} \Big).$$
Similarly for $ \sigma \in \E_{\extt}\cap \E(K) $ we have
$$ \int_\sigma |[u]| {\rm d}S \le \frac{2}{\theta_0^d}  || \nabla u||_{L^1(K)} $$
Then there exists $C_{1,3}=C(d,\theta_0) $ such that
$$ \stie \int_\sigma |[u]| {\rm d}S \le C_{1,3} || \nabla_h u ||_{L^1(\Omega)}. $$
and then,
$$ ||u||_{L^{1^*}(\Omega)} \le c(d,\theta_0) || \nabla_h u||_{L^1(\Omega)}. $$
\textbf{ Step 3: Proof for $1<p<d.$} \\
Let $1<p<d $ and $ p^* = \frac{pd}{d-p} $ and let $ u \in V_h$. We set $ u=0$ outside $\Omega$. Let $ \alpha = \frac{p(d-1)}{d-p} $, so that $\alpha > 1$ and $\alpha 1^*=p^* $. We call $v$ the element of $V_h$ such that $ v_\sigma = |u_\sigma|^\alpha $ for $ \sigma \in \E $.  One has $v \ne |u|^\alpha $ but there exits $C_{2,1} $ only depending on $d$ and $p$ (see lemma $\ref{leman3}$) such that
$$ || u||_{L^{p^*}(\Omega)}^\alpha \le C_{2,1} ||v||_{L^{1^*}(\Omega)} \le c(d,p,\theta_0) ||\nabla_h v||_{L^{1}(\Omega)}.  $$
Moreover using a scalling argument  we obtain
$$ || \nabla_h v ||_{L^1(K)} \le c(d,p,\theta_0) \stiek |u_\sigma|^{\alpha-1}| \nabla u_K | |K|. $$
Then, using H\"older Inequality, we have, with $q=\frac{p}{p-1}$ (so that $q( \alpha-1) =p^*$),
$$  || \nabla_h v ||_{L^1(K)} \le c(d,p,\theta_0) || \nabla u ||_{L^p(K)} || u||_{L^{p^*}(K)}^{\frac{p^*}{q}}.$$
Summing on $ K \in \T$ we obtain
$$||u||_{L^{p^*}(\Omega)} \le C_2 || \nabla_h u||_{L^p(\Omega)}.$$
\textbf{Step 4: Proof for $p\ge d.$} \\
Let $ 1 \le q < \infty $. There exists $r=r(d,q)$ such that  $r <d$ and $r^* \ge q$. We have $$ ||u||_{L^{r^*}(\Omega)} \le c(r,d,q,\theta_0) || \nabla_h u ||_{L^r(\Omega)}. $$
Moreover $$ ||u||_{L^{q}(\Omega)} \le |\Omega|^{\frac{1}{q}-\frac{1}{r^*}} ||u||_{L^{r^*}(\Omega)} \le c(d,q,\theta_0) |\Omega|^{\frac{1}{q}-\frac{1}{r^*}} || \nabla_h u ||_{L^r(\Omega)} $$ and
$$ || \nabla_h u ||_{L^r(\Omega)} \le |\Omega|^{\frac{1}{r}-\frac{1}{p}} || \nabla_h u ||_{L^p(\Omega)}.$$
Finally
$$ ||u||_{L^q(\Omega)} \le c(\Omega,d,p,q,\theta_0) || \nabla_h u ||_{L^p(\Omega)}.$$
\end{proof}

A Combination of Lemma \ref{Lemma2+} with estimates (\ref{sob4-}), (\ref{sob4}) and  the H\"older inequality
yields the following corollary.

\begin{cor}[Estimates of the norms of mean values]\label{cor4}
We have under the assumptions of Lemma \ref{Lemma2+}:
\begin{description}
\item {\it (1) Poincar\'e type inequality involving mean values on tetrahedra}

There exists $c=c(\theta_0,p)$ such that for all $v\in V_h$,
\begin{equation}\label{norms2}
\|\hat v\|_{L^p(\Omega)}\equiv\Big(\sum_{K\in {\cal T}}|K||v_K|^p\Big)^{1/p}\le c( ||v||_{L^p(\Omega)}+h|v|_{ V_h^p(\Omega)}).
\end{equation}
\item {\it (2) Sobolev type inequality involving mean values on tetrahedrons}

Let $1\le p <d$, there exists $c=c(\theta_0,p)$ such that for all $v\in V_h$,
\begin{equation}\label{norms3}
\|\hat v\|_{L^{p^*}(\Omega)}\equiv\Big(\sum_{K\in {\cal T}}|K||v_K|^{p^*}\Big)^{1/{p^*}}\le c( ||v||_{L^{p^*}(\Omega)}+|v|_{V_h^p}).
\end{equation}
\end{description}
\end{cor}

Note that the Last but not least, we recall a result on equivalence of norms in the space $V_h(\Omega)$ which is a consequence of a discrete Poincar\'e inequality on the broken Sobolev space $V_h$  \cite[proposition 4.13]{tem-84-nav}.

\begin{lm}[Discrete and continuous norms in $V_h$]\label{Lemma4}
Let $1\le p<\infty$. Let $ \theta_0 > 0$ and $ {\cal{T}} $ be a  triangulation of $ \Omega $ such that $ \theta \ge \theta_0 $ where $ \theta $ is defined in $ (\ref{reg}) $. Then
the norms
\begin{equation}\label{norms1}
\Big(\sti |\sigma| h |v_\sigma|^p\Big)^{1/p}\quad\mbox{and}\quad ||v||^p_{L^p(\Omega)}
\end{equation}
are equivalent on $V_h(\Omega)$ uniformly with respect to $h>0$.
\end{lm}

The last lemma in this overview deals with the estimates of jumps over faces.
The reader can consult   \cite[Lemma 3.32]{ern-04-the} or \cite[Lemma 2.2]{GHL2009iso} for its proof.
\begin{lm}[Jumps over faces in the Crouzeix-Raviart space]\label{Lemma6}
Let $ \theta_0 > 0$ and $ {\cal{T}} $ be a  triangulation of $ \Omega $ such that $ \theta \ge \theta_0 $ where $ \theta $ is defined in $ (\ref{reg}) $. Then there exists  $ c=c(\theta_0)>0 $ such that for all $ v \in V_h(\Omega)$,
\begin{equation}\label{tbound}
\sum_{\sigma \in {\cal{E}}} \frac{1}{h} \int_\sigma [v]_{\sigma,\bn_\sigma}^2 \dS \le c |v|_{V_h^2(\Omega)}^2,
\end{equation}
where  $[v]_{\sigma,\bn_\sigma}$ is a jump of $v$  with respect to a normal $\bn_\sigma$ to the face $\sigma$,
\[
\forall x\in\sigma=K|L\in{\cal E}_{\rm int},\quad [v]_{\sigma,\bn_\sigma}(x)=\left\{
\begin{array}{c}
 v|_K(x)-v|_L(x)\;\mbox{if $\bn_\sigma=\bn_{\sigma,K}$}\\
v|_L(x)-v|_K(x)\;\mbox{if $\bn_\sigma=\bn_{\sigma,L}$}
\end{array}\right.
\]
and
\[
\forall x\in\sigma\in{\cal E}_{\rm ext},\quad [v]_{\sigma,\bn_\sigma}(x)=v(x),\;\mbox{with $\bn_\sigma$ an exterior normal to $\partial\Omega$}.
\]
\end{lm}

\end{document}